\documentclass{tran-l}

\usepackage[utf8]{inputenc}
\usepackage{fontenc}
\usepackage{graphicx}
\usepackage{float}
\usepackage{textcomp}
\usepackage{gensymb}
\usepackage[shortlabels]{enumitem}
\usepackage[margin=1.2in]{geometry}
\usepackage{color}
\definecolor{dark-blue}{rgb}{0.15,0.15,0.4}
\definecolor{medium-blue}{rgb}{0,0,0.5}
\usepackage{hyperref} %referenciranje unutar dokumenta
\hypersetup{
	colorlinks=true,
	citecolor=medium-blue,
	filecolor=black,
	linkcolor=dark-blue,
	urlcolor=black
}
\usepackage{multirow} %za tablicu
\usepackage{float}
\numberwithin{equation}{section} %numeriranje jednadžbi po sekcijama
\graphicspath{{figures/}} %lokacija slika
\usepackage{subcaption} %za više slika

\newcommand{\R}{\mathbb{R}}
\newcommand{\N}{\mathbb{N}}
\newcommand{\dt}{\Delta t}
\newcommand{\id}{\mathbf{id}}
\newcommand{\dis}{\displaystyle}

\newcommand{\bvarphi}{\boldsymbol{\varphi}}
\newcommand{\bphi}{\boldsymbol{\phi}}
\newcommand{\bpi}{\boldsymbol{\pi}}
\newcommand{\bP}{\mathbf{P}}

\newcommand{\bgamma}{\boldsymbol{\gamma}}
\newcommand{\brho}{\boldsymbol{\varrho}}

\newcommand{\bff}{\mathbf{f}}
\newcommand{\be}{\boldsymbol{\eta}}
\newcommand{\bpsi}{\boldsymbol{\psi}}
\newcommand{\bv}{\mathbf{v}}

\newcommand{\bet}{\boldsymbol{\tilde{\eta}}}
\newcommand{\et}{\tilde{\eta}}
\newcommand{\zt}{\tilde{z}}
\newcommand{\rt}{\tilde{r}}
\newcommand{\tht}{\tilde{\theta}}
\newcommand{\bg}{\mathbf{g}} 
\newcommand{\bh}{\mathbf{h}}

\newcommand{\bd}{\mathbf{d}}
\newcommand{\bk}{\mathbf{k}}
\newcommand{\bxi}{\boldsymbol{\xi}}

\newcommand{\bw}{\mathbf{w}}
\newcommand{\bz}{\mathbf{z}}
\newcommand{\bzeta}{\boldsymbol{\zeta}}

\newcommand{\bp}{\mathbf{p}}
\newcommand{\bq}{\mathbf{q}}
\newcommand{\bt}{\mathbf{t}}

\newcommand{\bu}{\mathbf{u}}
\newcommand{\bups}{\boldsymbol{\upsilon}}
\newcommand{\bsigma}{\boldsymbol{\sigma}}
\newcommand{\bD}{\mathbf{D}}
\newcommand{\bn}{\mathbf{n}}
\newcommand{\bI}{\mathbf{I}}
\newcommand{\bup}{\mathbf{\tilde{u}}}
\newcommand{\bG}{\mathbf{G}}

\newcommand{\bA}{\boldsymbol{\mathcal{A}}^{n+1,n}}
\newcommand{\bS}{S^{n+1,n}}
\newcommand{\bs}{\mathbf{s}^{n+1,n}}
\newcommand{\bAc}{\boldsymbol{\mathcal{A}}^{\bet}}
\newcommand{\bSc}{S^{\bet}}
\newcommand{\bsc}{\mathbf{s}^{\bet}}
\newcommand{\buh}{\mathbf{\hat{u}}}

\newcommand{\bul}{\mathbf{\overline{u}}}
\newcommand{\bel}{\boldsymbol{\overline{\eta}}}
\newcommand{\bvl}{\mathbf{\overline{v}}}
\newcommand{\bdl}{\mathbf{\overline{d}}}
\newcommand{\bwl}{\mathbf{\overline{w}}}
\newcommand{\bkl}{\mathbf{\overline{k}}}
\newcommand{\bzl}{\mathbf{\overline{z}}}

\newcommand{\bupst}{\boldsymbol{\tilde{\upsilon}}}
\newcommand{\bupsl}{\boldsymbol{\overline{\upsilon}}}

\newtheorem{defn}{Definition}[section]
\newtheorem{theorem}[defn]{Theorem}
\newtheorem{prop}[defn]{Proposition}
\newtheorem{lemma}[defn]{Lemma}
\newtheorem{cor}[defn]{Corollary}
\newtheorem*{remark}{Remark}
\newtheorem*{notation}{Notation}
\newtheorem{assumption}{Assumption}

\begin{document}
	
\title{Analysis of a 3D Nonlinear, Moving Boundary Problem describing
Fluid-Mesh-Shell Interaction}

\author{Sun\v cica \v Cani\' c}
\address{Department of Mathematics, UC Berkeley}
\curraddr{}
\email{canics@berkeley.edu}
\thanks{}

\author{Marija Gali\'c}
\address{Department of Mathematics, Faculty of Science, University of Zagreb}
\email{marijag5@math.hr}
\thanks{}

%	Information for third author
\author{Boris Muha}
\address{Department of Mathematics, Faculty of Science, University of Zagreb}
\email{borism@math.hr}
\thanks{The first author was supported in part by the US National Science Foundation under grants DMS-1853340 and DMS-1613757. The second and third author were supported in part by the Croatian Science Foundation (Hrvatska zaklada za znanost) grant number IP-2018-01-3706 and by the Croatia-USA bilateral grant “Fluid-elastic structure interaction with the Navier slip boundary
condition”.}

%	General info
\subjclass[2000]{Primary 74F10, 35D30; Secondary 74K25, 76D}

\date{\today}

\begin{abstract}
	We consider a nonlinear, moving boundary, fluid-structure interaction problem between 
	a time dependent incompressible, viscous fluid flow, and an elastic structure composed of a cylindrical shell supported by a mesh of elastic rods. 
	The fluid flow is modeled by the time-dependent Navier-Stokes equations in a three-dimensional cylindrical domain, while 
	the lateral wall of the cylinder is modeled by 
 the two-dimensional linearly elastic Koiter shell equations coupled to a one-dimensional system of conservation laws 
 defined on a graph domain, describing a mesh of curved rods.  
 The mesh supported shell allows  displacements in all three spatial directions.
 Two-way coupling based on kinematic and dynamic coupling conditions is assumed between the fluid and composite structure,
 and between the mesh of curved rods and Koiter shell. 
 Problems of this type arise in many applications, including blood flow through arteries treated with vascular prostheses called stents.
We prove the existence of a weak solution to this nonlinear, moving boundary problem by using the time discretization via Lie operator splitting method
combined with an Arbitrary Lagrangian-Eulerian approach, and a non-trivial extension of the 
	Aubin-Lions-Simon compactness result to problems on moving domains.
	 \end{abstract}

\maketitle

%\tableofcontents

\section{Introduction}

We study a fluid-structure interaction (FSI) problem between the flow of a viscous, incompressible, Newtonian fluid occupying a three-dimensional cylindrical domain 
with elastic lateral walls composed of a cylindrical shell supported by a mesh of curved rods.
See Fig.~\ref{sketch}.
\begin{figure}[t]
	\centering
	\includegraphics[width=0.7\linewidth]{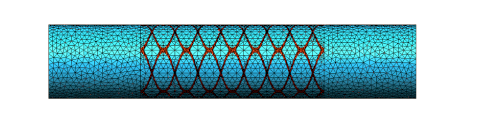}
	\caption{A sketch of a mesh-supported shell}
	\label{sketch}
\end{figure}

The elastic composite structure is modeled by 
the two-dimensional linearly elastic Koiter shell equations \cite{KoiterJust} coupled to a system of linear one-dimensional hyperbolic balance laws defined on a graph domain, 
modeling a mesh of elastic curved rods \cite{TambacaKosorCanic10}.
The system of 1D hyperbolic balance laws captures infinitesimal rotation and displacement of curved rods. 
The curved rods are mutually coupled at graphs' vertices through continuity of displacements and infinitesimal rotation, and via the balance of contact forces and moments,
defining a 1D hyperbolic net model. 
The 1D hyperbolic net model and Koiter shell are coupled via the no-slip condition and via the balance of contact forces.
The resulting mesh-supported shell allows displacements in 
all three spatial directions,
which presents one of the main mathematical difficulties in the existence proof. 
To the best of our knowledge, this is the first existence proof in which a composite structure consisting of a mesh-supported shell is studied,
and in which thin structure displacements
 in all three spatial directions are considered.
Each novelty gives rise to significant difficulties in the existence proof.

The fluid flow is modeled by the 
three-dimensional Navier-Stokes equations. No assumption on the symmetry of flow is used.
The composite structure is coupled to the fluid equations via the no-slip condition and via the balance of contact forces, which are evaluated along the current location 
of the fluid-structure interface. 
This two-way coupling gives rise to a strong geometric nonlinearity in the problem, since the location of the fluid domain is not known {\sl a priori} and is one of the 
unknowns in the problem. 

The coupled fluid-structure interaction problem is driven by the time dependent inlet and outlet dynamic pressure data.

We prove the existence of a weak solution to this nonlinear, moving boundary problem. 
The existence proof is constructive and  it is based on semi-discretizing the problem in time
using a Lie operator splitting strategy, and on using an Arbitrary Lagrangian-Eulerian (ALE) approach
to deal with the motion of the fluid domain. 

The coupled problem is discretized in time and at the same time split into a fluid and a structure subproblem using the so called Lie operator 
splitting strategy. 
Solutions of the fluid and structure subproblems communicate via the initial data, similar to the famous Lie-Trotter formula. 

The time discretization via Lie operator splitting defines a sequence of approximate solutions,
which we want to show converges to a weak solution, as the time discretization step converges to zero.
A key component in obtaining convergence is to split the coupled problem into subproblems  
in such a way that the approximate solutions satisfy 
a uniform energy estimate, which mimics the energy of the continuous problem.
Because of the fluid domain motion, 
obtaining a uniform energy estimate  
for the gradients of the approximate fluid velocities requires additional work, since the classical 
Korn's estimate  depends on the fluid domain motion.
Once uniform energy estimates are derived, 
the existence of weakly- and weakly*-convergent subsequences follows from the uniform energy estimates.

To show that the limits of these subsequences satisfy the weak formulation of the coupled, nonlinear problem, 
a compactness argument needs to be employed.
However, because of the geometric nonlinearity due to the fluid domain motion,
classical compactness arguments cannot be applied,
which presents another major difficulty in studying this class of problems. 
Thanks to 
a recent nontrivial generalization of the Aubin-Lions-Simon compactness lemma to problems on moving domains \cite{BorSunCompact},
we show that subsequences of approximate solutions converge strongly in the corresponding topologies, as the time discretization step goes to zero.

Using the strong convergence of approximate subsequences we want to pass to the limit in weak formulations of approximate problems
and show that the limits satisfy the weak formulation of the continuous, coupled problem.
However, again due to the motion of fluid domains, the classical approach
cannot be directly applied since the test functions of approximate problems themselves depend on fluid domains and on the time discretization.
This is why ``appropriate'' test functions need to be constructed from the approximate test functions, for which one can show that they
converge uniformly to the test functions of the continuous problem.

The strong convergence of approximate solutions obtained using the generalized compactness arguments from \cite{BorSunCompact}, 
combined with uniform convergence of the ``appropriate'' test function, allows passing to the limit in 
approximate weak formulations, and showing that the limits of approximate solutions satisfy the weak formulation of the coupled, continuous problem.

The result and the techniques developed here are robust, and they provide a significant step forward in the analysis of FSI problems.
For the first time a 1D-2D-3D nonlinearly coupled FSI problem involving a 1D hyperbolic net is analyzed, 
	and the existence of a weak solution proved.
	In contrast with the current literature on FSI problems where only normal (radial/transverse) component of a thin structure displacement is considered,
	in the present work all three spatial displacement components are taken into account.
	%This complicates the proof significantly: the uniform Korn inequality can no longer be applied, and the design of a uniform divergence-free extension operator
	%needed in the compactness argument, can no longer be constructed.
	%This is physically motivated; the 1D hyperbolic net composed of metallic componentsdeforms predominantly in the tangental direction under 
	%tube inflation, and so taking into account only radial (normal) structure displacement, as is currently done in the FSI literature, 
	%would be an over-simplication. 
	This is the first work in which the constructive existence proof mimics a computational Arbitrary Lagrangian-Eulerian (ALE) approach
	to dealing with the motion of the fluid domain.
	Instead of mapping the problem onto a fixed, reference domain, in the present work approximate
	solutions in time are constructed on the ``current'' fluid domain 
	where the ``ALE time derivative'' is used to ``calculate'' the time derivatives on moving domains.

A major difficulty in this work is related to the introduction of the tangential components of displacement, which 
are associated with possibly having non-Lipschitz fluid domains, or domains which may degenerate and lose the subgraph property.
This is why we introduced two assumptions. One is the uniform Lipschitz property condition for approximate structure displacements,
and the other is the condition which assumes that 
there exists a time $T>0$ such that for every $t\leq T,$ the fluid domain boundary remains a subgraph of a 
function. 
The first condition allowed us to use some known results from functional analysis that hold for Lipschitz domains and construct 
extensions of divergence-free functions on moving domains to a divergence-free function on a ``maximal'' domain. 
The second condition allowed us to explicitly 
define a family of Arbitrary Lagrangian-Eulerian mappings between discrete domains at different times, and calculate the discrete 
time derivatives. 
While the second condition could have been avoided at the expense of a more complicated existence proof (which would hold until the fluid domain degeneracy), 
the first condition is crucial, and is intimately related to working with incompressible fluids. 
The uniform Lipschitz condition can be avoided in the case when only transverse displacements are considered, since in that case 
divergence-free extensions from moving domains
to a maximal domain can be explicitly constructed without the uniform Lipschitz property.
In the general case when all three displacement components are different from zero, 
the uniform Lipschitz condition can be shown to be satisfied for structures with a slightly higher regularity ($\varepsilon$-higher regularity) than the Koiter shell, such as,
e.g., 
tripolar materials studied in \cite{Boulakia2005,multipolar}.
The sixth-order derivative in the models studied in \cite{Boulakia2005,multipolar}
guarantees coercivity of the structure operator in $H^3$, and thus implies Lipschitz displacements in $\mathbb{R}^3$.
Otherwise, special conditions on the data and/or on the topology of the 1D mesh would have to be accounted for in order to be able to
guarantee that the uniform Lipschitz property holds for approximate structure displacements.

Problems of the type studied in this work arise in many real-life applications.
One example is the interaction between blood flow and arterial walls treated
with vascular prostheses called stents, which are used to prop diseased arteries open. 
Optimal design and performance of stents depends on the understanding of FSI problems
studied in this manuscript (see, e.g. \cite{butany2005coronary,CanicSIAM_News}).

\section{Literature Review}

The development of existence theory for moving boundary, fluid-structure interaction problems, has become particularly active since the late 1990's. The first existence results were obtained for the cases in which the structure was completely immersed in the fluid, and the structure was considered to be either a rigid body, or described by a finite number of modal functions \cite{Boulakia2003,Conca,DE2,DE3,Feireisl,GaldiTimePeriodic,GaldiHandbook,GaldiHemo}. 

These results were extended to involve elastic structures modeled by 2D or 3D linear elasticity, 
coupled to the 2D or 3D Navier-Stokes equations across a {\emph{fixed fluid-structure interface}}
in \cite{avalos2008higher, avalos2009semigroup,DuGunz03} for linear models,
and in \cite{Barbu,Barbu2,chueshov2016interaction,chueshov2013interaction, chueshov2011well,KukavicaTuffahaZiane10} for nonlinear models.

The first fluid-structure interaction existence result
 in which the coupling between the fluid and an elastic structure was assumed across a deformed, {\emph{moving interface}} whose location was not known {\sl a priori},
 was obtained in \cite{Hugo}, where the existence, locally in time, of a strong solution was obtained
 for an interaction between an incompressible, viscous two-dimensional fluid and a one-dimensional viscoelastic string,
  assuming periodic boundary conditions. 
 This result was extended in \cite{Lequeurre}, where the existence of a unique, local in time, strong solution
for any data, and the existence of a global strong solution for small data, were obtained
for a clamped viscoelastic beam. Coutand and Shkoller proved the existence, locally in time, of a unique regular solution 
for the interaction between a three-dimensional incompressible, viscous fluid, and a three-dimensional structure, immersed in fluid, where the structure was modeled by the equations of linear \cite{CSS1}, or quasi-linear \cite{CSS2} elasticity. 
Assuming lower regularity for the initial data,
Kukavica and Tuffaha obtained an existence result 
in the case when the structure was modeled by the linear wave equation \cite{Kuk}.
Existence of strong solutions to fluid-structure interaction problems between  three-dimensional viscous, incompressible fluids and two-dimensional elastic {\emph{shells}} was 
considered in \cite{ChengShkollerCoutand,ChengShkoller}, where local-in-time existence of unique regular solutions was proved.
Recently,  Grandmont and Hillariet proved the existence of a {\emph{global}} strong solution to a 2D FSI problem involving 
a viscoelastic beam, where they also showed that contact involving viscoelastic structure does not occur in finite time \cite{Grandmont16}.

In the context of weak solutions, the first global existence result was obtained in  \cite{CDEM}
for an unsteady interaction
of a three-dimensional incompressible, viscous fluid, and a two-dimensional viscoelastic plate. 
Grandmont improved this result to hold for a two-dimensional elastic plate in \cite{Grandmont08}. 

In \cite{MuhaCanic13}, the authors of the present manuscript introduced a new methodology for proving existence of weak solutions to FSI problems
involving incompressible, viscous fluids. The methodology is
based on time discretizing the FSI problem and using an operator splitting strategy to construct approximate
solutions to the FSI problem, and then proving that they converge to a weak solution of the continuous problem (see also \cite{GuiGloCan}).
This methodology was first applied to 
a 2D FSI problem modeling the flow of an incompressible, viscous, Newtonian fluid flowing through a cylinder whose lateral wall was modeled by 
the Koiter shell equations (both viscoelastic and linearly elastic),
assuming nonlinear coupling at the deformed fluid-structure interface.
A similar problem was also studied by Lengeler and R{\r u}{\v z}i{\v c}ka in \cite{LenRuz}.
The time discretization via operator splitting was then used by Muha and \v{C}ani\'{c} to study existence of a weak solution
to a 2D FSI problem with the Navier slip boundary condition \cite{BorSunSlip}.
This was the first manuscript in which the existence of a weak solution was studied for a problem
involving a thin structure allowing both transverse and tangential components of displacement. 
In contrast with the present work, the 2D setting of the problem in \cite{BorSunSlip} simplified the analysis. 

Encouraged by the robustness of the constructive existence approach, the same techniques were then extended to tackle
{\sl{three-dimensional}} FSI with the no-slip condition involving a linearly elastic, cylindrical Koiter shell allowing only {\sl{radial}} displacement \cite{BorSun3d},
and to a {\sl{nonlinear shell}} involving nonlinear membrane energy with an additional regularizing term \cite{BorSunNonLinearKoiter}.
Furthermore, the constructive existence approach was then also extended to prove existence of a weak solution to an FSI problem
involving a {\sl{composite, ``sandwich'' structure}} consisting of two layers: a thin Koiter shell and a thick structure of finite thickness, modeled by the
2D equations of linear elasticity \cite{BorSunMultiLayered}.
 
None of the works mentioned above, however, considered a composite structure involving a lower-dimensional elastic mesh of curved rods.
The only work in which a lower-dimensional 1D mesh was coupled to the elastodynamics of a 2D shell and the flow of a 3D viscous incompressible fluid 
was in our recent work \cite{LinearFSIstent}, where linear coupling was considered and thus the fluid domain was fixed,
and the fluid flow was modeled by the linear, time-dependent Stokes equations.
For completeness, we also mention two other works where an FSI problem approximating a stent-supported artery interacting with the flow of blood was studied   
\cite{StentFSI1,PironneauFSI}. In both works, however, the presence of a stent was modeled by the jump in the elasticity coefficients of the thin
shell or membrane structure,
and only radial (transverse) displacement was assumed to be different from zero.

The 1D hyperbolic net model, considered in the present work to model the elastodynamics of an elastic mesh of curved rods,
was first introduced in \cite{TambacaKosorCanic10}, where a static version of the model was developed. 
The 1D hyperbolic net model was proposed as an alternative to the engineering approaches in which a stent is modeled as a single 3D elastic body,
and approximated using 3D finite elements. 
Simulating thin stent components, i.e., stent struts, using 3D approaches, is computationally very expensive and is associated with large computer memory requirements.
The 1D model introduced in \cite{TambacaKosorCanic10} significantly reduces computational costs, and provides a tool for real-time evaluation of 
mechanical properties of mesh-like devices. A comparison with full 3D model simulations, published in \cite{CanicTambaca12}, showed excellent approximation
properties of the 1D model.
Although the model is one-dimensional, it describes structural deformation in all three spatial directions \cite{Antman}. 
The resulting model was justified computationally in \cite{CanicTambaca12}, and mathematically in \cite{Griso,JurakTambaca1,JurakTambaca2},
where approximation of the full 3D model by the 1D hyperbolic net was investigated. 

Encouraged by the excellent approximation properties  obtained in \cite{CanicTambaca12,Griso,JurakTambaca1,JurakTambaca2},
the authors utilized the 1D hyperbolic net model to study the behavior of mesh-supported shells in \cite{StentShell17},
where a static elasticity problem coupling 
the 1D hyperbolic net to an elastic cylindrical shell of Naghdi type was investigated.
No fluid was considered in  \cite{StentShell17}.
A time-dependent, incompressible Stokes fluid was added in \cite{LinearFSIstent} 
where a {\sl{linear}} FSI problem, coupling the fluid to the mesh-supported structure via a fixed, undeformed interface was studied. 

The present work extends the weak solution existence result of \cite{LinearFSIstent}  to a {\sl{nonlinear}} problem
by (1) considering the nonlinear flow modeled by the Navier-Stokes equations, and 
(2) coupling the fluid to the mesh-supported shell along the current, deformed
interface, giving rise to a strong geometric nonlinearity.
Proving existence of weak solutions to this {\sl{3D nonlinear problem}} allowing structural displacements in all {\sl{three spatial directions}},
is a culmination of the development of the constructive existence proof methodology, and a significant step forward in the analysis of moving boundary problems.
Moreover,  the finite energy solution spaces considered here present a natural framework to study existence of physically reasonable solutions
to this class of problems, since the presence of the 1D hyperbolic mesh of curved rods will likely not allow solutions with higher regularity.

\section{Model description}

\subsection{The fluid}
We consider the flow of an incompressible, viscous fluid, modeled by the Navier-Stokes equations, in a three-dimensional time-dependent cylindrical domain of reference length $L$ and reference radius $R.$ The reference fluid domain will be denoted by $\Omega$, and the reference lateral boundary by 
$\Gamma=\{(z,R\cos\theta,R\sin\theta) \in \mathbb{R}^3:z\in(0,L), \theta\in(0,2\pi)\}$.  
The boundary of the cylindrical domain consists of three parts: the lateral boundary, whose location is not known {\sl{a priori}} but depends on the motion of the fluid occupying the domain, the inlet boundary $\Gamma_{in}$ and the outlet boundary $\Gamma_{out}.$ 
See Fig.~\ref{fig:domain}. 
\begin{figure}[b]
	\centering
	\includegraphics[width=10cm]{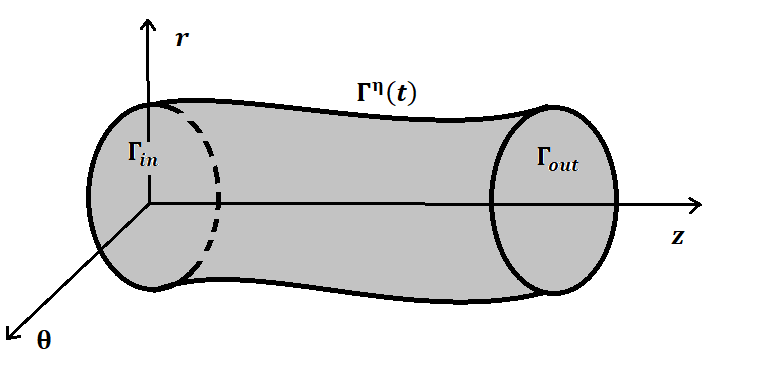}
	\caption{The fluid domain}
	\label{fig:domain}
\end{figure}
The lateral boundary is
a composite structure whose elastodynamics is modeled by the
linearly elastic Koiter shell equations coupled to a 1D hyperbolic net describing the motion of an elastic mesh of curved rods.
The 1D hyperbolic net is a collection of coupled linearly elastic one-dimensional curved rod equations defined on a graph domain,
interacting at graph's vertices, as described below.
The lateral boundary displacement is a vector function, which will be denoted by
 $\be:(0,T)\times \Gamma \to \R^2$ with $$\be(t,z,\theta)=(\eta_z(t,z,\theta),\eta_r(t,z,\theta),\eta_\theta(t,z,\theta)).$$ 
\begin{assumption}\label{3D_displ}
All three components of displacement will be assumed to be (non-zero) functions of $t,z$ and $\theta$. 
\end{assumption}
This contrasts other FSI works involving thin structures, where only radial  (transverse) component of a thin structure displacement is assumed to be different from zero.

The fluid domain deforms as a result of fluid-structure interaction between the flow of an incompressible, viscous fluid and the composite elastic structure. 
Therefore, the fluid domain is not known {\emph{a priori}}, as it depends on one of the unknowns in the problem, namely $\be$. 
We introduce $\bphi^{\be}(t,\cdot):\Omega\to \R^3, t\in (0,T),$ to describe the
time-dependent {\emph{deformation}} of the fluid domain.
The mapping $\bphi^{\be}$ is  an arbitrary, injective and orientation preserving mapping, and such that 
\begin{equation}\label{phi_mapping}
\bphi^{\be}(t)|_{\Gamma}=\id + \be(t,z,\theta).
\end{equation}
We denote by
\begin{equation}\label{mapped_domain}
\Omega^{\be}(t)= \bphi^{\be}(t,\Omega)\ {\rm and} \ \Gamma^{\be}(t)=\bphi^{\be}(t,\Gamma)
\end{equation}
the deformed fluid domain 
at time $t$, and the corresponding deformed lateral boundary, respectively. 
The superscript $\be$ emphasizes the dependence 
on displacement $\be$.

We are interested in studying a dynamic pressure-driven flow through $\Omega^{\be}(t)$ of an incompressible, viscous, Newtonian fluid modeled by the Navier-Stokes equations 

\begin{equation}\label{fluid}
\left.
\begin{array}{rcl}
\rho_F\left ({\partial_t{\bu}}+({\bu}\cdot\nabla) \bu\right)&=&\nabla\cdot\bsigma,\\
\nabla\cdot \bu &=& 0, \\
\end{array}
\right \}
\text{ in }\Omega^{\be}(t), t \in (0,T),
\end{equation}
where $\rho_F$ denotes the fluid density, $\bu$ is the fluid velocity, $\bsigma=-p\bI+2\mu_F \bD(\bu)$ is the fluid Cauchy stress tensor, $p$ is the fluid pressure, $\mu_F$ is the dynamic viscosity coefficient, and $\bD(\bu)=\frac{1}{2}(\nabla \bu + \nabla ^T \bu)$ is the symmetrized gradient of $\bu$.
At the inlet and outlet boundaries, we prescribe zero tangential velocity and dynamic pressure \cite{CMC94}:
\begin{equation}\label{fluid_bdry}
\left.
\begin{array}{rcl}
\dis p + \frac{\rho_F}{2}|\bu|^2 &=& P_{in/out}(t),\\
\bu\times \mathbf{e}_z &=& 0,
\end{array}
\right\}
\text{ on }\Gamma_{in/out},
\end{equation}
where $P_{in/out}$ are given. Therefore, the fluid flow is driven by a prescribed dynamic pressure drop, and the flow enters and leaves the fluid domain orthogonally to the inlet and outlet boundary.
No symmetry on the fluid flow is assumed.

\subsection{The shell}
The lateral boundary of the fluid domain is modeled by a mesh-supported shell.
The elastodynamics of the shell is described by the linearly elastic cylindrical Koiter shell equations capturing displacement in all three spatial directions \cite{Koiter} . The shell thickness will be denoted by $h>0,$ the length by $L$, and its reference radius of the middle surface by $R$. 
%Furthermore, we assume that the cylindrical shell is clamped at its end points. 
We use $\bvarphi$ to denote the parameterization that maps the set $\omega = (0,L)\times(0,2\pi)$ onto $\Gamma$:
$$ \bvarphi : \omega \to \mathbb{R}^3, \quad \bvarphi(z,\theta)=(z,R\cos\theta,R\sin\theta).$$
The first fundamental form of the cylinder $\Gamma$, or the metric tensor, is given in covariant $A_c$ or contravariant $A^c$ components  by
\begin{center}
	$A_c=
	\begin{pmatrix}
	1 & 0 \\
	0 & R^2\\
	\end{pmatrix}, 
	$\quad
	$A^c=
	\begin{pmatrix}
	1 & 0 \\
	0 & \frac{1}{R^2}
	\end{pmatrix},
	$
\end{center}
and the area element is $dS=\sqrt{\det A_c}\ dzd\theta=R\ dzd\theta.$
The second fundamental form of the cylinder $\Gamma$, or the curvature tensor in covariant components, is given by 
$$B_c=
\begin{pmatrix}
0 & 0 \\
0 & R
\end{pmatrix}.$$
Under the action of force, the Koiter shell is displaced from its reference configuration $\Gamma$ by a displacement $\be=\be(t,z,\theta)=(\eta_z,\eta_r,\eta_\theta),$ where $\eta_z, \eta_r$ and $\eta_\theta$  denote the tangential, radial and azimuthal components of displacement.  

The cylindrical Koiter shell is assumed to be clamped at the end points, giving rise to the following boundary conditions:
\begin{align*}
\be(t,0,\theta) &= \be(t,L,\theta) = 0,\;\theta\in(0,2\pi),\\
\partial_z \eta_r(t,0,\theta) &= \partial_z \eta_r(t,L,\theta) = 0,\;\theta\in (0,2\pi).
\end{align*}
At $\theta\in\{0,2\pi\}$, we assume periodic boundary conditions for structure displacement:
\begin{align*}
\be(t,z,0)&=\be(t,z,2\pi),\;z\in(0,L),\\
\partial_\theta \eta_r(t,z,0) &= \partial_\theta \eta_r(t,z,2\pi),\;z\in(0,L).
\end{align*}

The Koiter shell deformation under loading depends on its elastic properties, defined by the following elasticity tensor $\mathcal{A}$:
$$ \mathcal{A}E=\frac{4\lambda\mu}{\lambda+2\mu}(A^c\cdot E)A^c + 4\mu A^cEA^c,\quad E \in \text{Sym}(\R^2),$$
where $\mu$ and $\lambda$ are the Lam\'e constants. Stretching of the middle surface
and flexure will be measured by the following
linearized change of metric tensor $\bgamma$, and 
the following linearized change of curvature tensor $\brho$:

$$\bgamma(\be)=
\begin{pmatrix}
\partial_z\eta_z & \frac{1}{2}(\partial_\theta\eta_z + R\partial_z\eta_\theta)\\
\frac{1}{2}(\partial_\theta\eta_z + R\partial_z\eta_\theta) & R\partial_\theta\eta_\theta + R\eta_r
\end{pmatrix},$$

$$\brho(\be)=
\begin{pmatrix}
-\partial_{zz}\eta_r & -\partial_{z\theta}\eta_r + \partial_z\eta_\theta \\
-\partial_{z\theta}\eta_r + \partial_z\eta_\theta & -\partial_{\theta\theta}\eta_r + 2\partial_\theta\eta_\theta + \eta_r
\end{pmatrix}.
$$
Using $\bgamma(\be)$ and $\brho(\be)$ we can now define 
 the elastic energy of the deformed linear Koiter shell to be:
\begin{equation}\label{elastic_energy}
E(\be)=\frac{h}{2}\int_{\omega}\mathcal{A}\bgamma(\be):\bgamma(\be)R+
\frac{h^3}{24}\int_{\omega}\mathcal{A}\brho(\be):\brho(\be)R.
\end{equation}

The elastic energy of the Koiter shell, together with the boundary conditions, motivate the following solution and test spaces:
\begin{align}\label{shell_space}
\begin{split}
V_K = \{&\be=(\eta_z,\eta_r,\eta_{\theta})\in H^2(\omega)\times H^2(\omega)\times H^2(\omega):\\
&\be(t,z,\theta) =\partial_z \eta_r(t,z,\theta)= 0, z\in\{0,L\},\theta\in(0,2\pi), \\	
&\be(t,z,0)=\be(t,z,2\pi), \partial_\theta \eta_r(t,z,0) =\partial_\theta \eta_r(t,z,2\pi), z\in (0,L) \},
\end{split}
\end{align}
equipped with the corresponding norm:
\begin{equation*}
\|\be\|_{H^2(\omega)}^2:=\|\be\|_{H^2(\omega;\R^3)}^2 = \|\eta_z\|_{H^2(\omega)}^2 + \|\eta_r\|_{H^2(\omega)}^2 + \|\eta_\theta\|_{H^2(\omega)}^2.
\end{equation*}
Given a load $\mathbf{f}$, the displacement $\be$ of the Koiter shell is a solution to the following elastodynamics problem in weak form: 
find $\be=(\eta_z,\eta_r,\eta_\theta)\in V_K$ such that:
\begin{equation}\label{shell}
\rho_K h \int_{\omega}\partial_t^2\be\cdot \bpsi R + \langle\mathcal{L}\be,\bpsi\rangle = \int_{\omega}\mathbf{f}\cdot\bpsi R,\; \forall \bpsi \in V_K.
\end{equation}
Here, $\rho_K$ is the shell density and $\mathcal{L}$ is the following operator
describing the elastic properties of the shell, obtained from the elastic energy of the Koiter shell \eqref{elastic_energy}
with an added small regularizing term to guarantee coercivity in the axial and azimuthal directions:
\begin{align*}
\langle \mathcal{L}\be,\bpsi\rangle :=&h\int_{\omega}\mathcal{A}\bgamma(\be):\bgamma(\bpsi)R+\frac{h^3}{12}\int_{\omega}\mathcal{A}\brho(\be):\brho(\bpsi)R\\&+\varepsilon_K\int_{\omega}(\Delta\eta_z\Delta \psi_z+\Delta\eta_\theta\Delta\psi_\theta)R.
\end{align*}
By using the same procedure as in the proof of Theorem 2.6-4 \cite{CiarletVol3} (inequality of Korn's type on general surfaces), 
one can show that operator $\mathcal{L}$ is coercive on $H^2$:
 $$
 \langle \mathcal{L}\be,\be \rangle \geq c \|\be\|_{H^2(\omega)}^2, \forall \be \in V_K.
 $$

\noindent
The differential formulation of the shell elastodynamics problem on $(0,T)\times\omega$ is then given by:
\begin{equation}\label{shell_diff}
\rho_K h \partial_{t}^2 \be R + \mathcal{L}\be = \mathbf{f}R.
\end{equation}
Mathematical justification of the Koiter shell model can be found in \cite{KoiterJust}.

\subsection{The elastic mesh}
Elastic mesh is a three-dimensional elastic body composed of a union of three-dimensional slender components called struts. 
Since the aspect ratio of stent struts is small,
i.e., the ratio between the square root of the cross-sectional area and length is small,
a one-dimensional curved rod model can be used to approximate struts' elastodynamic properties. 
The one-dimensional curved rod model describes displacement and infinitesimal rotation of the cross-sections
of the rod as a function of time and of the location along the middle line of the curved rod. 
For the $i$-th curved rod of length $l_i$, the middle line of the curved rod is parameterized by
$$\bP_i: [0,l_i] \to \bvarphi(\overline{\omega}),\quad i = 1,\dots,n_E,$$ 
where $s \in [0,l_i]$ will be used to denote the parameter, and $n_E$ denotes the number of curved rods in the mesh. 

The 1D curved rod model for the $i$-th curved rod is given in terms of 
displacement $\bd_i(t,s)$ of the middle line from its reference configuration, 
infinitesimal rotation of cross-sections $\bw_i(t,s)$, 
contact moment $\bq_i(t,s)$, and contact force $\bp_i(t,s)$:
\begin{equation}\label{stent}
\begin{array}{rcl}
\rho_S A_i\partial_t^2 \bd_i &=&  \partial_s\bp_i  + \bff_i,\\
\rho_S M_i \partial_t^2 \bw_i &= & \partial_s\bq_i + \bt_i \times \bp_i, \\
0 &= & \partial_s\bw_i- Q_iH_i^{-1}Q_i^T\bq_i,\\
0 &= & \partial_s\bd_i + \bt_i \times \bw_i.
\end{array} 
\end{equation}
Here, $\rho_S$ is the strut's material density, $A_i$ is the cross-sectional area of the $i$-th rod, $M_i$ is the matrix related to the moments of inertia of cross-sections, $\bff_i$ is the line force density acting on the $i$-th rod, and $\bt_i$ is the unit tangent on the middle line of the $i$-th rod. Matrix $H_i$ is a positive definite matrix which describes the elastic properties and the geometry of cross sections, while matrix $Q_i\in SO(3)$ represents the local basis at each point of the middle line of the $i$-th rod (see \cite{Antman} for more details).

The first two equations describe the linear impulse-momentum law and the angular impulse-momentum law, respectively, while the last two equations describe a constitutive relation for a curved, linearly elastic rod, and the condition of inextensibility and unshearability, respectively. 

To model the elastodynamics of an elastic mesh, equations \eqref{stent} are posed on a graph domain whose edges correspond
to the middle lines of curved rods meeting at graph's vertices. More precisely,
let $\mathcal{V}$ denote the set of graph's vertices (points where the middle lines meet), and let $\mathcal{E}$ denote the set of graph's edges (pairing of vertices). 
Ordered pair $\mathcal{N} = (\mathcal{V},\mathcal{E})$ defines the mesh net topology.
\begin{defn}
A 1D hyperbolic net modeling an elastic mesh of curved rods, is given by
the system of equations \eqref{stent}, defined on a graph domain $\mathcal{N} = (\mathcal{V},\mathcal{E})$, such that 
the following coupling conditions hold at every vertex $V \in \mathcal{V}$:
\begin{itemize}
	\item kinematic coupling conditions describing continuity of displacements and continuity of infinitesimal rotations of all the rods meeting at $V$, 
	\item dynamic coupling conditions describing balance of contact forces and contact moments for all the rods meeting at $V$.
\end{itemize}
\end{defn}

To define weak solutions to the 1D hyperbolic net problem,
we first introduce
a function space consisting of all the $H^1$-functions $(\bd,\bw)$ 
defined on the entire net $\mathcal{N}$, such that they satisfy the kinematic coupling conditions at each vertex $V\in\mathcal{V}:$
\begin{align*}
H^1(\mathcal{N};\R^6) = \{ &(\bd,\bw)=((\bd_1,\bw_1),\dots,(\bd_{n_E},\bw_{n_E}))\in \prod_{i=1}^{n_E}H^1(0,l_i;\R^6): \\
&\bd_i(\bP_i^{-1}(V))=\bd_j(\bP_j^{-1}(V)),
\bw_i(\bP_i^{-1}(V))=\bw_j(\bP_j^{-1}(V)),\\
&\forall V\in \mathcal{V}, V = e_i \cap e_j,\; i,j=1,\dots,n_E\}.
\end{align*}
This space is used in the definition of the solution space, which additionally contains the  condition of inextensibility and unshearability:
\begin{equation}\label{stent_space}
V_S = \{ (\bd,\bw)\in H^1(\mathcal{N};\R^6) : \partial_s\bd_i+\bt_i\times \bw_i = 0, i = 1,\dots,n_E\}.
\end{equation}
For a function $(\bd,\bw)\in V_S$ we consider the following norms on $H^1(\mathcal{N};\R^3)$:
\begin{align*}
\|\bd\|_{H^1(\mathcal{N};\R^3)}^2 := \sum_{i=1}^{n_E}\|\bd_i\|_{H^1(0,l_i;\R^3)}^2,\quad
\|\bw\|_{H^1(\mathcal{N};\R^3)}^2 := \sum_{i=1}^{n_E}\|\bw_i\|_{H^1(0,l_i;\R^3)}^2,
\end{align*}
and the following norms on $L^2(\mathcal{N};\R^3)$:
\begin{align*}
\|\bd\|_{L^2(\mathcal{N};\R^3)}^2  :=  \sum_{i=1}^{n_E}\|\bd_i\|_{L^2(0,l_i;\R^3)}^2,\quad
\|\bw\|_{L^2(\mathcal{N};\R^3)}^2  :=  \sum_{i=1}^{n_E}\|\bw_i\|_{L^2(0,l_i;\R^3)}^2.
\end{align*}

The {\bf{weak formulation for a single curved rod}} is then obtained 
by adding the first equation in \eqref{stent} multiplied by a test function $\bxi_i$, and the second equation in \eqref{stent} multiplied by a test function $\bzeta_i$, integrating by
parts over $[0,l_i]$, and using the constitutive relation and the condition of inextensibility and unshearability.
The resulting weak formulation reads: find $(\bd_i,\bw_i)$ such that
\begin{align}\label{rod_weak}
\begin{split}
&\rho_S A_i \int_0^{l_i}\partial_t^2 \bd_i \cdot\bxi_i  +\rho_S \int_0^{l_i} M_i\partial_t^2\bw_i \cdot\bzeta_i  + \int_0^{l_i} Q_iH_iQ_i^T\partial_s\bw_i \cdot\partial_s\bzeta_i  \\
&=\int_0^{l_i}\bff_i \cdot\bxi_i +\bp_i(l_i)\cdot\bxi_i(l_i)-\bp_i(0)\cdot\bxi_i(0)+\bq_i(l_i)\cdot\bzeta_i(l_i)-\bq_i(0)\cdot\bzeta_i(0),
\end{split}
\end{align}
for all $(\bxi_i,\bzeta_i)\in H^1(0,l_i)\times H^1(0,l_i)$ that satisfy the condition of inextensibility and unshearability.

The {\bf{weak formulation for the 1D hyperbolic net problem}} is obtained by adding the weak 
 formulations for each component (i.e. curved rod) and using the dynamic coupling conditions at each vertex. 
The boundary terms from \eqref{rod_weak} involving $\bp_i$ and $\bq_i$ will add up to zero, giving rise to the following
weak formulation for the 1D hyperbolic net problem: find $(\bd,\bw)=((\bd_1,\bw_1),\dots,(\bd_{n_E},\bw_{n_E}))\in V_S$ such that
\begin{align}\label{stent_weak}
\begin{split}
\dis\rho_S\sum_{i=1}^{n_E}A_i\int_{0}^{l_i}\partial_t^2 \bd_i\cdot \bxi_i +
\dis\rho_S \sum_{i=1}^{n_E}\int_{0}^{l_i}M_i\partial_t^2 \bw_i\cdot \bzeta_i \\
+\dis\sum_{i=1}^{n_E}\int_0^{l_i}Q_iH_iQ_i^T\partial_s\bw_i\cdot \partial_s\bzeta_i
= \sum_{i=1}^{n_E}\int_{0}^{l_i} \bff_i \cdot \bxi_i,
\end{split}
\end{align}
for all the test functions $(\bxi,\bzeta)=((\bxi_1,\bzeta_1),\dots,(\bxi_{n_E},\bzeta_{n_E}))\in V_S.$

\subsection{The elastic shell-mesh coupling}
We will be assuming that the elastic mesh is affixed to the shell surface so that 
$$\dis \bigcup_{i=1}^{n_E}\bP_i([0,l_i])\subset \Gamma = \bvarphi(\overline{\omega}).$$
Since $\bvarphi$ is injective on $\omega$, the functions $\bpi_i$, defined by 
$$\bpi_i=\bvarphi^{-1}\circ \bP_i : [0,l_i] \to \overline{\omega},\quad i=1,\dots,n_E,$$ are well defined. 
See Fig.~\ref{fig:parameterizations}.
\begin{figure}[htp!]
	\centering
	\includegraphics[width=0.45\linewidth]{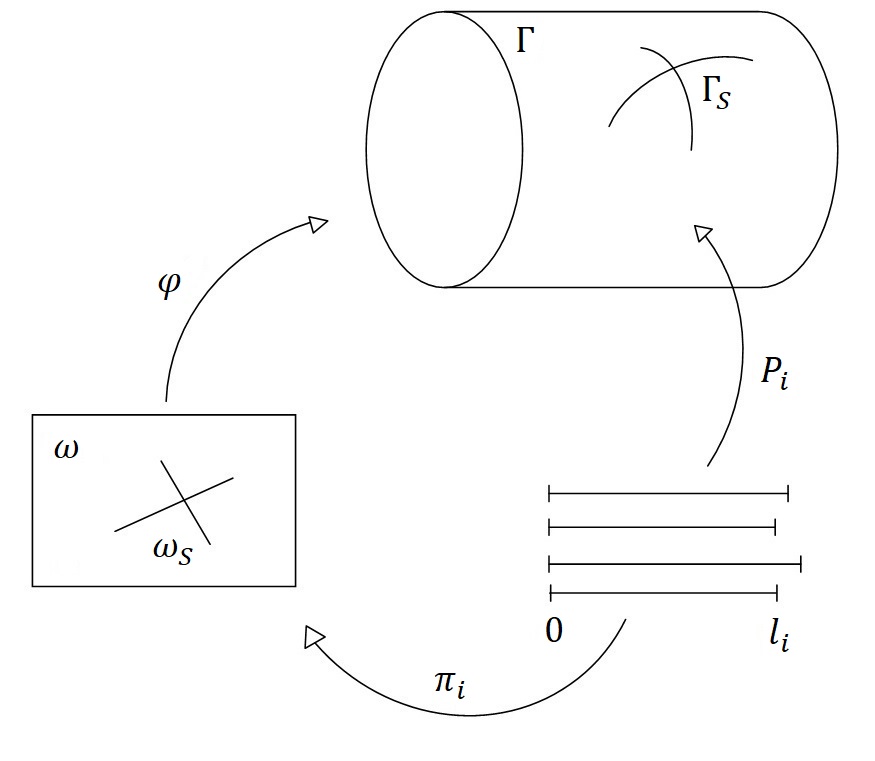}	
	\caption{Parameterization of the mesh and shell}
	\label{fig:parameterizations}
\end{figure}

The reference configuration of the mesh as a subset of $\omega$ will be denoted by $\omega_S=\bigcup_{i=1}^{n_E}\bpi_i([0,l_i])$, and the reference configuration of the mesh 
as a subset of $\Gamma$ will be denoted by $\Gamma_S=\bigcup_{i=1}^{n_E}\bP_i([0,l_i])$, see Fig.~\ref{fig:parameterizations}.
The coupling between the elastic mesh and shell is defined via the kinematic and dynamic coupling conditions.

The {\bf{kinematic coupling condition}} states that the displacement of the shell at the point\\ $(z,R\cos\theta,R\sin\theta)\in \Gamma$, that is associated with the point $(z,\theta)\in \omega_S$ via the mapping $\bvarphi$, is equal to the displacement of the mesh at the point $s_i=\bpi_i^{-1}(z,\theta)$, that is associated to the same point $(z,\theta)\in \omega_S$ via the mapping $\bpi_i.$
For a point $s_i\in [0,l_i]$ such that $\bpi_i(s_i)=(z,\theta)\in\omega_S$, the kinematic coupling condition reads:
\begin{equation}
\be(t,\bpi_i(s_i))=\bd_i(t,s_i).
\end{equation}

The {\bf{dynamic coupling condition}} describes the balance of forces. The force exerted by the Koiter shell onto the mesh is balanced by the force exerted by the mesh onto the Koiter shell. More precisely, let $J_i=\bpi_i([0,l_i])$, and $$\langle \delta_{J_i},f\rangle=\int_{J_i}fd\gamma_i, \quad i =1,\dots, n_E,$$
where $d\gamma_i$ is the curve element associated with the parameterization $\bpi_i$.
The weak formulation of the shell \eqref{shell} can then be written as:
\begin{align*}
\rho_K h \int_{\omega}\partial_t^2\be\cdot \bpsi R+ 
\langle\mathcal{L}\be,\bpsi\rangle  &= \dis\sum_{i=1}^{n_E}\langle \delta_{J_i},\bff\cdot\bpsi R\rangle
= \sum_{i=1}^{n_E}\int_{J_i}\bff(z,\theta)\cdot\bpsi(z,\theta)Rd\gamma_i\\
&= \sum_{i=1}^{n_E}\int_{0}^{l_i}\bff(\bpi_i(s))\cdot\bpsi(\bpi_i(s))\|\bpi_i'(s)\|Rds.
\end{align*}
If we denote by $\bff_i$ the force exerted by the $i$-th mesh strut onto the shell, 
by force balance, the right-hand side has to be equal to $-\dis\sum_{i=1}^{n_E}\int_{0}^{l_i}\bff_i(s)\cdot\bxi_i(s)ds$.
Thus, 
$$ \bff(\bpi_i(s_i))\|\bpi_i'(s_i)\|R= -\bff_i(s_i), \quad {\rm i.e.}, \quad \bff(\bpi_i(s_i))R=-\dis\frac{\bff_i(s_i)}{\|\bpi_i'(s_i)\|},\  s_i \in (0,l_i).$$
For a point $(z,\theta) = \bpi_i(s_i) \in \omega_S$, which came from an $s_i \in (0,l_i)$, the dynamic coupling condition reads: 
$ \dis \bff R=- \frac{\bff_i\circ\bpi_i^{-1}}{\|\bpi_i'\circ \bpi_i^{-1}\|}.$
For an arbitrary point $(z,\theta)\in \omega$, the dynamic coupling condition reads:
\begin{equation}
\dis \bff R = -\sum_{i=1}^{n_E}\frac{\bff_i\circ\bpi_i^{-1}}{\|\bpi_i'\circ \bpi_i^{-1}\|}\delta_{J_i},
\end{equation}
which implies the following weak formulation for the coupled mesh-shell problem:
\begin{equation}
\rho_K h \int_{\omega}\partial_t^2\be\cdot \bpsi R+ 
\langle\mathcal{L}\be,\bpsi\rangle  =
-\sum_{i=1}^{n_E} \langle \delta_{J_i}, \frac{\bff_i\circ\bpi_i^{-1}}{\|\bpi_i'\circ \bpi_i^{-1}\|}\bpsi \rangle,\quad \forall \bpsi \in V_K.
\end{equation}
Here $\bff_i$'s are defined in \eqref{stent_weak}, and the test functions $\bxi_i$ are such that 
$\bpsi \circ \bpi_i = \bxi_i$.
The coupled mesh-shell  weak solution space is given by:
$$V_{KS}= \{(\be,\bd,\bw)\in V_K\times V_S:\be\circ\bpi = \bd\text{ on }\prod_{i=1}^{n_E}(0,l_i)\},$$
where  $\be\circ\bpi=(\be\circ\bpi_1,\dots,\be\circ\bpi_{n_E})$.

\subsection{The fluid-structure coupling}

From this point on, we refer to the Koiter shell coupled with an elastic mesh as the ``composite structure''. The coupling between the fluid and the
composite structure is defined by two sets of boundary conditions satisfied at the lateral boundary $\Gamma^{\be}(t),$ giving rise to a \emph{nonlinear fluid-structure coupling.} They are the kinematic and dynamic boundary conditions. 

The {\bf{kinematic coupling condition}} describes continuity of velocity at the moving interface:
 $$\partial_t\be = (\bu\circ\bphi^{\be})|_{\Gamma}\circ\bvarphi \ {\rm on} \ (0,T)\times \omega.$$
 
 The {\bf{dynamic coupling condition}} describes balance of forces at the fluid-structure interface:
$$\rho_Kh\partial_t^2 \be R + \mathcal{L}\be + \dis\sum_{i=1}^{n_E}\dis \frac{\bff_i\circ\bpi_i^{-1}}{\|\bpi_i'\circ \bpi_i^{-1}\|}\delta_{J_i} = -J((\bsigma\circ\bphi^{\be})|_{\Gamma}
	\circ\bvarphi)((\bn\circ\bphi^{\be})|_{\Gamma}\circ\bvarphi) \ {\rm on} \ (0,T)\times \omega,$$
where $J$ denotes the Jacobian of the composite transformation from 
Eulerian to Lagrangian coordinates and the transformation from Cartesian to cylindrical coordinates, 
and $\bn,$ which depends on $\be,$ denotes the outer unit normal on $\Gamma^{\be}(t)$ at the point $\bphi^{\be}(t,\bvarphi(z,\theta)).$ 

In summary, we study the following fluid-structure interaction problem.
\vskip 0.05in
\textbf{Main Problem:} Find $(\bu,p,\be,\bd,\bw)$ such that

\begin{equation}\label{FSIeq1}
\small
\left.
\begin{array}{rcl}
\rho_F\left(\partial_t{\bu} + (\bu\cdot\nabla)\bu\right)&=&\nabla\cdot\bsigma  \\ 
\nabla\cdot \bu&=&0  
\end{array}
\right\}
\ \textrm{in}\;\Omega^{\be}(t),t\in (0,T), 
\end{equation}

\begin{equation}\label{FSIeq2}
\small
\left.
\begin{array}{rcl}
\partial_t\be &=& (\bu\circ\bphi^{\be})|_{\Gamma}\circ\bvarphi\\
\rho_Kh\partial_t^2 \be R + \mathcal{L}\be + \dis\sum_{i=1}^{n_E}\dis \frac{\bff_i\circ\bpi_i^{-1}}{\|\bpi_i'\circ \bpi_i^{-1}\|}\delta_{J_i} &=&
-J((\bsigma\circ\bphi^{\be})|_{\Gamma}
\circ\bvarphi)
((\bn\circ\bphi^{\be})|_{\Gamma}\circ\bvarphi)
\end{array}
\right\} \ \textrm{on}\; (0,T)\times\omega,
\end{equation}

\begin{equation}\label{FSIeq3}
\small
\left.
\begin{array}{rcl}
\rho_S A_i\partial_t^2 \bd_i &=&  \partial_s\bp_i  + \bff_i\\
\rho_S M_i \partial_t^2 \bw_i &= & \partial_s\bq_i + \bt_i \times \bp_i  \\
0 &= & \partial_s\bw_i- Q_iH_i^{-1}Q_i^T\bq_i \\
0 &= & \partial_s\bd_i + \bt_i \times \bw_i 
\end{array}
\right\}
\ \textrm{on}\; (0,T)\times (0,l_i), i\in(1,\dots,n_E).
\end{equation}
Problem \eqref{FSIeq1}-\eqref{FSIeq3} is supplemented with the following set of boundary and initial conditions:
\begin{equation}\label{FSIeqBC}
\small
\left \{
\begin{array}{rcll}
\dis p+\frac{\rho_F}{2}|\bu|^2&=& P_{in/out}(t), \;  \; &\textrm{on}\; (0,T)\times\Gamma_{in/out},\\
\bu\times \mathbf{e}_z &=& 0, \;  \; &\textrm{on}\;(0,T)\times\Gamma_{in/out},\\
\be(t,0,\theta)&=&\be(t,L,\theta)=0,\;\; &\textrm{on}\;(0,T)\times (0,2\pi),\\
\partial_z \eta_r(t,0,\theta) &= &\partial_z \eta_r(t,L,\theta)=0,\;\; &\textrm{on}\;(0,T)\times (0,2\pi),\\
\be(t,z,0)&=&\be(t,z,2\pi),\;\; &\textrm{on}\;(0,T)\times (0,L),\\
\partial_\theta \eta_r(t,z,0) &=&\partial_\theta \eta_r(t,z,2\pi),\;\; &\textrm{on}\;(0,T)\times (0,L),
\end{array}
\right .
\end{equation}
\begin{equation}\label{FSIeqIC}
\small
\begin{array}{c}
\bu(0)=\bu_0,\; \be(0)=\be_0,\; \partial_t\be(0)=\bv_0,\\
\bd_i(0)=\bd_{0i},\; \partial_t\bd_i(0)=\bk_{0i},\;
\bw_i(0)=\bw_{0i},\; \partial_t\bw_i(0)=\bz_{0i}.
\end{array}
\end{equation}
\vskip 0.05in
This is a nonlinear, moving boundary problem in 3D, which captures the full, two way fluid-structure interaction coupling. 
The nonlinearity in the problem is represented by the quadratic term in the fluid equations, 
and by the geometric nonlinearity due to the fluid-structure coupling at the current location of the moving boundary
$\Gamma^{\be}(t)$, which is one of the unknowns in the problem. 

\section{The energy of the coupled problem}
Problem \eqref{FSIeq1}-\eqref{FSIeqIC} satisfies the following energy inequality:
\begin{equation}\label{energy}
\frac{d}{dt} E(t) + D(t) \leq C(P_{in}(t),P_{out}(t)),
\end{equation}
where $E(t)$ denotes the sum of the total kinetic and elastic energy:
\begin{align}
E(t) &= \frac{\rho_F}{2}\|\bu\|_{L^2(\Omega^{\be}(t))}^2 + \frac{\rho_K h}{2}\|\partial_t\be\|_{L^2(R;\omega)}^2 + \frac{\rho_S}{2}\sum_{i=1}^{n_E}A_i\|\partial_t \bd_i \|_{L^2(0,l_i)}^2
\nonumber
\\
&\quad +
\frac{\rho_S}{2}\|\partial_t\bw\|_m^2 + \frac{1}{2}\langle \mathcal{L}\be,\be\rangle + \|\bw\|_r^2,
\label{E}
\end{align} 
and $D(t)$ denotes dissipation due to fluid viscosity:
$$D(t) = 2\mu_F \|\bD(\bu)\|_{L^2(\Omega^{\be}(t))}^2.$$
The constant $C(P_{in}(t),P_{out}(t))$ depends only on the inlet and outlet pressure data, which are both functions of time.

The norms $ \|\bw\|_m$ and $ \|\bw\|_r$ in \eqref{E} denote the kinetic and elastic energy of the mesh:
$$
\dis \|\bw\|_m^2 := \sum_{i=1}^{n_E}\|\bw_i\|_m^2 = \sum_{i=1}^{n_E}\int_0^{l_i}M_i\bw_i\cdot\bw_i,\ 
\dis\|\bw\|_r^2:=\sum_{i=1}^{n_E}\|\bw_i\|_r^2 =\sum_{i=1}^{n_E}\int_0^{l_i}Q_iH_iQ_i^T \partial_s\bw_i\cdot\partial_s\bw_i,
$$
and $\|\be\|_{L^2(R;\omega)}$ denotes the weighted $L^2$-norm on $\omega$ associated with the kinetic energy of the Koiter shell:
$$
\|\be\|^2_{L^2(R;\omega)} := \int_\omega |\be|^2 R \ d\omega,
$$
where the weight $R$ comes from the geometry of 
the Koiter shell (Jacobian).

\begin{remark}
	The norm $\|\cdot\|_m$ is equivalent to the standard $L^2(\mathcal{N})$ norm.
\end{remark}

The formal energy inequality \eqref{energy} can be derived in a standard way as follows.
First, after multiplying the first equation in \eqref{fluid} by $\bu$ and integrating over $\Omega^{\be}(t)$ we obtain:
\begin{align}\label{NScalc}
\int_{\Omega^{\be}(t)}\rho_F\left(\partial_t\bu\cdot \bu + (\bu\cdot\nabla)\bu \cdot\bu\right)=\int_{\Omega^{\be}(t)}(\nabla\cdot\bsigma)\cdot \bu.
\end{align} 
The first term on the left-hand side can be rewritten by using the Reynold's transport theorem and keeping in 
mind that the velocity of the lateral boundary equals $\bu|_{\Gamma^{\be}(t)}$:
\begin{align*}
\int_{\Omega^{\be}(t)}\partial_t\bu \cdot \bu &= \frac{d}{dt} \int_{\Omega^{\be}(t)} \frac{|\bu|^2}{2} - \int_{\partial\Omega^{\be}(t)}\frac{|\bu|^2}{2}\bu\cdot\bn
= \frac{1}{2}\frac{d}{dt}\int_{\Omega^{\be}(t)}|\bu|^2 - \frac{1}{2}\int_{\Gamma^{\be}(t)}|\bu|^2 \bu\cdot \bn.
\end{align*}
The convective part of the Navier-Stokes equations can be rewritten by using integration by parts and divergence-free condition to obtain:
\begin{align*}
\int_{\Omega^{\be}(t)}(\bu\cdot\nabla)\bu\cdot \bu &= \frac{1}{2}\int_{\Omega^{\be}(t)}\nabla\cdot(|\bu|^2\bu)\\
&=\frac{1}{2}\int_{\Gamma^{\be}(t)}|\bu|^2\bu\cdot \bn
- \frac{1}{2}\int_{\Gamma_{in}}|\bu|^2\bu\cdot \mathbf{e}_z
+ \frac{1}{2}\int_{\Gamma_{out}}|\bu|^2\bu\cdot \mathbf{e}_z.
\end{align*}
These two terms added together give:
\begin{align}\label{NS_left}
\begin{split}
\int_{\Omega^{\be}(t)}\partial_t \bu \cdot \bu + \int_{\Omega^{\be}(t)}(\bu\cdot\nabla)\bu\cdot \bu
&= \frac{1}{2}\frac{d}{dt}\|\bu\|_{L^2(\Omega^{\be}(t))}^2
- \frac{1}{2}\int_{\Gamma_{in}}|\bu|^2u_z + \frac{1}{2}\int_{\Gamma_{out}}|\bu|^2u_z.
\end{split}
\end{align}
Next, using integration by parts the right-hand side of \eqref{NScalc} can be rewritten as:
\begin{equation*}
\int_{\Omega^{\be}(t)}(\nabla\cdot\bsigma)\cdot\bu =
\int_{\partial\Omega^{\be}(t)}\bsigma\bn\cdot\bu - 2\mu_F \int_{\Omega^{\be}(t)}|\bD(\bu)|^2.
\end{equation*}

To deal with the boundary integral on the right-hand side we first notice that on $\Gamma_{in/out}$ the boundary condition $\bu\times \mathbf{e}_z=0$ implies $u_x =u_y=0$.  Using the divergence-free condition we also obtain $\partial_z u_z = 0.$ These two facts combined imply that $\bD(\bu)\bn\cdot\bu=0$ on $\Gamma_{in/out}.$ 
Finally, because the normal on $\Gamma_{in/out}$ is $\bn=(\mp 1,0,0),$ we get:
\begin{align}\label{NS_right}
\begin{split}
\int_{\partial\Omega^{\be}(t)}\bsigma\bn\cdot\bu 
= \int_{\Gamma^{\be}(t)}\bsigma\bn\cdot\bu + \int_{\Gamma_{in}}pu_z - \int_{\Gamma_{out}}pu_z.
\end{split}
\end{align}
What is left is to calculate the boundary integral over $\Gamma^{\be}(t).$  By enforcing the kinematic and dynamic coupling conditions on $\omega,$ we obtain:
\begin{align}\label{boundary1}
\begin{split}
-\int_{\Gamma^{\be}(t)}\bsigma \bn \cdot \bu &= -\int_{\omega}J((\bsigma\circ\bphi^{\be})|_{\Gamma}\circ\bvarphi) ((\bn\circ\bphi^{\be})|_{\Gamma}\circ\bvarphi) \cdot ((\bu\circ\phi^{\be})|_{\Gamma}\circ\bvarphi)\\
&= \int_{\omega} \bff\cdot\partial_t\be R + \sum_{i=1}^{n_E}\int_{J_i}   \frac{\bff_i\circ\bpi_i^{-1}}{\|\bpi_i'\circ \bpi_i^{-1}\|}\delta_{J_i}\cdot \partial_t \be \\
&= \int_{\omega} \bff\cdot \partial_t\be R + \sum_{i=1}^{n_E}\int_{0}^{l_i}  \bff_i\cdot \partial_t \be\circ\bpi_i
= \int_{\omega} \bff\cdot \partial_t\be R + \sum_{i=1}^{n_E}\int_{0}^{l_i}  \bff_i\cdot\partial_t \bd_i.
\end{split}
\end{align}
Next we multiply the Koiter shell equation \eqref{shell_diff} by $\partial_t\be$ 
and integrate over $\omega,$ and use $(\partial_t \bd,\partial_t \bw)
=((\partial_t \bd_1,\partial_t \bw_1),\dots,(\partial_t \bd_{n_E},\partial_t \bw_{n_E}))$ 
as a test function in the weak formulation for the mesh problem \eqref{stent_weak} 
to obtain that \eqref{boundary1} equals:
\begin{align}\label{boundary2}
\begin{split}
-\int_{\Gamma^{\be}(t)}\bsigma \bn \cdot \bu
&= \frac{\rho_K h}{2}\frac{d}{dt}\|\partial_t\be\|_{L^2(R;\omega)}^2
+ \frac{1}{2}\frac{d}{dt}\langle \mathcal{L}\be,\be\rangle
+ \frac{\rho_S}{2}\frac{d}{dt}\sum_{i=1}^{n_E}A_i\|\partial_t \bd_i\|_{L^2(0,l_i)}^2 \\
&\quad + \frac{\rho_S}{2}\frac{d}{dt}\sum_{i=1}^{n_E}\|\partial_t \bw_i\|_m^2
+ \frac{d}{dt}\sum_{i=1}^{n_E}\|\bw_i\|_r^2.
\end{split}
\end{align}
Finally, by combining \eqref{boundary2} with \eqref{NS_right}, and by adding the remaining contributions to the energy of the FSI problem calculated in equations \eqref{NS_right} and \eqref{NS_left}, one obtains the following energy equality:
\begin{align}
\begin{split}
&\frac{\rho_F}{2}\frac{d}{dt}\|\bu\|_{L^2(\Omega^{\be}(t))}^2 + 2\mu_F \|\bD(\bu)\|_{L^2(\Omega^{\be}(t))}^2 +  \frac{\rho_K h}{2}\frac{d}{dt}\|\partial_t\be\|_{L^2(R;\omega)}^2\\
& + \frac{1}{2}\frac{d}{dt}\langle \mathcal{L}\be,\be\rangle + \frac{\rho_S}{2}\frac{d}{dt}\sum_{i=1}^{n_E}A_i\|\partial_t \bd_i \|_{L^2(0,l_i)}^2 + \frac{\rho_S}{2}\frac{d}{dt}\sum_{i=1}^{n_E}\|\partial_t \bw_i \|_m^2\\
& + \frac{d}{dt}\sum_{i=1}^{n_E}\|\bw_i\|_r^2 =  \int_{\Gamma_{in}}P_{in}(t)u_z - \int_{\Gamma_{out}}P_{out}(t)u_z.
\end{split}
\end{align}
Using the trace theorem, Korn's inequality and Cauchy inequality (with $\varepsilon$), one can estimate:
\begin{align*}
\begin{split}
\Big|\int_{\Gamma_{in/out}}P_{in/out}(t)u_z \Big| &\leq C |P_{in/out}| \|\bu\|_{H^1(\Omega^{\be}(t))}\leq C |P_{in/out}| \|\bD(\bu)\|_{L^2(\Omega^{\be}(t))}
\\ &\leq \frac{C}{2\varepsilon}|P_{in/out}|^2 + \frac{C\varepsilon}{2}\|\bD(\bu)\|_{L^2(\Omega^{\be}(t))}^2.
\end{split}
\end{align*}
We note that the fluid velocity $\bu$ indeed satisfies the conditions for Korn's inequality (Theorem~6.3-4 in \cite{CiarletVol1}),
i.e.,  $\bu=0$ on $\Gamma_{in/out}$. 
Namely, the boundary condition $\bu\times\mathbf{e}_z=0$ on $\Gamma_{in/out}$ 
and divergence-free condition $\nabla\cdot\bu=0$ imply $\partial_z u_z =0.$ 
From the kinematic coupling condition $u_z=\partial_t \eta_z$ (on $\omega$), we obtain $u_z=0$, so $\bu=0$ on $\Gamma_{in/out}$. 
Finally, by choosing $\varepsilon$ such that $\frac{C\varepsilon}{2}\leq\mu_F,$ we get the energy inequality \eqref{energy}.

\begin{remark}
 Notice that the constant in the trace inequality depends on the fluid domain,
which in our case depends on $\boldsymbol\eta$. To get an energy estimate in which the constant is
independent of  $\boldsymbol\eta$, one can use Gronwall's inequality, and obtain the result above in
which the constant C depends on time T.
 \end{remark}

\section{Weak solutions}
To define weak solutions to problem \eqref{FSIeq1}-\eqref{FSIeqIC} we introduce the necessary function spaces. For the fluid velocity we will be using
 the following classical function space:
\begin{equation}\label{fluid_space}
V_F(t)=\{\bu\in H^1(\Omega^{\be}(t)): \nabla\cdot \bu = 0, \bu\times\mathbf{e}_z = 0\text{ on }\Gamma_{in/out}\}.
\end{equation}
Motivated by the energy inequality \eqref{energy}, we also define the following evolution spaces associated with the fluid problem, the Koiter shell problem, the elastic mesh problem, and the coupled mesh-shell problem, respectively:
\begin{itemize}
	\item $V_F(0,T)=L^\infty(0,T;L^2(\Omega^{\be}(t)))\cap L^2(0,T;V_F(t)),$
	where $V_F(t)$ is defined by \eqref{fluid_space},\smallskip
	\item $V_K(0,T)=W^{1,\infty}(0,T;L^2(R;\omega))\cap L^\infty(0,T;V_K),$
	where $V_K$ is defined by \eqref{shell_space},\smallskip
	\item $V_S(0,T)=W^{1,\infty}(0,T;L^2(\mathcal{N}))\cap L^\infty(0,T;V_S),$
	where $V_S$ is defined by \eqref{stent_space},\smallskip
	\item $\dis V_{KS}(0,T)=\{(\be,\bd,\bw)\in V_K(0,T)\times V_S(0,T):\be\circ\bpi = \bd\text{ on }\prod_{i=1}^{n_E}(0,l_i)  \}.$
\end{itemize}

\noindent
The solution space for the coupled fluid-mesh-shell interaction problem includes the kinematic coupling condition, which is, thus, enforced in a strong way:
$$\mathcal{V}(0,T)=\{(\bu,\be,\bd,\bw)\in V_F(0,T)\times V_{KS}(0,T):  (\bu\circ\bphi^{\be})|_{\Gamma}\circ\bvarphi = \partial_t\be\text{ on }\omega \}.$$

\noindent
The corresponding test space will be denoted by:
\begin{equation}\label{test_space_Q}
\mathcal{Q}(0,T)=\{(\bups,\bpsi,\bxi,\bzeta)\in C_c^1([0,T);V_F\times V_{KS}):  (\bups\circ\bphi^{\be})|_{\Gamma}\circ\bvarphi = \bpsi\text{ on }\omega \}.
\end{equation}

%\subsection{Definition of a weak solution}
We are now in position to define weak solutions to our moving boundary, fluid-mesh-shell interaction problem.

\begin{defn}\label{def:weak}
	We say that $(\bu,\be,\bd,\bw)\in \mathcal{V}(0,T)$ is a weak solution of problem \eqref{FSIeq1}-\eqref{FSIeqIC}, if for all test functions $(\bups,\bpsi,\bxi,\bzeta)\in \mathcal{Q}(0,T)$ the following equality holds:
	\begin{align}\label{weak}
	\begin{split}
	&\rho_F\left( -\int_{0}^{T}\int_{\Omega^{\be}(t)}\bu\cdot\partial_t \bups + \int_{0}^T b(t,\bu,\bu,\bups) -\frac{1}{2}\int_{0}^T\int_{\Gamma^{\be}(t)}(\bu\cdot\bups)(\bu\cdot \bn) \right)\\
	&+2\mu_F\int_{0}^{T}\int_{\Omega^{\be}(t)}\bD(\bu):\bD(\bups)
	- \rho_K h \int_0^T\int_{\omega}\partial_t\be\cdot\partial_t \bpsi R + \int_0^T a_K(\be,\bpsi)\\
	&- \rho_S\sum_{i=1}^{n_E}A_i\int_0^T\int_0^{l_i}\partial_t \bd_i \cdot \partial_t \bxi_i - \rho_S\sum_{i=1}^{n_E}\int_0^T\int_0^{l_i}M_i\partial_t \bw_i \cdot \partial_t \bzeta_i \\
	&+ \int_0^T a_S(\bw,\bzeta) =\int_0^T\langle F(t),\bups\rangle_{\Gamma_{in/out}} + \rho_F\int_{\Omega}\bu_0\cdot \bups(0) +\rho_K h \int_{\omega}\partial_t\be_0\cdot\bpsi(0)R\\
	&+ \rho_S \sum_{i=1}^{n_E}A_i\int_{0}^{l_i}\partial_t\bd_{0i}\cdot \bxi_i(0) + \rho_S \sum_{i=1}^{n_E}\int_0^{l_i}M_i \partial_t\bw_{0i}\cdot \bzeta_i(0),
	\end{split}
	\end{align}
	where $$b(t,\bu,\bu,\bups)=\frac{1}{2}\int_{\Omega^{\be}(t)}(\bu\cdot\nabla)\bu\cdot\bups - \frac{1}{2}\int_{\Omega^{\be}(t)}(\bu\cdot\nabla)\bups\cdot \bu,$$
	$$ a_K(\be,\bpsi)=\langle \mathcal{L}\be,\bpsi\rangle,$$
	$$a_S(\bw,\bzeta)=\sum_{i=1}^{n_E}\int_{0}^{l_i}Q_iH_iQ_i^T\partial_s\bw_i\cdot\partial_s\bzeta_i,$$
	and $$\langle F(t),\bups\rangle_{\Gamma_{in/out}}=P_{in}(t)\int_{\Gamma_{in}}\upsilon_z - P_{out}(t)\int_{\Gamma_{out}}\upsilon_z.$$
\end{defn}
In deriving the weak formulation, we used integration by parts in a classical way. Here we only show the transformation of the fluid inertial and convective terms:
\begin{align*}
\rho_F\int_{\Omega^{\be}(t)}\partial_t\bu\cdot\bups &=
\rho_F\int_{\Omega^{\be}(t)}\partial_t(\bu\cdot\bups)- \rho_F\int_{\Omega^{\be}(t)}\bu\cdot\partial_t\bups\\
&=\rho_F\frac{d}{dt}\int_{\Omega^{\be}(t)}\bu\cdot\bups - \rho_F \int_{\partial\Omega^{\be}(t)}(\bu\cdot\bups)(\bu\cdot\bn)-\rho_F\int_{\Omega^{\be}(t)}\bu\cdot \partial_t\bups\\
&= - \rho_F \bu_0\cdot\bups(0) - \rho_F\int_{\Gamma^{\be}(t)}(\bu\cdot\bups)(\bu\cdot\bn)-\rho_F\int_{\Omega^{\be}(t)}\bu\cdot \partial_t\bups,
\end{align*}
\begin{align*}
\rho_F\int_{\Omega^{\be}(t)}(\bu\cdot\nabla)\bu\cdot\bups &=
\frac{\rho_F}{2}\int_{\Omega^{\be}(t)}(\bu\cdot\nabla)\bu\cdot\bups + \frac{\rho_F}{2}\int_{\Omega^{\be}(t)}(\bu\cdot\nabla)\bu\cdot\bups\\
&= \frac{\rho_F}{2}\int_{\partial\Omega^{\be}(t)}(\bu\cdot\bups)(\bu\cdot\bn)-
\frac{\rho_F}{2}\int_{\Omega^{\be}(t)}(\nabla\cdot\bu)\bu\cdot\bups\\
&\quad - \frac{\rho_F}{2}\int_{\Omega^{\be}(t)}(\bu\cdot\nabla)\bups\cdot\bu+
\frac{\rho_F}{2}\int_{\Omega^{\be}(t)}(\bu\cdot\nabla)\bu\cdot\bups\\
&= \rho_F b(t,\bu,\bu,\bups)
+ \frac{\rho_F}{2}\int_{\Gamma^{\be}(t)}(\bu\cdot\bups)(\bu\cdot\bn)
- \frac{\rho_F}{2}\int_{\Gamma_{in}}|\bu|^2\upsilon_z + 
\frac{\rho_F}{2}\int_{\Gamma_{out}}|\bu|^2\upsilon_z.
\end{align*}

\section{Lie splitting and ALE mapping}

Approximate solutions are constructed by discretizing the problem in time,
and by splitting the coupled problem into a fluid  and a structure subproblem using the Lie splitting strategy \cite{glowinski2003finite}.
To deal with the motion of the fluid domain, an Arbitrary Lagrangian-Eulerian (ALE) approach is used. 
%Because we allow nonzero tangential displacements, 
%and because the coupling between the fluid and structure is assumed along the moving interface, 
%the fluid domain boundary can cease to be a sub-graph of a function.
%To avoid dealing with this degeneracy, we introduce an assumption and a reparameterization of the fluid domain boundary,
%and work with domains satisfying the subgraph property. 
These main tools in our constructive existence proof are outlined next.

\subsection{The Lie operator splitting scheme}
Let $A$ be an operator defined on a Hilbert space, such that $A$ can be written as a non-trivial sum $A=A_1+A_2.$
Consider the following  initial-value problem:
\begin{align*}
\dis\frac{d\phi}{dt}+A\phi &= 0\quad\text{in}\quad(0,T),\\
\phi(0) &= \phi_0.
\end{align*}
The time discretization using Lie operator splitting is defined by solving on each time interval $(t_n,t_{n+1})$, defined by
the discretization step $ \dt=T/N$, where $N\in\N$, and  $t_n=n\dt, n=0,\dots,N-1$,
the following two subproblems for $i = 1$ and $i = 2$:
\begin{align*}
\dis\frac{d\phi_i}{dt} + A_i\phi_i &= 0\quad\text{in}\quad(t_n,t_{n+1}),\\
\phi_i(t_n)&=\phi^{n+\frac{i-1}{2}},
\end{align*}
and then set $\phi^{n+\frac{i}{2}}=\phi_i(t_{n+1}).$
Thus, the two problems communicate only via the initial data, mimicking the famous Lie-Trotter product formula for exponentials
as solutions to $\phi' = -A \phi = -(A_1+A_2) \phi$.

To apply this strategy to our coupled FSI problem, we rewrite the problem as a first-order system in time by 
introducing three new variables corresponding to structure velocities:
the Koiter shell velocity $\bv = \partial_t\be$, 
the mesh velocity $\bk = \partial_t \bd$,
and the mesh angular velocity $\bz = \partial_t \bw.$

The coupled problem \eqref{FSIeq1}-\eqref{FSIeqIC} is split into two subproblems, a fluid and a structure subproblem, see Sec.~\ref{splitting},
and in each time step 
 the structure subproblem is first solved on $(t_n,t_{n+1})$ with  the solution of the fluid subproblem from the previous time step serving as the initial data.
Then
 the fluid subproblem  is solved on $(t_n,t_{n+1})$ with the solution of the just calculated structure subproblem serving as the initial data.

In the structure subproblem, the structure is driven by the initial velocity obtained from the trace of the fluid velocity in the previous time step. The fluid velocity $\bu$ remains unchanged.
In the fluid subproblem, the Navier-Stokes equations are driven by a 
"Robin-type" boundary condition on $\omega$ (i.e., $\Gamma)$ which involves the shell inertia and the trace of the fluid normal stress. 
In this step, the structure displacement, the velocity of the mesh displacement and the velocity of infinitesimal rotation of cross sections remain unchanged. 

The inclusion of shell inertia into the fluid subproblem 
guarantees energy balance at the time-discrete level, thereby avoiding stability problems due to the so called added mass effect.  Here we emphasize
that there is no added mass effect associated with the mesh since the fluid
velocity does not have the trace defined on the 1D set describing the mesh,
and therefore it is enough to include only the shell inertia into the fluid
subproblem. The shell inertia is affected by the presence of the mesh, accounted for in the structure subproblem.

Before we can apply the Lie splitting method to our problem,  we first need to explain how to deal with the difficulties 
associated with  the motion of the fluid domain boundary.
One difficulty is related to the possible geometric degeneracy of the fluid domain boundary, and the other to the fact that at every time step the fluid subproblem
is defined on a different domain.
The following two subsections deal with these two issues.

\subsection{Reparameterization of the fluid domain boundary as a subgraph}
First, recall that the lateral boundary of the fluid domain, which coincides with the fluid-structure interface, is given by:
\begin{equation*}
\Gamma^{\be}(t)=\{(z+\eta_z(t,z,\theta),R+\eta_r(t,z,\theta),\theta+\eta_\theta(t,z,\theta)):z\in (0,L),\theta\in(0,2\pi) \}.
\end{equation*}
Due to the fact that the structure/shell is moving in all three spatial directions, see Assumption~\ref{3D_displ},
 the fluid domain boundary may degenerate in such a way that it ceases to be a subgraph of a function. To avoid such a degeneracy, we introduce the following:
\begin{assumption}\label{subgraph}
There exists a time $T>0$ such that for every $t\leq T, \Gamma^{\be}(t)$ remains a subgraph of a function. 
\end{assumption}
Under this assumption,  the lateral boundary of the fluid domain can be reparameterized 
in such a way that the radial displacement becomes the only unknown. 
This will be useful in explicitly writing the Arbitrary Lagrangian-Eulerian (ALE) mapping, discussed in the next section.
We remark that this is not a necessary condition under which the result of this manuscript holds. 
Assumption~\ref{subgraph} simplifies the proof as it is used in the explicit construction of the ALE mapping. 

More precisely, introduce 
\begin{align}\label{reparam}
\begin{split}
\zt=z+\eta_z(t,z,\theta),\\
\et(t,\zt,\tht) = \eta_r(t,z,\theta),\\
\tht = \theta + \eta_\theta(t,z,\theta),
\end{split}
\end{align}
and define the reparameterized lateral boundary of the fluid domain to be
\begin{equation}\label{reparam_bdry}
\Gamma^{\bet}(t)=\{(\zt,R+\et(t,\zt,\tht),\tht):\zt\in (0,L),\tht\in(0,2\pi) \}.
\end{equation}
\noindent
It is easy to check that displacement $\et$ satisfies the following:
\begin{equation*}
\et(t,\zt,\tht) = \eta_r\circ (\text{id}_z+\eta_z,\text{id}_\theta+\eta_\theta)^{-1}(t,\zt,\tht),
\end{equation*}
where $\text{id}_z$ and $\text{id}_\theta$ are projections of the identity to the second and third variable. 

With this reparameterization the shell displacement is given by the function 
\begin{equation}\label{eta_repara}
\bet=\et\mathbf{e}_r,
\end{equation}
 where $\mathbf{e}_r=\mathbf{e}_r(\theta)=(0,\cos\theta,\sin\theta)$ is the unit vector in the radial direction. 
Notice that the reparameterization \eqref{reparam_bdry} is well-defined if for every $t$ the following mapping:
\begin{equation}\label{inject}
\bg:(0,T)\times\omega\to\R^2, \quad \bg(t,z,\theta) =(z+\eta_z(t,z,\theta),\theta + \eta_\theta(t,z,\theta))
\end{equation}
is an injection from $\omega$ to $\R^2$.
We will see later that this will, indeed, be true for our FSI problem as a consequence of 
Assumption~\ref{LipschitzAssumption} on approximate structure displacements, introduced in Sec.~\ref{sec:StrongConvergence}. 
%Our existence result rests on the Assumption~\ref{LipschitzAssumption}.

\subsection{An Arbitrary Lagrangian-Eulerian (ALE) mapping}\label{sec:ale}
Since the fluid domain moves in time, 
 at each time step $t_n=n\dt,\;n=0,\dots,N-1,$ the fluid subproblem has to be solved on a different fluid domain, defined by
\begin{equation}\label{phi}
\Omega^n=\bphi^n(\Omega):=\bphi^{\bet}(n\dt,\Omega).
\end{equation} 
This presents a significant difficulty when studying existence of solutions to this class of problems. 
One way to deal with this difficulty is to map the current fluid domain onto a fixed, reference domain $\Omega$,
and work with the entire problem
reformulated on $\Omega$,
as was done in our earlier works, see e.g., \cite{MuhaCanic13,BorSun3d,BorSunMultiLayered}. This, however, introduces additional nonlinearities in the problem, especially in the definition of differential operators such as divergence. 
Moreover, the mapped velocity is not divergence-free, but rather it is divergence-free only in terms of the mapped, nonlinear gradient operator. 
Proving Korn's inequality for the transformed divergence operator is nontrivial, and obtaining a compactness result for the transformed
problem is even more complicated. Therefore, often times the problem needs to be mapped back onto the physical 
domain to deal with those issues. 

Another way to deal with the motion of the fluid domain is to adopt the strategy which has been used in
numerical ALE-based algorithms since the early 1980's \cite{donea1983arbitrary,hughes1981lagrangian},
but never in existence proofs. 
In this approach the fluid domain is updated at every time step, 
the problem is studied in the physical domain, the gradient operator is the 
gradient operator in the physical space, and the time derivative of fluid velocity 
is calculated using the ALE derivative, as shown below.
Because the problem is posed in the physical space,
Korn's inequality can be used in the standard way, and compactness arguments are easier to construct. 
%The trick is to define the problem on the updated domain using the ALE mapping so that the 
%geometric conservation law is satisfied. This will lead to the desired energy estimates, which are crucial 
%for the existence proof. 
Thus, we introduce the following mapping:
 \begin{equation*}
\boldsymbol{\mathcal{A}}^{n+1,n}:\Omega^{n+1}\to \Omega^{n},\quad
\boldsymbol{\mathcal{A}}^{n+1,n}=\bphi^{n}(\cdot)\circ \bphi^{n+1}(\cdot)^{-1},
\end{equation*}
which is explicitly written in terms of the location of the lateral boundary at times $n$ and $n+1$ as follows:
\begin{equation}\label{ale_disc}
\boldsymbol{\mathcal{A}}^{n+1,n}
%\begin{pmatrix}
(\zt, \rt, \tht)
%\end{pmatrix}
 =
%\begin{pmatrix}
\left(\zt,
\frac{R+\et^{n}(\zt,\tht)}{R+\et^{n+1}(\zt,\tht)}\rt,
\tht \right),
%\end{pmatrix},
\end{equation}
where $(\zt,\rt,\tht)$ denote the cylindrical coordinates in the reference domain $\Omega^{n+1}.$ 
This is a discrete analogue of the mapping $\bAc$ defined by:
\begin{equation*}
\bAc(t):\Omega^{n+1}\to \Omega^{\bet}(t), \quad \bAc(t)=\bphi^{\bet}(t,\cdot)\circ\bphi^{n+1}(\cdot)^{-1}, \ \forall t\in[t^n,t^{n+1}),
\end{equation*}
which can be explicitly written,
using the reparameterization of the fluid domain boundary described in the previous subsection,  by the following explicit formula:
\begin{align}\label{ale_cont}
\bAc(t):\Omega^{n+1}\to \Omega^{\bet}(t),\quad
\bAc(t)
%\begin{pmatrix}
(\zt, \rt, \tht)
%\end{pmatrix} 
=
%\begin{pmatrix}
\left(\zt,
\frac{R+\et(t,\zt,\tht)}{R+\et^{n+1}(\zt,\tht)}\rt,
\tht \right),
%\end{pmatrix},
\end{align} 
where $(\zt,\rt,\tht)$ denote the cylindrical coordinates in the discrete physical domain $\Omega^{n+1}.$ Due to the fact that we are working with the Navier-Stokes equations written in Cartesian coordinates, it is useful to write an explicit form of the ALE mapping $\bAc(t)$ in the Cartesian coordinates as well:
\begin{equation*}
\bAc(t)
%\begin{pmatrix}
(z, x , y)
%\end{pmatrix} 
=
%\begin{pmatrix}
\left(z,
\frac{R+\et(t,\zt,\tht)}{R+\et^{n+1}(\zt,\tht)}x,
\frac{R+\et(t,\zt,\tht)}{R+\et^{n+1}(\zt,\tht)}y
\right).
%\end{pmatrix}.
\end{equation*}
Mapping $\bAc(t)$ is a bijection, and its Jacobian $\bSc$ and the ALE domain velocity $\bsc$ are respectively given by:
\begin{align}
\begin{split}
\bSc(t)&=|\det \nabla \bAc (t)|=\left| \frac{R+\et(t,\zt,\tht)}{R+\et^{n+1}(\zt,\tht)}  \right|^2,\\
\bsc &= \partial_t \bAc(t) = \frac{\partial_t\et(t,\zt,\tht)}{R+\et^{n+1}(\zt,\tht)}\rt \mathbf{e}_r.
\end{split}
\end{align}
We define the ALE derivative as the time derivative evaluated on the domain $\Omega^{n+1}:$
\begin{equation*}
\partial_t\bu|_{\Omega^{n+1}}=\partial_t \bu + (\bsc\cdot\nabla)\bu.
\end{equation*}
Using the ALE mapping, we can rewrite the Navier-Stokes equations in the ALE formulation as follows:
\begin{equation*}
\partial_t\bu|_{\Omega^{n+1}}+((\bu-\bsc)\cdot\nabla)\bu=\nabla\cdot\bsigma.
\end{equation*}

\noindent
Discrete versions of the Jacobian and the ALE domain velocity are given by:
\begin{equation*}
\bS=|\det \nabla \bA|=\left|\frac{R+\et^{n}(\zt,\tht)}{R+\et^{n+1}(\zt,\tht)}\right|^2,
\end{equation*}

\begin{equation}\label{sn}
\bs=\frac{\et^{n+1}(\zt,\tht)-\et^{n}(\zt,\tht)}{\dt(R+\et^{n+1}(\zt,\tht))}\rt\mathbf{e}_r.      
\end{equation}
Composite functions with this ALE mapping will be denoted by 
\begin{equation}\label{ale_comp_fun}
\buh^n=\bu^n\circ\bA.
\end{equation}

\section{Approximate solutions}\label{splitting}
We use the Backward Euler scheme to discretize the problem in time,
and Lie splitting to separate the fluid from a structure subproblem. 
Let $\Delta t=T/N$ be the time discretization parameter so that the time interval $(0,T)$ is subdivided into $N$ subintervals of width $\Delta t$. 
For $n=0,1,\dots,N-1$, define the vector of unknown approximate solutions 
$$\mathbf{X}_N^{n+i/2}=(\bu_N^{n+i/2},\be_N^{n+i/2},\bv_N^{n+i/2},\bd_N^{n+i/2},\bw_N^{n+i/2},\bk_N^{n+i/2},\bz_N^{n+1/2}),$$
where $i=1,2$ denotes the solution of the structure and of the fluid subproblem, respectively. 

We aim at performing the time discretization via Lie operator splitting
in such a way that the discrete version of the energy inequality \eqref{energy} is preserved at every time step.
We define the semi-discrete versions of the kinetic and elastic energy, and of dissipation, by the following:
\begin{align}\label{disc_kinetic}
\begin{split}
E_N^{n+i/2}&=\rho_F\int_{\Omega^{n+1}}|\bu_N^{n+1}|^2+\rho_K h\int_{\omega}|\bv_N^{n+i/2}|^2R + a_K(\be_N^{n+i/2},\be_N^{n+i/2})\\
&\quad+ \rho_S \sum_{i=1}^{n_E}A_i\int_0^{l_i}|(\bk_N^{n+i/2})_i|^2 + \rho_S \sum_{i=1}^{n_E}\int_0^{l_i}M_i|(\bz_N^{n+i/2})_i|^2\\
&\quad+ a_S(\bw_N^{n+i/2},\bw_N^{n+i/2}),\quad i=1,2,
\end{split}
\end{align}
\begin{align}\label{disc_dissip}
D_N^{n+1}=2\dt\mu_F\int_{\Omega^{n+1}}|\bD(\bu_N^{n+1})|^2, \quad n=0,1,\dots,N-1.
\end{align}
Throughout the rest of this section, we fix the time step $\dt,$ i.e. we keep $N\in\N$ fixed, and study the semi-discretized subproblems defined by the Lie splitting. To simplify notation, we omit the subscript $N$ and write $(\bu^{n+i/2},\be^{n+i/2},\bv^{n+i/2},\bd^{n+i/2},\bw^{n+i/2},\bk^{n+i/2},\bz^{n+1/2}),$ instead of\\ $(\bu_N^{n+i/2},\be_N^{n+i/2},\bv_N^{n+i/2},\bd_N^{n+i/2},\bw_N^{n+i/2},\bk_N^{n+i/2},\bz_N^{n+1/2}).$

\subsection{The semi-discretized structure subproblem}
%We write a semi-discrete version of the composite structure subproblem defined by the Lie splitting. 
In this step the fluid velocity $\bu$ does not change, so
$$\bu^{n+1/2} = \bu^n. $$
The structure variables $(\be^{n+1/2},\bv^{n+1/2},\bd^{n+1/2},\bw^{n+1/2},\bk^{n+1/2},\bz^{n+1/2})\in W_S$
are calculated as the solution of the following problem, written in weak form:
\begin{align}\label{sub_struct}
&\dis \rho_K h \int_{\omega}\frac{\bv^{n+1/2}-\bv^n}{\dt}\cdot\bpsi R + a_K(\be^{n+1/2},\bpsi)
+\dis \rho_S\sum_{i=1}^{n_E}A_i\int_{0}^{l_i}\frac{\bk_i^{n+1/2}-\bk_i^n}{\dt}\cdot\bxi_i \nonumber\\
&\ \ \ \ \ \ \ \ +\dis \rho_S\sum_{i=1}^{n_E}\int_{0}^{l_i}M_i\frac{\bz_i^{n+1/2}-\bz_i^n}{\dt}\cdot\bzeta_i
+a_S(\bw^{n+1/2},\bzeta)= 0,\nonumber\\
&\ \ \ \ \dis \int_\omega \frac{\be^{n+1/2}-\be^n}{\dt} \cdot \bpsi R = \int_\omega \bv^{n+1/2}\cdot \bpsi R ,\\
&\left.
\begin{array}{rcl}
\dis \int_0^{l_i} \frac{\bd_i^{n+1/2}-\bd_i^n}{\dt} \cdot \bxi_i &=&\dis \int_0^{l_i} \bk_i^{n+1/2}\cdot \bxi_i,\nonumber\\
\dis \int_0^{l_i} \frac{\bw_i^{n+1/2}-\bw_i^n}{\dt} \cdot \bzeta_i &=&\dis \int_0^{l_i}  \bz_i^{n+1/2}\cdot \bzeta_i,
\end{array}
\right\} \  i = 1,\dots,n_E,	   
\nonumber
\end{align}
for all the test functions $(\bpsi,\bxi,\bzeta)\in V_{KS},$
where the solution space is defined by:
\begin{equation*}
W_S:= \{ (\be,\bv,\bd,\bw,\bk,\bz) \in V_K\times H^2(\omega)\times V_S\times H^1(\mathcal{N}) \times H^1(\mathcal{N}):
\be\circ\bpi=\bd\text{ on }\prod_{i=1}^{n_E}(0,l_i)\}.
\end{equation*}

\begin{prop}
	For each fixed $\Delta t > 0$, the structure subproblem has a unique solution\\ $(\be^{n+1/2},\bv^{n+1/2},\bd^{n+1/2},\bw^{n+1/2},\bk^{n+1/2},\bz^{n+1/2})\in W_S.$
\end{prop}
\begin{proof}
	Since the semi-discretized structure subproblem is a linear elliptic problem,  the Lax--Milgram lemma implies the existence of a unique solution.
	 Details of the proof can be found in Proposition 6.1. in \cite{LinearFSIstent}.
\end{proof}

\begin{prop}
	For each fixed $\dt>0,$ the structure subproblem \eqref{sub_struct} satisfies the following discrete energy equality:
	\begin{align}\label{energy_struct}
	\begin{split}
	E_N^{n+1/2}
	&+\rho_Kh\|\bv^{n+1/2}-\bv^n\|^2_{L^2(R;\omega)}+a_K(\be^{n+1/2}-\be^n,\be^{n+1/2}-\be^n)
	+ \rho_S\|\bk^{n+1/2}-\bk^n\|_a^2\\
	&+ \rho_S \|\bz^{n+1/2}-\bz^n\|_m^2
	+ a_S(\bw^{n+1/2}-\bw^n,\bw^{n+1/2}-\bw^n)=E_N^n,
	\end{split}
	\end{align}
	where
	$$
	\|\bk\|_a^2:=\sum_{i=1}^{n_E}A_i\|\bk_i\|_{L^2(0,l_i)}.
	$$
\end{prop}

\begin{proof}
	The proof is similar to the corresponding proof in \cite{LinearFSIstent}.
\end{proof}

\subsection{The semi-discretized fluid subproblem}

In this step the structure variables remain unchanged:
$$
\bet^{n+1}=\bet^{n+1/2},\quad \bd^{n+1}=\bd^{n+1/2}, \quad \bw^{n+1}=\bw^{n+1/2},\quad \bk^{n+1}=\bk^{n+1/2},\quad \bz^{n+1} = \bz^{n+1/2}.  
$$
Recall that $\bet^{n+1}$ is the reparameterized structure displacement.

The fluid and shell velocities $(\bu^{n+1},\bv^{n+1})\in W_F$ are updated by 
solving the semi-discretized fluid subproblem in ALE formulation, defined on the just updated domain  $\Omega^{n+1}$, written in weak form:
\begin{align}\label{sub_fluid}
\begin{split}
\dis &\rho_F \int_{\Omega^{n+1}}  \frac{\bu^{n+1}-\buh^{n}}{\dt}\cdot \bups + \dis \frac{\rho_F}{2}\int_{\Omega^{n+1}}(\nabla\cdot \bs)(\buh^n\cdot \bups)\\
&+\frac{\rho_F}{2}\int_{\Omega^{n+1}}\left[((\buh^n-\bs)\cdot\nabla)\bu^{n+1}\cdot \bups - ((\buh^n-\bs)\cdot\nabla)\bups\cdot \bu^{n+1}\right]\\
&+2\mu_F \int_{\Omega^{n+1}}\bD(\bu^{n+1}):\bD(\bups) 
+\dis \rho_K h\int_{\omega}\frac{\bv^{n+1}-\bv^{n+1/2}}{\dt}\cdot\bpsi R \\  
&= \dis P_{in}^n\int_{\Gamma_{in}}\upsilon_z-  P_{out}^n\int_{\Gamma_{out}}\upsilon_z,\quad \forall (\bups,\bpsi)\in W_F,
\end{split}
\end{align}
where the weak solution space is defined by:
$$W_F=\{ (\bu,\bv)\in V_F^{n+1}\times L^2(R;\omega): (\bu\circ \bphi^{n+1})|_{\Gamma}\circ \bvarphi = \bv \text{ on }\omega\},$$
with
$$V_F^{n+1}=\{\bu\in H^1(\Omega^{n+1}): \nabla\cdot\bu=0,\bu\times\mathbf{e}_z=0\text{ on }\Gamma_{in/out}\}.$$
The pressure terms are given by  $P_{in/out}^n=\dis\frac{1}{\dt}\int_{n\dt}^{(n+1)\dt}P_{in/out}(t)\;dt.$

\begin{prop} For each fixed $\dt>0,$ the fluid subproblem \eqref{sub_fluid} has a unique solution $(\bu^{n+1},\bv^{n+1})\in W_F.$
\end{prop}

\begin{proof}
	We rewrite the first equation in \eqref{sub_fluid} as follows:
	\begin{align*}
	&\dis\frac{\rho_F}{\dt}\int_{\Omega^{n+1}}\bu^{n+1}\cdot\bups 
	+\frac{\rho_F}{2}\int_{\Omega^{n+1}}\big[((\buh^n-\bs)\cdot\nabla)\bu^{n+1}\cdot \bups
	- ((\buh^n-\bs)\cdot\nabla)\bups\cdot \bu^{n+1} \big]\\
	& + 2\mu_F \int_{\Omega^{n+1}}\bD(\bu^{n+1}):\bD(\bups) + \frac{\rho_K h}{\dt}\int_{\omega}\bv^{n+1}\cdot\bpsi R\\
	&=\dis\frac{\rho_F}{\dt}\int_{\Omega^{n+1}}\buh^{n}\cdot\bups  +
	\frac{\rho_F}{2}(\nabla\cdot\bs)(\buh^n\cdot\bups)+
	\frac{\rho_K h}{\dt}\int_{\omega}\bv^{n+1/2}\cdot\bpsi R
	+ P_{in}^n\int_{\Gamma_{in}} \upsilon_z -P_{out}^n \int_{\Gamma_{out}}\upsilon_z,
	\end{align*}
	and define the bilinear form associated with problem \eqref{sub_fluid}:
	\begin{align*}
	a((\bu,\bv),(\bups,\bpsi))&= \rho_F\int_{\Omega^{n+1}}\bu\cdot\bups +\frac{\rho_F\dt}{2}\int_{\Omega^{n+1}}\big[((\buh^n-\bs)\cdot\nabla)\bu\cdot \bups\\
	&\quad - ((\buh^n-\bs)\cdot\nabla)\bups\cdot \bu \big]
	+ 2\dt \mu_F\int_{\Omega^{n+1}}\bD(\bu):\bD(\bups)
	+ \rho_Kh \int_{\omega}\bv\cdot\bpsi R.
	\end{align*}
	We need to prove that this bilinear form $a$ is coercive and continuous on $W_F.$ To see that $a$ is coercive, we write
	\begin{align*}
	a((\bu,\bv),(\bu,\bv))&=\rho_F\int_{\Omega^{n+1}}|\bu|^2 + 2\dt\mu_F \int_{\Omega^{n+1}}|\bD(\bu)|^2 + \rho_Kh \int_{\omega}|\bv|^2 R\\ 
	&\geq c \left( \|\bu\|_{L^2(\Omega^{n+1})}^2 + \|\bD(\bu)\|_{L^2(\Omega^{n+1})}^2 + \|\bv\|_{L^2(R;\omega)}^2 \right)\\
	&\geq c \left( \|\bu\|_{H^1(\Omega^{n+1})}^2 + \|\bv\|_{L^2(R;\omega)}^2 \right).
	\end{align*}
	To prove continuity, we apply the generalized H\"older's inequality to obtain:
	\begin{align*}
	&a((\bu,\bv),(\bups,\bpsi))\leq \rho_F\|\bu\|_{L^2(\Omega^{n+1})}\|\bups\|_{L^2(\Omega^{n+1})} + 2\dt\mu_F \|\bD(\bu)\|_{L^2(\Omega^{n+1})}\|\bD(\bups)\|_{L^2(\Omega^{n+1})}\\
	&\quad + \rho_K h \|\bv\|_{L^2(R;\omega)}\|\bpsi\|_{L^2(R;\omega)}+
	\rho_F\dt\|\nabla\bu\|_{L^2(\Omega^{n+1})}\|\bu\|_{L^4(\Omega^{n+1})}\|\bups\|_{L^4(\Omega^{n+1})}.
	\end{align*}
	Using the continuous embedding of $H^1$ into $L^6$, and  the continuous embedding of $L^p$ into $L^q$, for $q<p,$ we obtain:
	\begin{align*}
	a((\bu,\bv),(\bups,\bpsi))&\leq C\left(\rho_F \|\bu\|_{H^1(\Omega^{n+1})}\|\bups\|_{H^1(\Omega^{n+1})} + 2\dt\mu_F \|\bu\|_{H^1(\Omega^{n+1})}\|\bups\|_{H^1(\Omega^{n+1})}\right.\\
	&\quad+\left. \rho_K h \|\bv\|_{L^2(R;\omega)}\|\bpsi\|_{L^2(R;\omega)}+
	\rho_F\dt\|\bu\|_{H^1(\Omega^{n+1})}\|\bu\|_{H^1(\Omega^{n+1})}\|\bups\|_{H^1(\Omega^{n+1})}\right).
	\end{align*}
	This shows that $a$ is continuous. The Lax-Milgram lemma now implies the existence of a unique solution $(\bu^{n+1},\bv^{n+1})$ of problem \eqref{sub_fluid}.	
\end{proof}

\begin{prop}
	For each fixed $\dt>0,$ the solution of problem \eqref{sub_fluid} satisfies the following discrete energy inequality:
	\begin{align}\label{energy_fluid}
	\begin{split}
	&E_N^{n+1}+\rho_F\|\bu^{n+1}-(1-\Delta\eta^{n+1,n})\buh^n\|_{L^2(\Omega^{n+1})}^2+\rho_K h \|\bv^{n+1}-\bv^{n}\|_{L^2(R;\omega)}^2 \\
	&+ D_N^{n+1} \leq E_N^{n+1/2} + C\dt ((P_{in}^n)^2 + (P_{out}^n)^2),
	\end{split}
	\end{align}
	where $\Delta\eta^{n+1,n}$ is a constant that will be specified in the proof.
\end{prop}

\begin{proof}
	We begin by replacing the test function $(\bups,\bpsi)$ by $(\bu^{n+1},\bv^{n+1})$ in the weak formulation \eqref{sub_fluid} to obtain:
	\begin{align*}
	&\dis\frac{\rho_F}{\dt}\int_{\Omega^{n+1}}(\bu^{n+1}-\buh^n)\cdot\bu^{n+1} + \frac{\rho_F}{2}\int_{\Omega^{n+1}}(\nabla\cdot \bs)(\buh^n \cdot \bu^{n+1})
	+ 2\mu_F \int_{\Omega^{n+1}}\bD(\bu^{n+1}):\bD(\bu^{n+1})\\
	&+\dis\frac{\rho_Kh}{\dt}\int_{\omega}(\bv^{n+1}-\bv^{n+1/2})\cdot\bv^{n+1}R
	=P_{in}^n\int_{\Gamma_{in}} u_z^{n+1}-P_{out}^n\int_{\Gamma_{out}} u_z^{n+1}.
	\end{align*}
	After applying algebraic identity $(a-b)\cdot a=\frac{1}{2}(|a|^2+|a-b|^2-|b|^2)$ to take care of the terms involving mixed products,
	 and multiplying the resulting equation by $2\dt,$ we get:
	\begin{align*}
	&\rho_F \int_{\Omega^{n+1}}|\bu^{n+1}|^2 + \rho_F \int_{\Omega^{n+1}} |\bu^{n+1} - \buh^n|^2 - \rho_F\int_{\Omega^{n+1}}|\buh^n|^2
	+\rho_F \dt \int_{\Omega^{n+1}} (\nabla\cdot \bs)(\buh^n\cdot\bu^{n+1})\\
	&+ 4\dt \mu_F \int_{\Omega^{n+1}}\bD(\bu^{n+1}):\bD(\bu^{n+1})
	+\rho_K h \int_{\omega}|\bv^{n+1}|^2 R + \rho_K h\int_{\omega}|\bv^{n+1}-\bv^n|^2 R \\
	&-\rho_K h\int_{\omega}|\bv^{n}|^2 R
	=P_{in}^n\int_{\Gamma_{in}} u_z^{n+1}-P_{out}^n\int_{\Gamma_{out}}u_z^{n+1}.
	\end{align*}
	To get the desired energy inequality we add and subtract the term $\rho_F\int_{\Omega^n}|\bu^n|^2$ on the left-hand side of the equality,
	and convert one of those terms into an integral over the current domain $\Omega^{n+1}$ by recalling the ALE Jacobian
	$S^{n+1,n}=|\det\nabla A^{n+1,n}|$ to obtain:
	\begin{equation*}
	\rho_F\int_{\Omega^n}|\bu^n|^2 = \rho_F\int_{\Omega^{n+1}}S^{n+1,n}|\buh^n|^2.
	\end{equation*} 
	Furthermore, by using the formula for divergence in cylindrical coordinates 
	$\dis\nabla\cdot f = \frac{\partial f_z}{\partial z}+\frac{1}{r}\frac{\partial(rf_r)}{\partial_r}+\frac{1}{r}\frac{\partial f_\theta}{\partial \theta},$ 
	we calculate:
	\begin{equation}
	\nabla\cdot \bs = 2 \frac{\et^{n+1}(\zt,\tht)-\et^{n}(\zt,\tht)}{\dt(R+\et^{n+1}(\zt,\tht))} 
	= 2 \frac{\Delta \eta^{n+1,n}}{\Delta t},
	\end{equation}
	where we used 
	$\Delta\eta^{n+1,n}$ to denote the difference between the current and previous location of the lateral boundary,
	normalized by the current lateral boundary location:
	$\dis \Delta\eta^{n+1,n}:=\frac{\et^{n+1}(\zt,\tht)-\et^{n}(\zt,\tht)} {R+\et^{n+1}(\zt,\tht)}.$
	
	By combining these manipulations, the fluid kinetic energy and numerical dissipation terms in the energy equality can be rewritten in the following way:
	\begin{align*}
	&\int_{\Omega^{n+1}}|\bu^{n+1}|^2 + \int_{\Omega^{n+1}} |\bu^{n+1} - \buh^{n}|^2 - \int_{\Omega^{n+1}}|\buh^n|^2
	+ \int_{\Omega^{n+1}}\dt (\nabla\cdot \bs)(\buh^n\cdot\bu^{n+1})
	\pm \int_{\Omega^n}|\bu^n|^2\\
	&= \int_{\Omega^{n+1}}|\bu^{n+1}|^2 + \int_{\Omega^{n+1}}|\bu^{n+1}|^2 - \int_{\Omega^{n+1}}2\bu^{n+1}\cdot \buh^{n} + \int_{\Omega^{n+1}}|\buh^n|^2 
	- \int_{\Omega^{n+1}}|\buh^n|^2 \\
	&\quad + \int_{\Omega^{n+1}}2\Delta\eta^{n+1,n} \buh^n\cdot\bu^{n+1} + \int_{\Omega^{n+1}}S^{n+1,n}|\buh^n|^2 - \int_{\Omega^n}|\bu^n|^2\\
	&= \int_{\Omega^{n+1}}|\bu^{n+1}|^2 - \int_{\Omega^n}|\bu^n|^2 + \int_{\Omega^{n+1}}|\bu^{n+1}|^2 
	- \int_{\Omega^{n+1}}2(1-\Delta\eta^{n+1,n})\bu^{n+1}\cdot \buh^n+\int_{\Omega^{n+1}}S^{n+1,n}|\buh^n|^2\\
	&= \int_{\Omega^{n+1}}|\bu^{n+1}|^2 - \int_{\Omega^n}|\bu^n|^2 + \int_{\Omega^{n+1}}|\bu^{n+1}-(1-\Delta\eta^{n+1,n})\buh^n|^2 
	- \int_{\Omega^{n+1}}|(1- \Delta\eta^{n+1,n})\buh^{n}|^2\\
	&\quad +\int_{\Omega^{n+1}}S^{n+1,n}|\buh^n|^2.
	\end{align*}
	The first two terms on the right-hand side correspond to the discrete kinetic energy at time $t_{n+1}$ and $t_n$, respectively, while the third term corresponds to numerical dissipation. By a simple calculation we obtain that $S^{n+1,n}-(1-\Delta\eta^{n+1,n})^2=0$ so the last two terms cancel out. 
	After estimating the pressure terms, we obtain the following energy inequality:
	\begin{align*}
	&\rho_F\left(\|\bu^{n+1}\|_{L^2(\Omega^{n+1})}^2 + \|\bu^{n+1}-(1-\Delta\eta^{n+1,n})\buh^n\|_{L^2(\Omega^{n+1})}^2\right) + 2\dt\mu_F \|\bD(\bu^{n+1})\|_{L^2(\Omega^{n+1})}^2 \\
	&\quad + \rho_K h \|\bv^{n+1}\|_{L^2(R;\omega)}^2  + \rho_K h \|\bv^{n+1}-\bv^{n+1/2}\|_{L^2(R;\omega)}^2\\
	&\leq \rho_F\|\bu^n\|_{L^2(\Omega^n)} + \rho_K h \|\bv^{n+1/2}\|_{L^2(R;\omega)}^2 + C\dt\left((P_{in}^n)^2+(P_{out}^n)^2\right).
	\end{align*}
	
	Therefore, with the help of nonlinear advection and by adding and subtracting the  term $\rho_F\int_{\Omega^n}|\bu^n|^2$
	we were able to show that the fluid kinetic energy and the shell kinetic energy are both decreasing in time, 
	and satisfy an energy estimate, which additionally provides uniform boundedness of the viscous fluid dissipation, and numerical dissipation 
	captured by the terms $ \rho_F\|\bu^{n+1}-(1-\Delta\eta^{n+1,n})\buh^n\|_{L^2(\Omega^{n+1})}^2$ and 
	$\rho_K h \|\bv^{n+1}-\bv^{n+1/2}\|^2$.
		
	Finally, recall that $\be^{n+1}=\be^{n+1/2}$ and $\bw^{n+1}=\bw^{n+1/2}$ in the fluid subproblem, so we can add $a_K(\be^{n+1},\be^{n+1})$ and $a_S(\bw^{n+1},\bw^{n+1})$  on the left-hand side, and 
	$a_K(\be^{n+1/2},\be^{n+1/2})$ and $a_S(\bw^{n+1/2},\\ \bw^{n+1/2})$ on the right-hand side of the previous inequality. Furthermore, since $\bk^{n+1}=\bk^{n+1/2}$ and $\bz^{n+1}=\bz^{n+1/2}$, we add $\|\bk^{n+1}\|_a^2$ and $\|\bz^{n+1}\|_m^2$ on the left-hand side,
	and $\|\bk^{n+1/2}\|_a^2$ and $\|\bz^{n+1/2}\|_m^2$ on the right-hand side, to obtain exactly the energy inequality \eqref{energy_fluid}. 
	
\end{proof}

\section{Uniform energy estimates}
Based on the energy inequalities satisfied by the fluid and structure subproblems derived above, here we show that sequences of approximations, defined by the 
time discretization via Lie operator splitting and parameterized by $N$ (which depends on $\dt$), are uniformly bounded with respect to $\dt$
in energy norm. This is the first step in our program to show that subsequences of those approximate solutions converge to a weak solution of
the coupled problem \eqref{FSIeq1}-\eqref{FSIeqIC}. The uniform energy estimates below hold under the geometric assumptions that
the lateral boundary of the fluid domain is a subgraph of a function, and that the ALE mapping,
defined in Sec.~\ref{sec:ale}, is injective. More precisely, we have the following result.

\begin{theorem}\label{thm:unif}{\bf{(Uniform Energy Estimates)}}
	Let $\dt>0$ and $N=T/\dt.$ Furthermore, let $E_N^{n+1/2},E_N^{n+1}$ and $D_N^{n+1}$ be the kinetic energy and dissipation given by \eqref{disc_kinetic} and \eqref{disc_dissip}, respectively. 
	Suppose that (geometric assumptions):
	\begin{enumerate}
	\item There exists a time $T>0$ such that for every $t\leq T$ the lateral boundary $\Gamma^{\be}(t)$ remains a subgraph of a function;
	\item The function $\bg$ defined in \eqref{inject} is injective.
	\end{enumerate}
Then there exists a constant $K>0$ independent of $\dt$ (and $N$) such that the following holds:
	\begin{enumerate}[leftmargin=*]
		\item {{The Kinetic Energy Estimate:}} $E_N^{n+1/2}\leq K, E_N^{n+1}\leq K,\forall n = 0,1,\dots,N-1$;
		\item {The Fluid Viscous Dissipation Estimate:} $\sum_{n=0}^{N-1}D_N^{n+1}\leq K$;
		\item {The Numerical Dissipation Estimates:} 
		\begin{enumerate}
		        \item $\sum_{n=0}^{N-1}\left(\rho_F \|\bu^{n+1}-(1-\Delta\eta^{n+1,n})\buh^n\|^2 + \rho_K h(\|\bv^{n+1}-\bv^{n+1/2}\|^2+\|\bv^{n+1/2}-\bv^{n}\|^2) \right)\leq~K$;
		        \item $\sum_{n=0}^{N-1}\rho_S\left(\|\bk^{n+1}-\bk^{n}\|_a^2 + \|\bz^{n+1}-\bz^{n}\|_m^2\right)\leq K$;
		\end{enumerate}
		\item {The Elastic Energy Estimates:} 
		\begin{enumerate}
		         \item $\sum_{n=0}^{N-1}a_K(\be^{n+1}-\be^n,\be^{n+1}-\be^{n})\leq K$; 
		        \item $\sum_{n=0}^{N-1}a_S(\bw^{n+1}-\bw^n,\bw^{n+1}-\bw^n)\leq K$.	
		\end{enumerate}
	\end{enumerate}
\end{theorem}

\begin{proof}
	We begin by adding the energy estimates \eqref{energy_struct} and \eqref{energy_fluid} to obtain:
	\begin{align*}
	&E_N^{n+1/2}+\rho_Kh\|\bv^{n+1/2}-\bv^n\|^2+a_K(\be^{n+1/2}-\be^n,\be^{n+1/2}-\be^n)\\
	&+ \rho_S\|\bk^{n+1/2}-\bk^n\|_a^2+ \rho_S \|\bz^{n+1/2}-\bz^n\|_m^2\\
	&+a_S(\bw^{n+1/2}-\bw^n,\bw^{n+1/2}-\bw^n)+E_N^{n+1}+\rho_F\|\bu^{n+1}-(1-\Delta\eta^{n+1,n})\buh^{n}\|^2\\
	&+ D_N^{n+1} + \rho_Kh\|\bv^{n+1}-\bv^{n+1/2}\|^2
	\leq E_N^n+ E_N^{n+1/2}+C\dt((P_{in}^n)^2+(P_{out}^n)^2).
	\end{align*}
	By taking the sum from $n=0$ to $n=N-1$ on both sides we get:
	\begin{align*}
	&E_N^N+\sum_{n=0}^{N-1}D_N^{n+1}+\sum_{n=0}^{N-1} \Big( \rho_F\|\bu^{n+1}-(1-\Delta\eta^{n+1,n})\buh^{n}\|^2 +  \rho_Kh\|\bv^{n+1}-\bv^{n+1/2}\|^2\\
	&+ \rho_Kh\|\bv^{n+1/2}-\bv^n\|^2 + \rho_S\|\bk^{n+1/2}-\bk^n\|_a^2+\rho_S \|\bz^{n+1/2}-\bz^n\|_m^2\\
	&+ a_K(\be^{n+1/2}-\be^n,\be^{n+1/2}-\be^n)+ a_S(\bw^{n+1/2}-\bw^n,\bw^{n+1/2}-\bw^n)\Big)\\
	&\leq E_0 + C\dt\sum_{n=0}^{N-1}\left((P_{in}^n)^2+(P_{out}^n)^2\right).
	\end{align*}
	The term involving the inlet and outlet pressure data can be easily estimated by recalling that 
	on each subinterval $(t_n,t_{n+1})$ the pressure data is approximated by a constant,
	which is equal to the average value of the pressure over that time interval. Therefore, after using H\"older's inequality, we have:
	\begin{align*}
	\dt\sum_{n=0}^{N-1}(P_{in}^n)^2 &= \dt \sum_{n=0}^{N-1}\left( \frac{1}{\dt}\int_{n\dt}^{(n+1)\dt}P_{in}(t)\, dt \right)^2
	=\frac{1}{\dt}\sum_{n=0}^{N-1}\left(\int_{n\dt}^{(n+1)\dt} P_{in}(t)\, dt\right)^2\\
	&\leq \frac{1}{\dt}\sum_{n=0}^{N-1}\int_{n\dt}^{(n+1)\dt} P_{in}^2(t)\, dt
	\int_{n\dt}^{(n+1)\dt} 1 \, dt
	%= 	\sum_{n=0}^{N-1}\int_{n\dt}^{(n+1)\dt} P_{in}^2(t)\, dt 
	=\|P_{in}\|_{L^2(0,T)}^2.
	\end{align*}
	By using this pressure estimate to bound the right-hand side in the above energy estimate, we have obtained all the statements of the Theorem, with the constant $K$ given by $K=E_0 + C\|P_{in/out}\|_{L^2(0,T)}^2.$
\end{proof}

\section{Weak convergence of approximate solutions}
The time discretization via Lie operator splitting
defines a set of discrete approximations in time,
 which we now use to define a sequence of approximate functions on $(0,T).$ 
 Indeed, we define approximate solutions to be the {\bf{piecewise constant functions}} on each subinterval $((n-1)\dt,n\dt]$  such that
 for $n=1,2,\dots,N$:
\begin{equation}\label{approx_sol}
\begin{array}{l}
\bu_N(t,\cdot)=\bu_N^n,\; \be_N(t,\cdot)=\be_N^n,\; \bet_N(t,\cdot)=\bet_N^n,
\bv_N(t,\cdot)=\bv_N^n,\;\bv^*_N(t,\cdot)=\bv_N^{n-1/2},\;\\
\bd_N(t,\cdot)=\bd_N^n\;
\bw_N(t,\cdot)=\bw_N^n,\;\bk_N(t,\cdot)=\bk_N^n,\;\bz_N(t,\cdot)=\bz_N^n,
 \ \forall t\in ((n-1)\dt,n\dt].
 \end{array}
\end{equation}
Notice that we used $\bv^*_N(t,\cdot)=\bv_N^{n-1/2}$ to denote the approximate shell velocity functions determined by the structure subproblem, and $\bv_N(t,\cdot)=\bv_N^n$ to denote the approximate shell velocity functions determined by the  fluid subproblem. 
The staggered use of the no-slip condition in the Lie operator splitting strategy implies that these are not necessarily the same. 
However, later we will show that $\|\bv_N - \bv^*_N \|_{L^2(R;\omega)}\to 0$ as $N \to \infty$, and that they both converge to 
the structure velocity solution of the coupled FSI problem.

Using Theorem~\ref{thm:unif} we now show that these sequences are uniformly bounded in the appropriate solution spaces.
For this purpose we introduce the maximal fluid domain $\Omega_M$, which is a a cylinder of radius $R_{max}$,
and is such that it contains all the fluid domains $\Omega^n$, and, more generally, $\Omega^{\boldsymbol\eta}(t)\subset \Omega_M,\forall t \in [0,T]$. 
The existence of such a maximal domain follows directly from Theorem~\ref{thm:unif}.

We will be working with the velocity functions extended from $\Omega^n$ to $\Omega_M$ by a constant, which is equal to the trace of the fluid velocity on the lateral boundary $\Gamma^{\tilde\eta}(t)$.
More precisely, we define the fluid velocity on $\Omega_M$ by:
\begin{equation}\label{exten}
\bup_N^n(\zt,r,\tht) = 
\left\{
\begin{array}{cl} 
\bu_N^n(\zt,r,\tht), &(\zt,r,\tht)\text{ in }\Omega^n,\\
\bu_N^n(\zt,R+\et^n,\tht), &(\zt,r,\tht)\text{ in }\Omega_M\setminus \Omega^n.
\end{array}\right.
\end{equation}
For each subdivision $N$ of the time interval $(0,T)$, this defines the function $\bup_N(t,\cdot) = \bup_N^n, \forall t \in ((n-1) \Delta t, n \Delta t]$, defined on $(0,T)\times \Omega_M$.
 \begin{prop}
	The sequence $(\bup_N)_{N\in\N}$ is uniformly bounded in $L^\infty(0,T;L^2(\Omega_M)).$
\end{prop}

\begin{proof}
	From the definition of the extended sequence $\bup_N$ we have:
	\begin{align*}
	\|\bup_N(t)\|_{L^2(\Omega_M)}^2&=\sum_{n=1}^N\|\bup_N^n\|_{L^2(\Omega_M)}^2
	=\sum_{n=1}^N \left(\|\bu_N^n\|_{L^2(\Omega^n)}^2 + \|\bu_N^n(\zt,R+\et_N^n,\tht)\|_{L^2(\Omega_M\setminus\Omega^n)}^2\right)\\
	&=\sum_{n=1}^N \left(\|\bu_N^n\|_{L^2(\Omega^n)}^2 + C\|\bv_N^n\|_{L^2(R;\omega)}^2\right),
	\end{align*}
	where $C = R_{max}- (R+\et_N^n)$ is the positive constant. By using Theorem~\ref{thm:unif} (Statement 1), we obtain the desired result.
\end{proof}

\begin{prop}
	The sequence $(\be_N)_{N\in \N}$ is uniformly bounded in $L^\infty(0,T;V_K)$  and the sequence $(\bw_N)_{N\in \N}$ is uniformly bounded in $L^\infty(0,T;H^1(\mathcal{N})).$
\end{prop}

\begin{proof}
	The first statement of Theorem~\ref{thm:unif} states that $E_N^{n+1}\leq K,\; \forall n = 0,\dots,N-1,$ which implies
	\begin{equation*}
	\|\be_N(t)\|_{H^2(\omega)}^2\leq a_K(\be_N(t),\be_N(t))\leq K, \quad \forall t \in [0,T].
	\end{equation*}
	Therefore, 
	\begin{equation*}
	\|\be_N\|_{L^\infty(0,T;V_K)}\leq K.
	\end{equation*}
	The boundedness of the sequence $(\bw_N)_{N\in\N}$ also follows from the first statement of Theorem~\ref{thm:unif}. Namely, we have
	\begin{align*}
	&\|\bw_N(t)\|_{L^2(\mathcal{N})}^2\leq \|\bw_N(t)\|_m^2\leq K,\\
	&\|\partial_s\bw_N(t)\|_{L^2(\mathcal{N})}^2\leq a_S(\bw_N(t),\bw_N(t))\leq K,
	\end{align*}
	which implies the desired bound.
\end{proof}
Notice that from the uniform energy estimates we do not get any bounds on the sequence $(\bd_N)_{N\in\N}.$ However, using the condition of inextensibility and unshearability, together with the regularity of $\bw_N,$ we can easily prove the following result on the boundedness of the sequence $(\bd_N)_{N\in\N}.$

\begin{cor}
	The sequence $(\bd_N)_{N\in \N}$ is uniformly bounded in $L^\infty(0,T;H^1(\mathcal{N})).$
\end{cor}

\noindent
The following uniform bounds for the shell and mesh approximate velocities are a direct consequence of Theorem~\ref{thm:unif}.

\begin{prop}
	The following statements hold:
	\begin{enumerate}[(i)]
		\item $(\bv_N)_{N\in\N}$ is uniformly bounded in $L^\infty(0,T;L^2(R;\omega)),$\\
		$(\bv_N^*)_{N\in\N}$ is uniformly bounded in $L^\infty(0,T;L^2(R;\omega)),$
		\item $(\bk_N)_{N\in\N}$ is uniformly bounded in $L^\infty(0,T;L^2(\mathcal{N})),$\\
		$(\bz_N)_{N\in\N}$ is uniformly bounded in $L^\infty(0,T;L^2(\mathcal{N})).$
	\end{enumerate}
\end{prop}

To pass to the limit in the weak formulations associated with approximate solutions, we need additional regularity in time of the 
sequences $(\be_N)_{N\in\N},(\bd_N)_{N\in\N}$ and $(\bw_N)_{N\in\N}$.
For this purpose, we introduce a slightly different set of approximate functions. Namely, for each fixed $\dt$, define $\bel_N,\bdl_N$ and $\bwl_N$ to be  continuous, {\bf linear functions} on each subinterval $[(n-1)\dt,n\dt],\; n=1,\dots,N$, and such that
\begin{align}\label{linear_approx}
\begin{split}
\bel_N(n\dt,\cdot)&=\be_N(n\dt,\cdot),
\bvl_N(n\dt,\cdot)=\bv_N(n\dt,\cdot),\\
\bdl_N(n\dt,\cdot)&=\bd_N(n\dt,\cdot),
\bwl_N(n\dt,\cdot)=\bw_N(n\dt,\cdot),\\
\bkl_N(n\dt,\cdot)&=\bk_N(n\dt,\cdot),
\bzl_N(n\dt,\cdot)=\bz_N(n\dt,\cdot).
\end{split}
\end{align}
We now observe:
$$\partial_t \bel_N(t)=\dis\frac{\be_N^{n}-\be_N^{n-1}}{\dt} =\dis\frac{\be_N^{n-1/2}-\be_N^{n-1}}{\dt}=\bv_N^{n-1/2},\quad t\in((n-1)\dt,n\dt].$$
Since $\bv^*_N$ is a piecewise constant function, as defined before via $\bv^*_N(t,\cdot)=\bv_N^{n-1/2},$ for $t\in ((n-1)\dt,n\dt],$ we see that
\begin{equation}\label{vstar}
\partial_t \bel_N=\bv^*_N\text{ a.e.  on } (0,T).
\end{equation}
From \eqref{vstar}, and from the uniform boundedness 
of $E_N^{n+i/2}$ provided by Theorem~\ref{thm:unif}, 
we obtain the uniform boundedness of $(\bel_N)_{N\in\N}$ in $W^{1,\infty}(0,T;L^2(R;\omega))$.
Now, since sequences $(\bel_N)_{N\in\N}$ and $(\be_N)_{N\in\N}$ have the same limit (distributional limit is unique), we get that the weak* limit 
of $\be_N$ is, in fact, in  $W^{1,\infty}(0,T;L^2(R;\omega))$. 

Using analogous arguments, one also obtains that the weak* limits of 
$(\bd_N)_{N\in\N}$ and $(\bw_N)_{N\in\N}$ are in $W^{1,\infty}(0,T;L^2(\mathcal{N}))$.
This is because the corresponding velocity approximations are uniformly bounded
in the corresponding norms, as stated in Statement 3(b) of Theorem~\ref{thm:unif}.

From the uniform boundedness of approximate sequences we can now conclude that for each approximate sequence there exists a 
weakly- or a weakly*-convergent subsequence (depending on the function space).
With a slight abuse of notation, we use the same notation to denote the convergent subsequences.
More precisely, the following result holds:
\begin{lemma}\label{lemma:weak_conv}
	There exist subsequences $(\bup_N)_{N\in\N},(\be_N)_{N\in\N},(\bv_N)_{N\in\N},(\bv^*_N)_{N\in\N},\\
	(\bd_N)_{N\in\N},(\bw_N)_{N\in\N},(\bk_N)_{N\in\N},(\bz_N)_{N\in\N},$ and the functions $\bup \in L^\infty(0,T;L^2(\Omega_M)),$\\
	$\be \in L^\infty(0,T;V_K)\cap W^{1,\infty}(0,T;L^2(R;\omega)),
	\bd,\bw \in L^\infty(0,T;H^1(\mathcal{N}))\cap W^{1,\infty}(0,T;L^2(\mathcal{N})),$\\
	$\bv,\bv^*\in L^\infty(0,T;L^2(R;\omega))$, and
	$\bk,\bz\in L^\infty(0,T;L^2(\mathcal{N}))$, such that
	\begin{align}
	\begin{split}
	&\bup_N\rightharpoonup\bup \text{ weakly* in } L^\infty(0,T;L^2(\Omega_M)),\\
	&\be_N \rightharpoonup \be \text{ weakly* in } L^\infty(0,T;V_K),\\
	&\be_N \rightharpoonup \be \text{ weakly* in } W^{1,\infty}(0,T;L^2(R;\omega)),\\
	&\bd_N \rightharpoonup \bd \text{ weakly* in } L^\infty(0,T;H^1(\mathcal{N})),\\
	&\bd_N \rightharpoonup \bd \text{ weakly* in } W^{1,\infty}(0,T;L^2(\mathcal{N})),\\
	&\bw_N \rightharpoonup \bw \text{ weakly* in } L^\infty(0,T;H^1(\mathcal{N})),\\
	&\bw_N \rightharpoonup \bw \text{ weakly* in } W^{1,\infty}(0,T;L^2(\mathcal{N})),\\
	%	\nonumber
	%	\end{split}
	%	\end{align}
	%	\begin{align}
	%	\begin{split}
	&\bv_N \rightharpoonup \bv \text{ weakly* in } L^\infty(0,T;L^2(R;\omega)),\\
	&\bv^*_N \rightharpoonup \bv^* \text{ weakly* in } L^\infty(0,T;L^2(R;\omega)),\\
	&\bk_N \rightharpoonup \bk \text{ weakly* in } L^\infty(0,T;L^2(\mathcal{N})),\\
	&\bz_N \rightharpoonup \bz \text{ weakly* in } L^\infty(0,T;L^2(\mathcal{N})).	
	\nonumber
	\end{split}
	\end{align}
	Furthermore, $$\bv=\bv^*.$$
\end{lemma}
\begin{proof}
	We only need to show that $\bv=\bv^*.$ To show this, we use the definition of approximate sequences as step functions in $t$
	and Statement 3 of Theorem~\ref{thm:unif} to obtain:
	\begin{align*}
	\|\bv_N-\bv^*_N\|_{L^2(0,T;L^2(R;\omega))}^2 &= \int_0^T\|\bv_N-\bv^*_N\|_{L^2(R;\omega)}^2\;dt
	=\sum_{n=0}^{N-1}\int_{t_n}^{t_{n+1}}\|\bv_N^{n+1}-\bv_N^{n+1/2}\|_{L^2(R;\omega)}^2\;dt\\
	&=\sum_{n=0}^{N-1}\|\bv_N^{n+1}-\bv_N^{n+1/2}\|_{L^2(R;\omega)}^2\;\dt \leq K\dt.
	\end{align*}
	By letting $\dt\to 0$, we get that $\bv=\bv^*.$
\end{proof}

\section{Strong convergence of approximate sequences}\label{sec:StrongConvergence}
To show that the limits obtained in the previous section satisfy the weak form of the coupled FSI problem
\eqref{FSIeq1}-\eqref{FSIeqIC}, we need to show that sequences of approximate functions converge strongly in the appropriate function spaces. 

\subsection{Strong convergence of shell displacements}\label{subsec:conv_shell}
We first show that $\be_N\to\be\text{ in }L^{\infty}(0,T;C(\overline{\omega})).$ For this purpose, we start by  investigating 
the convergence of $(\bel_N)_{N\in\N}$.  
Recall that from Lemma~\ref{lemma:weak_conv} we have 
\begin{equation*}
(\bel_N)_{N\in\N}\text{ is bounded in }L^\infty(0,T;V_K)\cap W^{1,\infty}(0,T;L^2(R;\omega)).
\end{equation*}
Convergence of $(\bel_N)_{N\in\N}$ then follows from the classical Aubin-Lions compactness result, stated below for completeness:
\begin{lemma}[Aubin-Lions]\label{AubinLions}
Let $X_0$, $X$ and $X_1$ be three Banach spaces with $X_0\subseteq X \subseteq X_1$. 
Suppose that $X_0$ is compactly embedded in $X$, and that $X$ is continuously embedded in $X_1$. 
For $1\le p, q \le \infty$, let
$W = \{ u \in L^p([0,T]; X_0)  : \partial_t u\in L^q([0,T];X_1) \}$.
\begin{enumerate}
\item If $p < \infty$, then the embedding of $W$ into $L^p([0,T]; X)$ is compact.
\item If $p = \infty$, and $q > 1$, then the embedding of $W$ into $C([0,T]; X)$ is compact.
\end{enumerate}
\end{lemma}
Indeed, by taking $X_0 = H_0^2$, $X = H^{\alpha}, 0 < \alpha < 2$, $X_1 = L^2$, $p,q = \infty$, we obtain
convergence of $\bar\eta_N$ in $L^{\infty}(0,T;H^s(\omega))$, $0 < s<2$.
Because sequences $(\bel_N)_{N\in\N}$ and $(\be_N)_{N\in\N}$ have the same limit, we obtain the following result:
\begin{prop}
$\bel_N\to\be\text{ in }L^{\infty}(0,T;H^s(\omega)),$ for $0<s<2.$
\end{prop}

We can now prove the following result on strong convergence of structure displacements.
\begin{prop}
	$\be_N\to\be\text{ in }L^{\infty}(0,T;C(\overline{\omega})).$
\end{prop}

\begin{proof}
	First, we prove that $\be_N\to\be\text{ in }L^{\infty}(0,T;H^s(\omega)),$ for $0<s<2.$
	That result follows from the continuity in time of $\be,$ and from the fact that $\bel_N\to\be$ in $C([0,T];H^s(\omega)),$ for $0<s<2.$ Namely, we write
	\begin{align*}
	\|\be_N(t)-\be(t)\|_{H^s(\omega)}&=\|\be_N(t)-\be(n\dt)+\be(n\dt)-\be(t)\|_{H^s(\omega)}\\
	&= \|\be_N(n\dt)-\be(n\dt)+\be(n\dt)-\be(t)\|_{H^s(\omega)}\\
	&\leq \|\be_N(n\dt)-\be(n\dt)\|_{H^s(\omega)}+ \|\be(n\dt)-\be(t)\|_{H^s(\omega)}\\
	&=\|\bel_N(n\dt)-\be(n\dt)\|_{H^s(\omega)}+ \|\be(n\dt)-\be(t)\|_{H^s(\omega)}\\
	&<\varepsilon.
	\end{align*}
	Here, we used the fact that for $t\in ((n-1)\dt,n\dt]$ the following holds: $\bel_N(n\dt)=\be_N(n\dt)=\be_N(t).$
	The proof follows by observing that for $s>1$ we have $H^s(\omega)\hookrightarrow C(\overline{\omega})$.
\end{proof}

The spatial regularity of the sequence $(\be_N)_{N\in\N}$ obtained in the previous proposition will not be sufficient to pass to the limit.
Since we are working in 3D, the result above does not guarantee the uniform Lipschitz property of the structure displacements, 
which is used at several places to obtain the final existence result. 
This is why we need to make an additional assumption, related to the bi-Lipschitz property of structure displacements:
\vskip 0.1in
\noindent
\begin{assumption}[Regularity of approximate structure displacements]\label{LipschitzAssumption}
There exists a constant $C >0$, independent of $N$, such that the structure displacements $(\be_N)_{N\in\N}$ satisfy:
\begin{equation}\label{Lipschitz}
\|\be_N\|_{C([0,T];W^{1,\infty}(\omega))}\leq C.
\end{equation}
\end{assumption}
\vskip 0.1in

\begin{remark} This assumption is not necessary for structures with higher regularity, such as those studied in \cite{Boulakia2005} and \cite{multipolar}.
For FSI problems in 3D, the Koiter shell allowing both tangential and transverse displacements is just short
of the $H^{2+\varepsilon}$ regularity, necessary for the uniform Lipschitz property.
\end{remark}

The assumption above implies that $(\be_N)_{N\in \N}$ is a sequence of uniformly Lipschitz functions. In particular, we have that $\|\be_N\|_{W^{1,\infty}(\omega)}\leq C$. By applying the same procedure as in the proof of Theorem~5.5-1 from \cite{CiarletVol1}, we can prove that the mappings $\id + \be_N$ are injective. This implies the injectivity of function $\bg$ defined in \eqref{inject} and ensures that the reparameterizations $\bet_N$ are well defined. Therefore, we now have that the sequence $\id + \be_N$ is a sequence of injective, uniformly Lipschitz functions. Furthermore, from $\|\be_N\|_{W^{1,\infty}(\omega)}\leq C,$ we have that the gradient of $\id + \be_N$ is bounded from below, which implies that the gradient of $(\id + \be_N)^{-1}$ is bounded from above, i.e. $(\id + \be_N)^{-1}$ is a sequence of uniformly Lipschitz functions. Finally, we have that the sequence $\id + \be_N$ is a sequence of uniformly bi-Lipschitz functions.

Using the previous conclusions and the definition of the reparameterized shell displacement, we can easily prove the following proposition:
\begin{prop}\label{bi-Lipschitz}
	The sequence $(\bet_N)_{N\in\N}$ is a sequence of uniformly bi-Lipschitz functions, i.e. sequences $(\bet_N)$ and $(\bet_N^{-1})$ are bounded in $C([0,T];W^{1,\infty}(\omega)).$
\end{prop}

\subsection{Convergence of the gradients}
In order to be able to pass to the limit in the weak formulation of the fluid-mesh-shell interaction problem, we need to show that the sequence of the gradients of fluid velocities converges weakly to the gradient of the limiting velocity.  From Theorem~\ref{thm:unif}, we know that the symmetrized gradients are 
uniformly bounded in the following way: 
\begin{equation*}
\sum_{n=0}^{N-1}\int_{\Omega^{n+1}}|\bD(\bu_N^{n+1})|^2\dt \leq K.
\end{equation*}
To show that the gradients of approximate sequences are uniformly bounded, we will use the classical Korn's inequality. 
However, since the Korn's constant depends on the fluid domain, to obtain the desired uniform estimate,
we will use the uniform bi-Lipschitz property from Proposition~\ref{bi-Lipschitz} to obtain the uniform boundedness of the gradients.
This implies weak convergence of the gradients, for which we then show converge to the gradient of the limiting velocity.

We start by introducing the following characteristic functions, which are defined on the maximal domain $\Omega_M$:
\begin{align}\label{char_fun}
\chi_N (t,\mathbf{x})=\begin{cases}
1, \quad t\in (n\dt,(n+1)\dt],\;\mathbf{x}\in\Omega^{n+1},\\
0,\quad \text{ otherwise},
\end{cases}
\end{align}
\begin{align*}
\chi (t,\mathbf{x})=\begin{cases}
1, \quad t\in (0,T],\;\mathbf{x}\in\Omega^{\bet}(t),\\
0,\quad \text{ otherwise}.
\end{cases}
\end{align*}
Here, $\bet$ is the weak* limit of $\bet_N$ in $L^\infty(0,T;W^{1,\infty}(\omega)).$
Next, we show that $\chi_N\nabla\bup_N$ are uniformly bounded:
\begin{align*}
\int_0^T \|\chi_N\nabla\bup_N\|_{L^2(\Omega_M)}^2
&=\sum_{n=0}^{N-1} \|\nabla \bu_N^{n+1}\|_{L^2(\Omega^{n+1})}^2\dt 
\leq C(\Omega^{n+1})\sum_{n=0}^{N-1} \|\bD(\bu_N^{n+1})\|_{L^2(\Omega^{n+1})}^2\dt,
\end{align*}
where $C(\Omega^{n+1})$ is a constant from Korn's inequality, which depends on $\Omega^{n+1},\; n = 0,\dots, N-1.$ Since the approximate reparameterized shell displacements are uniformly bi-Lipschitz functions, from Lemma 1 in \cite{Velcic} we obtain the existence of a universal Korn's constant $D$ such that
\begin{equation*}
C(\Omega^{n+1})\leq D, \quad, n=0,\dots,N-1.
\end{equation*}
Finally, 
\begin{equation*}
\int_{0}^{T}\|\chi_N \nabla \bup_N\|_{L^2(\Omega_M)}^2 \leq D\sum_{n=0}^{N-1} \|\bD(\bu_N^{n+1})\|_{L^2(\Omega^{n+1})}^2 \dt\leq DK,
\end{equation*}
which implies that the sequence $(\chi_N\nabla\bup_N)_{N\in\N}$ is uniformly bounded in $L^2(0,T;L^2(\Omega_M)).$ 
Therefore, there exists a subsequence, which we denote the same way, and a function $\bG \in L^2(0,T;L^2(\Omega_M))$ such that
\begin{equation*}
\chi_N\nabla\bup_N\rightharpoonup \bG \text{ weakly in } L^2(0,T;L^2(\Omega_M)),
\end{equation*}
i.e. 
\begin{equation*}
\lim_{N\to\infty}\int_0^T\int_{\Omega_M}\chi_N\nabla\bup_N\cdot\bups = \int_0^T\int_{\Omega_M} \bG\cdot\bups,\quad \forall\bups\in C_c^{\infty}((0,T)\times\Omega_M).
\end{equation*}
We want to show that $\bG$ equals $\chi\nabla \bup.$ 
In order to do that, we first consider the set $\Omega_M\setminus\Omega^{\bet}(t)$ and show that $\bG=0$ there.
Then, we consider the set $\Omega^{\bet}(t)$ and show that $\bG=\nabla\bu$ there.

Let $\bups_1$ be a test function such that $\text{supp}\,\bups_1\subset (0,T)\times(\Omega_M\setminus\Omega^{\bet}(t)).$ Using the uniform convergence of the sequence $\bet_N,$ we can find an $N_1$ such that $\chi_N(\mathbf{x})=\chi(\mathbf{x})=0,N\geq N_1,\mathbf{x}\in \text{supp}\,\bups_1.$
Therefore, we have
\begin{equation*}
\int_0^T\int_{\Omega_M} \bG\cdot\bups_1=\lim_{N\to\infty}\int_0^T\int_{\Omega_M}\chi_N\nabla\bup_N\cdot\bups_1=0.
\end{equation*}
Thus, $\bG=0$ on $(0,T)\times (\Omega_M\setminus\Omega^{\bet}(t)).$

Let $\bups_2$ be a test function such that $\text{supp}\,\bups_2\subset (0,T)\times \Omega^{\bet}(t).$ Using the uniform convergence of the sequence $\bet_N,$ we can find 
an $N_2$ such that $\chi_N(\mathbf{x})=\chi(\mathbf{x})=1,N\geq N_2,\mathbf{x}\in \text{supp}\, \bups_2.$ Therefore, we have
\begin{equation*}
\int_0^T\int_{\Omega_M} \bG\cdot\bups_2=\lim_{N\to\infty}\int_0^T\int_{\Omega_M}\chi_N\nabla\bup_N\cdot\bups_2 = \int_0^T\int_{\Omega^{\bet}(t)}\nabla\bu\cdot \bups_2.
\end{equation*}
Thus, $\bG=\nabla\bu$ on $(0,T)\times \Omega^{\bet}(t).$
This shows that
\begin{equation}\label{grad_conv}
\chi_N\nabla\bup_N\rightharpoonup\chi\nabla\bup\quad\text{in}\quad L^2(0,T;L^2(\Omega_M)),
\end{equation}
which proves the weak convergence of the gradients of $\bup_N$ to the gradient of the limiting function $\bup$.

\subsection{Strong convergence of the fluid and structure velocities}
One of the main accomplishments of this work is the ability to show strong convergence of the fluid and structure velocities
to the weak solution of the coupled, nonlinear FSI problem. 
Crucial for proving this result is a recently published generalization of the Aubin-Lions-Simon compactness lemma to problems on moving domains \cite{BorSunCompact}.
For completeness, we state this result, and show that the assumptions of the  theorem 
hold true for our FSI problem,
thereby implying the strong convergence of the fluid and structure velocities.

\begin{theorem}[\cite{BorSunCompact}] \label{thm:compact}
	Let $V$ and $H$ be Hilbert spaces such that $V\subset\subset H.$ Suppose that $\{\bu_N\}\subset L^2(0,T;H)$ is a sequence such that $\bu_N(t,\cdot)=\bu_N^n(\cdot)$ on $((n-1)\dt,n\dt],$ $n=1,\dots,N$ with $N\dt = T.$
	Let $V_N^n$ and $Q_N^n$ be Hilbert spaces such that $(V_N^n,Q_N^n)\hookrightarrow V\times V,$ where the embeddings are uniformly continuous with respect to $N$ and $n$, and $V_N^n\subset\subset \overline{Q_N^n}^H\hookrightarrow(Q_N^n)'.$ Let $\bu_N^n\in V_N^n,\; n =1,\dots,N.$ If the following is true:
	\begin{enumerate}[(A)]
		\item There exists a universal constant $C>0$ such that for every $N$
		\begin{itemize}
			\item[(A1)]$\sum_{n=1}^N\|\bu_N^n\|_{V_N^n}^2\dt\leq C,$
			\item[(A2)]$\|\bu_N\|_{L^\infty(0,T;H)}\leq C,$
			\item[(A3)]$\|\tau_{\dt}\bu_N-\bu_N\|_{L^2(\dt,T;H)}\leq C\dt,$
		\end{itemize}
		where $\tau_{\dt}\bu_N(t,\cdot)=\bu_N(t-\dt,\cdot)$ denotes the time-shift.
		\item There exists a universal constant $C>0$ such that
		\begin{equation*}
		\|P_N^{n+1}\frac{\bu_N^{n+1}-\bu_N^n}{\dt}\|_{(Q_N^{n+1})'}\leq C(\|\bu_N^{n}\|_{V_N^{n}}+1),\; n= 0,\dots,N-1,
		\end{equation*}
		where $P_N^{n+1}$ is the orthogonal projector onto $\overline{Q_N^{n+1}}^H.$
		\item The function spaces $Q_N^n$ and $V_N^n$ depend smoothly on time in the following sense:
		\begin{itemize}
			\item[(C1)] For every $N\in\N$, and for every $l\in\{1,\dots,N\}$ and $n\in\{1,\dots,N-l\}$, there exists a space $Q_N^{n,l}\subset V$ and the operators $J_{N,l,n}^i : Q_N^{n,l}\to Q_N^{n+i},\; i = 0,\dots,l,$ such that for every $\bq\in Q_N^{n,l}$
			\begin{equation}\label{C1_first}
			\|J_{N,l,n}^i\bq\|_{Q_N^{n+i}}\leq C\|\bq\|_{Q_N^{n,l}},\quad i\in\{0,\dots,l\},
			\end{equation} 
			and
			\begin{equation}\label{C1_second}
			\left((J_{N,l,n}^{j+1}\bq - J_{N,l,n}^j\bq),\bu_N^{n+j+1}\right)_H\leq C\dt \|\bq\|_{Q_N^{n,l}}\|\bu_N^{n+j+1}\|_{V_N^{n+j+1}},\quad j\in\{0,\dots,l-1\},
			\end{equation}
			\begin{equation}\label{C1_third}
			\|J_{N,l,n}^i\bq-\bq\|_H\leq C\sqrt{l\dt}\|\bq\|_{Q_N^{n,l}},\quad i\in\{0,\dots,l\},
			\end{equation}
			where $C>0$ is independent of $N,n$ and $l.$
			\item[(C2)] Let $V_N^{n,l}=\overline{Q_N^{n,l}}^V.$ There exist the functions $I_{N,l,n}^i:V_N^{n+1}\to V_N^{n,l},\; i = 0,\dots,l,$ and a universal constant $C>0$ such that for every  $\bv\in V_N^{n+i}$
			\begin{equation}\label{C2_first}
			\|I_{N,l,n}^i\bv\|_{V_N^{n,l}}\leq C\|\bv\|_{V_N^{n+i}},\quad i\in\{0,\dots,l\},
			\end{equation}
			\begin{equation}\label{C2_second}
			\|I_{N,l,n}^i\bv-\bv\|_H\leq g(l\dt)\|\bv\|_{V_N^{n+i}},\quad i \in\{0,\dots,l\},
			\end{equation}
			where $g:\R^+\to\R^+$ is a universal, monotonically increasing function such that $g(h)\to 0$ as $h \to 0.$
			\item[(C3)]\text{Uniform Ehrling property.} For every $\delta>0$ there exists a constant $C(\delta)>0$ independent of $N,l$ and $n$ such that
			\begin{equation}\label{ehrling}
			\|\bv\|_H\leq \delta\|\bv\|_{V_N^{n,l}}+C(\delta)\|\bv\|_{(Q_N^{n,l})'};
			\end{equation} 
		\end{itemize}
		then $\{\bu_N\}$ is relatively compact in $L^2(0,T;H).$
		
	\end{enumerate}
	
\end{theorem}
Assumptions (A) and (B) correspond to the classical assumptions of the Aubin-Lions compactness lemma. Assumptions (C) are new 
in the sense that they apply to
problems on moving domains,
as they describe the smooth dependence of the fluid domains on time, needed for the compactness argument. 

To apply Theorem \ref{thm:compact} we start by defining the overarching functions spaces $V$ and $H$ from the theorem,
which must be such that $V \subset\subset H$:
\begin{equation*}
\begin{array}{rl}
V&=H^s(\Omega_M)\times H^s(\omega)\times L^2(\mathcal{N}) \times L^2(\mathcal{N}),\\
H&=L^2(\Omega_M)\times L^2(\omega)\times H^{-s}(\mathcal{N})\times H^{-s}(\mathcal{N}),
\end{array}
\quad 0<s<1/2,
\end{equation*}
where we recall that $\Omega_M$ is the maximal fluid domain containing all the time-dependent fluid domains.
\begin{remark}
	We assume that all the functions defined on the time-dependent fluid domains are extended by $0$ to $\Omega_M.$
\end{remark}

Furthermore, we define the Hilbert spaces $V_N^n$ and $Q_N^n$
such that $(V_N^n,Q_N^n)\hookrightarrow V\times V,$ where the embeddings are uniformly continuous with respect to $N$ and $n$, and 
$V_N^n\subset\subset \overline{Q_N^n}^H\hookrightarrow(Q_N^n)'$:
\begin{equation}
V_N^n=\{(\bu,\bv,\bk,\bz)\in V_F^n\times H^{1/2}(\omega)\times L^2(\mathcal{N})\times L^2(\mathcal{N}): (\bu\circ\bphi^n)|_{\Gamma}\circ\bvarphi = \bv\},
\end{equation}
\begin{equation}\label{test_space}
Q_N^n=\{ (\bups,\bpsi,\bxi,\bzeta)\in (V_F^n\cap H^5(\Omega^n))\times V_K\times V_S \times V_S: (\bups\circ\bphi^n)|_{\Gamma}\circ\bvarphi=\bpsi,\bpsi\circ\bpi=\bxi\}.
\end{equation}
These correspond to the approximation solution and test spaces, respectively. 
\begin{remark}
Notice that the test space \eqref{test_space} is slightly stronger than necessary for the weak formulation, since it is intersected with $H^5$.  This simplifies certain estimates, presented below, and it does not change the final result, since the stronger test space is 
dense in the natural test space of weak solutions. 
\end{remark}

The weak formulation of the coupled, semi-discretized problem, obtained by adding the weak formulation for the semi-discretized structure subproblem \eqref{sub_struct},
 and the weak formulation for the semi-discretized fluid subproblem \eqref{sub_fluid}, reads:
\begin{align}\label{discrete_weak}
\begin{split}
&\dis \rho_F \int_{\Omega^{n+1}}\frac{\bu^{n+1}-\buh^{n}}{\dt}\cdot \bups  + \frac{\rho_F}{2}\int_{\Omega^{n+1}}(\nabla\cdot\bs)(\buh^n\cdot\bups)\\
&+ \frac{\rho_F}{2}\int_{\Omega^{n+1}}\big[((\buh^n-\bs)\cdot\nabla)\bu^{n+1}\cdot \bups- ((\buh^n-\bs)\cdot\nabla)\bups\cdot \bu^{n+1} \big]\\
&+ 2\mu_F \int_{\Omega^{n+1}}\bD(\bu^{n+1}):\bD(\bups)
+\rho_K h \int_{\omega}\frac{\bv^{n+1}-\bv^n}{\dt}\cdot\bpsi R+a_K(\be^{n+1},\bpsi)\\
&+\rho_S\sum_{i=1}^{n_E}A_i\int_{0}^{l_i}\frac{\bk_i^{n+1}-\bk_i^n}{\dt}\cdot\bxi_i+\rho_S\sum_{i=1}^{n_E}\int_{0}^{l_i}M_i\frac{\bz_i^{n+1}-\bz_i^n}{\dt}\cdot\bzeta_i
+a_S(\bw^{n+1},\bzeta)\\
&=\dis P_{in}^n\int_{\Gamma_{in}}\upsilon_z - P_{out}^n\int_{\Gamma_{out}}\upsilon_z, \quad \forall (\bups,\bpsi,\bxi,\bzeta) \in  Q_N^n.
\end{split}
\end{align}
\begin{theorem}\label{thm:strong}
	Let $\{(\bu_N,\bv_N,\bk_N,\bz_N)\}$ be a sequence of approximate solutions defined by piecewise constant extensions \eqref{approx_sol}
	of approximate solutions satisfying the weak formulation \eqref{discrete_weak} and uniform energy estimates from Theorem~\ref{thm:unif}. Then $\{(\bu_N,\bv_N,\bk_N,\bz_N)\}$ is relatively compact in $L^2(0,T;H).$
\end{theorem}

\begin{notation}
Without loss of generality, to simplify notation, 
 throughout the rest of this section we will be assuming that all the physical constants are equal 1.
\end{notation}

\begin{proof}
	We show that  (A)-(C) from Theorem~\ref{thm:compact} are satisfied.
	\\[3mm]
	\noindent
	\textbf{Property A.} We need to show that there exists a universal constant $C>0$ such that for every $N,$ the estimates (A1)-(A3) hold. 

\vskip 0.1in
\noindent
		{\bf{(A1) The $L^2(0,T;V_N^n)$ estimate:}}
		\begin{align*}
		\sum_{n=1}^N\|(\bu_N^n,\bv_N^n,\bk_N^n,\bz_N^n)\|_{V_N^n}^2\dt = 
		\sum_{n=1}^N &\left(\|\bu_N^n\|_{H^1(\Omega^n)}^2 
		+ \|\bv_N^n\|_{H^{1/2}(\omega)}^2
		+ \|\bk_N^n\|_{L^2(\mathcal{N})}^2
		+ \|\bz_N^n\|_{L^2(\mathcal{N})}^2\right)\dt.
		\end{align*}
		The approximate fluid and mesh velocities are uniformly bounded due to the energy estimates from 
		Statement 1 of Theorem~\ref{thm:unif}. For the shell velocity, by the trace theorem, we have
		\begin{equation}
		\|\bv_N^n\|_{H^{1/2}(\omega)}^2 \leq C \|\bu_N^n\|_{H^1(\Omega^n)}^2,
		\end{equation}
		and the right-hand side is again bounded due to the uniform energy estimates provided by Theorem~\ref{thm:unif}.
\vskip 0.1in
\noindent		
		{\bf{(A2) The $L^\infty(0,T;H)$ estimate:}}
		\begin{align*}
		&\|(\bu_N,\bv_N,\bk_N,\bz_N)\|_{L^\infty(0,T;H)}\\
		&= \|\bu_N\|_{L^\infty(0,T;L^2(\Omega_M))}
		+ \|\bv_N\|_{L^\infty(0,T;L^2(\omega))}
		+ \|\bk_N\|_{L^\infty(0,T;H^{-s}(\mathcal{N}))}
		+ \|\bz_N\|_{L^\infty(0,T;H^{-s}(\mathcal{N}))}\\
		&=\max_{n=1,\dots,N}\left(\|\bu_N^n\|_{L^2(\Omega^n)}
		+\|\bv_N^n\|_{L^2(\omega)}
		+\|\bk_N^n\|_{H^{-s}(\mathcal{N})}
		+\|\bz_N^n\|_{H^{-s}(\mathcal{N})}
		\right)\\
		&\leq \max_{n=1,\dots,N}\left(\|\bu_N^n\|_{L^2(\Omega^n)}
		+\|\bv_N^n\|_{L^2(\omega)}
		+\|\bk_N^n\|_{L^2(\mathcal{N})}
		+\|\bz_N^n\|_{L^2(\mathcal{N})}
		\right).
		\end{align*}
		The uniform bounds of the right-hand side follow from Statement 1 of Theorem~\ref{thm:unif}.

	This completes the proof of Property A, since condition (A3) follows from Property B, as proved in Theorem 3.2. in \cite{BorSunCompact}.
	\\[3mm]
	\noindent
	\textbf{Property B.} We need to obtain a uniform bound of the time derivative in the following weak norm:
	\begin{align*}
	&\left\|P_N^{n+1}\frac{(\bu_N^{n+1},\bv_N^{n+1},\bk_N^{n+1},\bz_N^{n+1})-(\bu_N^n,\bv_N^n,\bk_N^n,\bz_N^n)}{\dt}\right\|_{(Q_N^{n+1})'}\\
	&=\sup_{\|(\bups,\bpsi,\bxi,\bzeta)\|_{Q_N^{n+1}}=1}\left| \int_{\Omega^{n+1}}\frac{\bu_N^{n+1}-\bu_N^{n}}{\dt}
	\cdot\bups +\int_{\omega}\frac{\bv_N^{n+1}-\bv_N^n}{\dt}\cdot\bpsi\right.\\
	&\quad\left.+ \sum_{i=1}^{n_E}\int_{0}^{l_i}\frac{(\bk_N^{n+1})_i- (\bk_N^{n})_i}{\dt}\cdot\bxi_i+  \sum_{i=1}^{n_E}\int_{0}^{l_i}\frac{(\bz_N^{n+1})_i- (\bz_N^{n})_i}{\dt}\cdot\bzeta_i \right|.
	\end{align*}
	We start by adding and subtracting the function $\buh_N^n$ which is defined in \eqref{ale_comp_fun}:
	\begin{align*}
	&\left| \int_{\Omega^{n+1}}\frac{\bu_N^{n+1}-\bu_N^{n}\pm \buh_N^n}{\dt} \cdot\bups +\int_{\omega}\frac{\bv_N^{n+1}-\bv_N^n}{\dt}\cdot\bpsi\right.\\
	&\quad+\left. \sum_{i=1}^{n_E}\int_{0}^{l_i}\frac{(\bk_N^{n+1})_i- (\bk_N^{n})_i}{\dt}\cdot\bxi_i+ \sum_{i=1}^{n_E}\int_{0}^{l_i}\frac{(\bz_N^{n+1})_i- (\bz_N^{n})_i}{\dt}\cdot\bzeta_i\right|\\
	&\leq \left| \int_{\Omega^{n+1}}\frac{\bu_N^{n+1}-\buh_N^{n}}{\dt} \cdot\bups +\int_{\omega}\frac{\bv_N^{n+1}-\bv_N^n}{\dt}\cdot\bpsi\right.\\  
	&\quad+\left. \sum_{i=1}^{n_E}\int_{0}^{l_i}\frac{(\bk_N^{n+1})_i- (\bk_N^{n})_i}{\dt}\cdot\bxi_i  + \sum_{i=1}^{n_E}\int_{0}^{l_i}\frac{(\bz_N^{n+1})_i- (\bz_N^{n})_i}{\dt}\cdot\bzeta_i \right| + \left|\int_{\Omega^{n+1}}\frac{\buh_N^n-\bu_N^n}{\dt}\cdot\bups  \right|.
	\end{align*}
	We rewrite the first term by using the weak formulation \eqref{discrete_weak} to obtain the following estimate:
	\begin{align*}
	&\left| \int_{\Omega^{n+1}}\frac{\bu_N^{n+1}-\buh_N^{n}}{\dt} \cdot\bups +\int_{\omega}\frac{\bv_N^{n+1}-\bv_N^n}{\dt}\cdot\bpsi\right.\\
	&\quad+\left. \sum_{i=1}^{n_E}\int_{0}^{l_i}\frac{(\bk_N^{n+1})_i- (\bk_N^{n})_i}{\dt}\cdot\bxi_i +  \sum_{i=1}^{n_E}\int_{0}^{l_i}\frac{(\bz_N^{n+1})_i- (\bz_N^{n})_i}{\dt}\cdot\bzeta_i \right|\\
	&\leq C_1 \|\bv_N^{n+1/2}\|_{L^2}\|\buh_N^n\|_{L^2}\|\bups\|_{L^\infty} + C_2\|\nabla\bu_N^{n+1}\|_{L^2}\|\buh_N^n\|_{L^2}\|\|\bups\|_{L^\infty} + C_3\|\nabla\bu_N^{n+1}\|_{L^2}\|\nabla\bups\|_{L^2}\\
	&\quad+ C_4\|\be\|_{H^2}\|\bpsi\|_{H^2} + C_5 \|\partial_s\bw\|_{L^2}\|\partial_s\bzeta\|_{L^2} + C_6\|\bups\|_{L^\infty}\\
	&\leq C (\|(\bu_N^{n+1},\bv_N^{n+1},\bk_N^{n+1},,\bk_N^{n+1})\|_{V_N^{n+1}}+1)\|(\bups,\bpsi,\bxi,\bzeta)\|_{Q_N^{n+1}}.
	\end{align*}
	
	\noindent
	To estimate the second term, we first notice that function $\buh_N^n$ is $0$ outside domain $\Omega^{n+1},$ while function $\bu_N^n$ is $0$ outside domain $\Omega^n$. 
	See Fig.~\ref{fig:inter} for an example of the mutual position of domains $\Omega^n$ and $\Omega^{n+1}$. To simplify the estimate we introduce
	$A=\Omega^{n+1}\cap\Omega^n, B_1=\Omega^{n+1}\backslash\Omega^{n}$ and $B_2=\Omega^n\setminus\Omega^{n+1},$ and estimate the integrals over $A,B_1$ and $B_2$ separately. 
	
	\begin{figure}[t]
		\centering
		\includegraphics[width=0.7\linewidth]{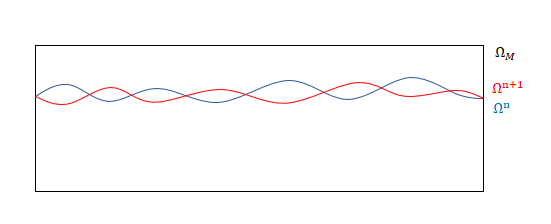}	
		\caption{The 2D fluid domains at time steps $t_n$ and $t_{n+1}$}
		\label{fig:inter}
	\end{figure}
	
	\noindent
	First we start with the integral over $A$, i.e. over the area of the intersection of two consecutive domains $\Omega^{n}$ and $\Omega^{n+1}:$
	\begin{align*}
	\left|\int_{A}(\buh_N^n-\bu_N^n)\cdot\bups\right|&=\left|\int_A(\bu_N^n\circ A^{n+1,n}-\bu_N^n)\cdot\bups\right|\\
	&=\left|\int_A\left( \bu_N^n(z,\frac{R+\et_N^n(z,\theta)}{R+\et_N^{n+1}(z,\theta)}r,\theta)- \bu_N^n(z,r,\theta)   \right)\cdot \bups(z,r,\theta)\; dzdrd\theta \right|
	\end{align*}
	By using the mean value theorem and H\"older's inequality we get:
	\begin{align*}
	\left|\int_{A}(\buh_N^n-\bu_N^n)\cdot\bups\right|
	&\leq C \|\nabla\bu_N^n\cdot(\et_N^n-\et_N^{n+1})\mathbf{e}_r\|_{L^1(A)}\|\bups\|_{L^\infty(A)}\\
	&\leq C\dt\|\nabla\bu_N^n\|_{L^2(A)}\|\bv_N^{n+1/2}\|_{L^2(A)}\|\bups\|_{L^\infty(A)}
	\leq C\|\bu_N^n\|_{H^1(A)}\|\bups\|_{H^5(A)}\\
	&\leq C\|(\bu_N^n,\bv_N^n,\bk_N^n,\bz_N^n)\|_{V_N^n}\|(\bups,\bpsi,\bxi,\bzeta)\|_{Q_N^n}.
	\end{align*}
	Notice how the higher regularity of the test space $Q_N^n$, commented in the Remark just below the definition of $Q_N^n$ in \eqref{test_space}, provided the upper bound for the $L^\infty$-norm of the test function $\bups.$
	
	To estimate the integral over $B_1,$ we use the fact that $\bu_N^n=0$ on $B_1$ to obtain:
	\begin{align*}
	\left|\int_{B_1}(\buh_N^n-\bu_N^n)\cdot\bups\right|
	&=\left| \int_{B_1} \buh_N^n(z,r,\theta)\cdot \bups(z,r,\theta)\;dzdrd\theta\right|\\
	&=\left|  \int_{\omega}\left(\int_{R+\et_N^{n}}^{R+\et_N^{n+1}}\buh_N^n(z,r,\theta)\cdot \bups(z,r,\theta)\;dr\right) dzd\theta \right|\\
	&\leq\left|\int_{\omega}\max_r (\buh_N^n(z,r,\theta)\cdot\bups(z,r,\theta))\int_{R+\et_N^{n}}^{R+\et_N^{n+1}}\; dr dz d\theta  \right|\\
	&\leq C \int_{\omega}\|\partial_r u_r^n(z,\cdot,\theta)\|_{L_r^2}\|\bups\|_{L^\infty} \|\dt\bv_N^{n+1/2}\|_{L^2}\;dzd\theta\\
	&\leq C\dt \|\nabla\bu_N^n\|_{L^2}\|\bups\|_{L^\infty}\\
	&\leq C\|(\bu_N^n,\bv_N^n,\bk_N^n,\bz_N^n)\|_{V_N^n}\|(\bups,\bpsi,\bxi,\bzeta)\|_{Q_N^n}.
	\end{align*}
	The integral over $B_2$ can be estimated in the same way as the integral over $B_1$ by using the fact that $\buh_N^n=0$ on $B_2.$ These estimates, together with the estimate obtained from the weak formulation, complete the proof of Property B, i.e. we have
	\begin{equation*}
	\left\|P_N^{n+1}\frac{(\bu_N^{n+1},\bv_N^{n+1},\bk_N^{n+1},\bz_N^{n+1})-(\bu_N^n,\bv_N^n,\bk_N^n,\bz_N^n)}{\dt}\right\|_{(Q_N^{n+1})'}\leq C \left(\|(\bu_N^n,\bv_N^n,\bk_N^n,\bz_N^n)\|_{V_N^n}+1\right).
	\end{equation*}
	\\[3mm]
	\noindent
	\textbf{Property C.} This property investigates smooth dependence of the test and solution spaces on time,
	namely, on the change of the fluid domain. Property C relies on being able to construct a ``common'' test space, 
	and a ``common'' solution space for all the time-shifts by $i \Delta t$, with $i = 0,\dots,l$, which are ``close'' in the relevant topologies to
	the original test and solutions spaces for these time-shift, as described by properties (C1)-(C3). The common functions spaces are based 
	on the existence of a ``local'' maximal domain $\Omega^{n,l}$ which contains all the fluid domains $\Omega^{n+i},i=0,\dots,l:$
	\begin{equation}
	\Omega^{n,l}=\{(z,r,\theta):z\in(0,L),r\leq R + \et_N^{n,l}(z,\theta),\theta\in(0,2\pi)  \},
	\end{equation}
	where $\et_N^{n,l}(z,\theta)=\underset{i=0,\dots,l}{\max}\et_N^{n+i}(z,\theta),$ mollified if necessary to get the smooth functions. 
	The existence of the local maximal domains is guaranteed by the uniform energy estimates, presented in Theorem~\ref{thm:unif}.
		\\[3mm]
	\noindent
	\textbf{Property C1.}
	A common test space is then defined in the following way:
	\begin{equation}\label{common_test_space}
	Q_N^{n,l}=\{(\bups,\bpsi,\bxi,\bzeta)\in (V_F(\Omega^{n,l})\cap H^5(\Omega^{n,l}))\times V_K\times V_S\times V_S: (\bups\circ\bphi^n)|_{\Gamma}\circ\bvarphi=\bpsi,\bpsi\circ\bpi=\bxi\}.
	\end{equation}
	For each $(\bups,\bpsi,\bxi,\bzeta)\in Q_N^{n,l}$, we define $J_{N,l,n}^i$ as the restriction:
	\begin{equation*}
	J_{N,l,n}^i (\bups,\bpsi,\bxi,\bzeta)=(\bups|_{\Omega^{n+i}},\bups|_{\Gamma^{n+i}}\circ\bvarphi,(\bups|_{\Gamma^{n+i}}\circ\bvarphi)|_{\mathcal{N}},\bzeta),
	\end{equation*}
	and set
	\begin{equation}
	(\bups^i,\bpsi^i,\bxi^i,\bzeta^i):= (\bups|_{\Omega^{n+i}},\bups|_{\Gamma^{n+i}}\circ\bvarphi,(\bups|_{\Gamma^{n+i}}\circ\bvarphi)|_{\mathcal{N}},\bzeta).
	\end{equation}
	The mappings $J_{N,l,n}^i $ satisfy all the properties from Theorem~\ref{thm:compact}. Indeed, 
	property \eqref{C1_first} follows directly from the definition of $J_{N,l,n}^i.$ To verify property \eqref{C1_second}, we need to calculate 
	\begin{equation}
	\left(J^{j+1}(\bups,\bpsi,\bxi,\bzeta)-J^{j}(\bups,\bpsi,\bxi,\bzeta),(\bu_N^{n+j+1},\bv_N^{n+j+1},\bk_N^{n+j+1},\bz_N^{n+j+1}) \right)_H.
	\end{equation}
	We estimate each term separately. The first term is estimated similar to Property B:
	\begin{align*}
	\left|\int_{\Omega_M}(\bups^{j+1}-\bups^j)\cdot \bu_N^{n+j+1}  \right|&=\left| \int_{\Omega^{n+j+1} \Delta \Omega^{n+j}}\bups \cdot \bu_N^{n+j+1} \right|\\
	&=\left|  \int_{\omega}\left(\int_{R+\et_N^{n+j}}^{R+\et_N^{n+j+1}}\bups(z,r,\theta)\cdot\bu_N^{n+j+1}(z,r,\theta)\;dr\right) dzd\theta \right|\\
	&\leq\left|\int_{\omega}\max_r (\bups(z,r,\theta)\cdot\bu_N^{n+j+1}(z,r,\theta))\int_{R+\et_N^{n+j}}^{R+\et_N^{n+j+1}}\; dr dz d\theta  \right|\\
	&\leq C \int_{\omega}\|\bups\|_{L^\infty} \|\partial_r u_r^{n+j+1}(z,\cdot,\theta)\|_{L_r^2} \|\dt\bv_N^{n+j+1/2}\|_{L^2}\;dzd\theta\\
	&\leq C \|\bups\|_{L^\infty}\|\nabla\bu_N^{n+j+1}\|_{L^2}\|\dt\bv_N^{n+j+1/2}\|_{L^2}\\
	&\leq C\dt \|(\bups,\bpsi,\bxi,\bzeta)\|_{Q_N^{n,l}}\|(\bu_N^{n+j+1},\bv_N^{n+j+1},\bk_N^{n+j+1},\bz_N^{n+j+1})\|_{V_N^{n+j+1}}.
	\end{align*}
	Before estimating the second term, we note that
	\begin{equation*}
	\bpsi^i=\bups|_{\Gamma^{n+i}}\circ\bvarphi=(\bups\circ\bphi^{n+i})|_{\Gamma}\circ\bvarphi.
	\end{equation*}
	Then, by using the mean value theorem, we get
	\begin{align}\label{pom1}
	\begin{split}
	\left| \int_{\omega} (\bpsi^{j+1}-\bpsi^j)\cdot \bv_N^{n+j+1}\right|&\leq \int_{\omega}|(\bpsi^{j+1}-\bpsi^j)\cdot \bv_N^{n+j+1}|\leq\|\bpsi^{j+1}-\bpsi^j\|_{L^2(\omega)}\|\bv_N^{n+j+1}\|_{L^2(\omega)}\\
	&=\|(\bups(\bphi^{n+j+1}(\cdot))-\bups(\bphi^{n+j}(\cdot)))|_{\Gamma}\circ\bvarphi\|_{L^2(\omega)}\|\bv_N^{n+j+1}\|_{L^2(\omega)}\\
	&\leq \|\nabla \bups\|_{L^\infty(\Omega^{n,l})}\|(\bphi^{n+j+1}-\bphi^{n+j})|_{\Gamma}\circ\bvarphi\|_{L^2(\omega)}\|\bv_N^{n+j+1}\|_{L^2(\omega)}.
	\end{split}
	\end{align}
	Recall that $\bphi^{i}|_{\Gamma}\circ\bvarphi=\id + \be_N^i,$ so we can further estimate the right-hand side to obtain
	\begin{align}\label{pom2}
	\begin{split}
	\left| \int_{\omega} (\bpsi^{j+1}-\bpsi^j)\cdot \bv_N^{n+j+1}\right| &\leq \|\nabla\bups\|_{L^\infty(\Omega^{n,l})}\|\be_N^{n+j+1}-\be_N^{n+j}\|_{L^2(\omega)}\|\bv_N^{n+j+1}\|_{L^2(\omega)}\\
	&= \|\nabla\bups\|_{L^\infty(\Omega^{n,l})}\dt\|\bv_N^{n+j+1/2}\|_{L^2(\omega)}\|\bv_N^{n+j+1}\|_{L^2(\omega)}\\
	&\leq C\dt \|(\bups,\bpsi,\bxi,\bzeta)\|_{Q_N^{n,l}}\|(\bu_N^{n+j+1},\bv_N^{n+j+1},\bk_N^{n+j+1},\bz_N^{n+j+1})\|_{V_N^{n+j+1}}.
	\end{split}
	\end{align}
	What is left is to take care of the term $\dis\left|\sum_{n=1}^{n_E}\int_0^{l_i}(\bxi_i^{j+1}-\bxi_i^j)\cdot\bk_i^{n+j+1}  \right|.$ Recall that
	\begin{equation*}
	\bxi^i=(\bups|_{\Gamma^{n+i}}\circ\bvarphi)|_{\mathcal{N}}=\bpsi^i|_{\mathcal{N}}=\bpsi^i\circ\bpi.
	\end{equation*}
	Thus
	\begin{align*}
	\dis\left|\sum_{n=1}^{n_E} \right. &\left. \dis\int_0^{l_i}(\bxi_i^{j+1}-\bxi_i^j)\cdot\bk_i^{n+j+1}  \right| \leq \|\bxi^{j+1}-\bxi^j\|_{L^2(\mathcal{N})}\|\bk_N^{n+j+1}\|_{L^2(\mathcal{N})}\\
	&=\|(\bpsi^{j+1}-\bpsi^j)|_{\mathcal{N}}\|_{L^2(\mathcal{N})}\|\bk_N^{n+j+1}\|_{L^2(\mathcal{N})}
	\leq C \|\bpsi^{j+1}-\bpsi^j\|_{L^2(\omega)}\|\bk_N^{n+j+1}\|_{L^2(\mathcal{N})}\\
	&\leq C \dt \|(\bups,\bpsi,\bxi,\bzeta)\|_{Q_N^{n,l}}\|(\bu_N^{n+j+1},\bv_N^{n+j+1},\bk_N^{n+j+1},\bz_N^{n+j+1})\|_{V_N^{n+j+1}},
	\end{align*}
	where in the last inequality we used the fact that $\|\bpsi^{j+1}-\bpsi^j\|_{L^2(\omega)}$ is bounded, which follows from \eqref{pom1} and \eqref{pom2}.
	At last, we have to check that property \eqref{C1_third} is valid, i.e.
	\begin{equation*}
	\|J_{N,l,n}^i(\bups,\bpsi,\bxi,\bzeta)-(\bups,\bpsi,\bxi,\bzeta)\|_{H}\leq C\sqrt{l\dt}\|(\bups,\bpsi,\bxi,\bzeta)\|_{Q_N^{n,l}}.
	\end{equation*}
	It is clear that $J_{N,l,n}^i(\bups,\bpsi,\bxi,\bzeta)$ and $(\bups,\bpsi,\bxi,\bzeta)$ differ only in the region $\Omega^{n,l}\backslash\Omega^{n+i},$ so the $H$-norm of the difference between the two functions can be bounded by the $Q_N^{n,l}$-norm of $(\bups,\bpsi,\bxi,\bzeta)$ and the $H^1$-norm of the difference $\bet_N^{n,l}-\bet_N^{n+i},i=0,\dots,l.$ 
	Since $\bet_N^{n,l}$ is the maximum of the finitely many functions $\bet_N^{n+i}, i=0,\dots,l$ 
	we calculate the $H^1$-norm of the difference $\bet_N^{n+i}-\bet_N^{n}$,
		by using the interpolation inequality:
	\begin{align}
\begin{split}\label{interp}
&\|\bet_N(t^n+i\dt)-\bet_N(t^n)\|_{H^{2\alpha}(\omega)} \\
&\leq C \|\bet_N(t^n+i\dt)-\bet_N(t^n)\|_{L^2(\omega)}^{1-\alpha} \|\bet_N(t^n+i\dt)-\bet_N(t^n)\|_{H^2(\omega)}^{\alpha}\\
&\leq C (l\dt)^{1-\alpha}, \text{ where } 0<\alpha<1,
\end{split}
\end{align}
with $\alpha = 1/2$, where we have used the uniform energy estimates, and the fact that
the upper   bound on $\|\bet_N^{n+i}-\bet_N^{n}\|_{H^1(\omega)}$ only depends on the width of the time interval, which is $l\dt,$ 
to get:
	\begin{equation*}
	\|\bet_N^{n+i}-\bet_N^{n}\|_{H^1(\omega)}\leq C\sqrt{l\dt},\; i = 1,\dots,l.
	\end{equation*}
	This completes the verification of Property C1.
	\\[3mm]
	\noindent
	\textbf{Property C2.} We define a common solution space $V_N^{n,l}$ to be the closure of $Q_N^{n,l}$ in $V$ (for $s=1/2$ ):
	\begin{align}\label{common_sol_space}
	\begin{split}
	V_N^{n,l}=\{(&\bu,\bv,\bk,\bz)\in H^{1/2}(\Omega^{n,l})\times H^{1/2}(\omega)\times L^2(\mathcal{N})\times L^2(\mathcal{N}):\\
	&\nabla\cdot \bu = 0,((\bu\circ\bphi^n)|_{\Gamma}\circ\bvarphi-\bv)\cdot \bn = 0  \}.
	\end{split}
	\end{align}
	To construct the mappings $I_{N,l,n}^i:V_N^{n+i}\to V_N^{n,l}$ possessing the approximation properties \eqref{C2_first} and \eqref{C2_second},
	we will need to be able to extend the functions $\bu_N^{n}\in V_F^{n}$ to a divergence-free function defined on the maximal domain $\Omega_M$.
	This can be done by using the following nontrivial result:
	\begin{lemma}\label{lemma:div}
		Let $\bu\in V_F^{n}$. Then there exists a divergence-free function $\bul\in V$ such that $\bul|_{\Omega^n}=\bu$ and
		\begin{equation}
		\|\bul\|_V \leq C \|\bu\|_{V_F^n},
		\end{equation}
		where $C$ is independent of $N$ and $n$.	
	\end{lemma}
	
	\begin{proof} To construct a divergence-free extension onto $\Omega_M$ of the fluid velocity $\bu\in V_F^{n}$ defined on $\Omega^n$,
	we take the following approach. First we ``straighten" the moving boundary by mapping the moving domain $\Omega^n$ onto the fixed, reference domain $\Omega$
	using the mapping $\bphi^n$, which is the mapping defined in \eqref{phi} via \eqref{phi_mapping} ($\bphi^n$ is an injective, orientation preserving mapping
		which maps $\Omega$ onto $\Omega^n$, $\bphi^n(\Omega)=\Omega^n$).
        Once we have a straight boundary, we can easily extend the velocity to some maximal domain $\Omega_M$
        \cite{adams}. 
	Because of the uniform Lipschitz property 
	of our approximate moving domains, and thus $\bphi^n$, we can construct such an extension so that the $H^1$ norm is ``preserved''. 
	However, since we are interested in working with velocity extensions in the physical space, we map everything back onto the physical space in the following way:
	we use $(\bphi^n)^{-1}$  to map back the part that came from $\Omega^n$ (thereby obtaining the original $\bu$ there), but, to 
	map to the physical domain the extension from $\Omega$ to $\Omega_M$, we use a mapping $(\tilde{\bphi}^n)^{-1}$, which is defined to be an 
	{\emph{extension}} of $(\bphi^n)^{-1}$ onto $\Omega_M\setminus\Omega^n$. Because of the uniform bi-Lipschitz property of the approximate domains,
	the extension $(\bphi^n)^{-1}$ can be constructed in the physical space so that it also preserves the  $H^1$ norm of the velocity. 
	Namely, the final, resulting extension mapping is uniformly continuous in $N$ and $n$,
	which is what we wanted.
	Unfortunately, the extension defined this way is not necessarily divergence-free. For this reason we must construct a ``correction'' 
	giving rise to an extension 
	that is divergence-free. The construction and existence of such an extension rests on Theorem III.3.1 from \cite{GaldiNS}.

	More precisely, we start by constructing an {\bf{extension}} $\tilde{\bu}$ to $\Omega_M$ as follows.
		Take $\bu\in V_F^{n}$  and set 
		\begin{equation}
		\bu_R=\bu\circ\bphi^n.
		\end{equation}
				The function $\bu_R$ is defined on the reference domain $\Omega$, and the following estimate holds:
		\begin{equation*}
		\|\bu_R\|_{H^1(\Omega)}\leq\|\bu\|_{H^1(\Omega^n)}\|\bphi^n\|_{W^{1,\infty}(\Omega)}\leq C \|\bu\|_{H^1(\Omega^n)}.
		\end{equation*}
		Now we {\emph{extend}} $\bu_R$ to $\R^3$ and then define $\bup_R$ as a restriction of that particular extension to $\Omega_M.$ It is clear that the following bound holds:
		\begin{equation*}
		\|\bup_R\|_{H^1(\Omega_M)}\leq C\|\bu\|_{H^1(\Omega^n)}.
		\end{equation*}
		Since we are interested in constructing an extension of the fluid velocity (which respects the $H^1$ norm and is divergence-free)
		in physical space, namely from $\Omega^n$ to $\Omega_M,$ 
		we map everything back to the physical space by defining the extended function $\bup$ in the following way:
		\begin{equation*}
		\bup =
		\left\{ 
		\begin{array}{l}
		\bup_R\circ (\bphi^n)^{-1}\text{ in }\Omega^n,\\
		\bup_R\circ (\tilde{\bphi}^n)^{-1}\text{ in }\Omega_M\setminus\Omega^n,
		\end{array}
		\right.
		\end{equation*}
		where $\tilde{\bphi}^n: \Omega_M\setminus\Omega \to \Omega_M\setminus\Omega^n$ is 
		an {\emph{extension of the mapping $\bphi^n =\bphi^{\bet}(n\dt,\cdot)$}} to the maximal domain $\Omega_M.$ One can easily check that, due to the
		bi-Lipschitz property stated in Proposition~\ref{bi-Lipschitz}, the following uniform bound holds:
		\begin{equation*}
		\|\bup\|_{H^1(\Omega_M)}\leq C \|\bu\|_{H^1(\Omega^n)}.
		\end{equation*}
		
		Unfortunately, $\bup$ is not {\bf{divergence-free}} so we need to ``correct" it. 
		The correction is designed with the help of 
	        Theorem III.3.1 from \cite{GaldiNS}, which deals with the problem of finding  a vector field $\bv\in W_0^{1,p}(\Omega)$ such that
		\begin{equation}\label{galdi}
		\nabla\cdot\bv = \bff \ {\rm in} \ \Omega,
		\end{equation}
		where $\bff \in L^p(\Omega)$ is such that 
		\begin{equation}\label{fint}
		\int_{\Omega}\bff=0,
		\end{equation}
		 and $\Omega$ has certain regularity properties, which we discuss below.
		
	        We look for the velocity ``correction'' which consists of two parts: $\bu_c$ and $\bv$, so that the velocity $\bup$, corrected by 
	        $\bu_c + \bv$ is divergence-free, and has all the desired properties. The part $\bu_c$ is introduced so that 
	        condition \eqref{fint} can be satisfied when \eqref{galdi} is solved for $\bv$, where $\bv$ is such that
	        \begin{equation}\label{fintnew}
	        \nabla \cdot \bv = - \nabla \cdot \left(\bup + \bu_c\right)\ {\rm in} \ \Omega_M\setminus \Omega^n.
	        \end{equation}
	        The resulting corrected fluid velocity $\bup + \bu_c + \bv$ is divergence-free.
	         More precisely, we want to construct a $\bu_c:(\Omega_M\setminus\Omega^n)\to \R^3$ such that it does not change the trace of the fluid velocity  on
	        the moving boundary $\Gamma^n$, such that condition \eqref{fint} is satisfied for $ \bff = - \nabla \cdot \left(\bup + \bu_c\right)$, and  such that
	        the $H^1$ norm of $\bu_c$ is controlled by the $H^1$-norm of $\bu$:
		\begin{enumerate}[(i)]
			\item $\bu_c|_{\Gamma^n}=0,$
			\item $\dis\int_{\partial(\Omega_M\setminus\Omega^n)}(\bup+\bu_c)\cdot \bn = 0,$
			\item $\|\bu_c\|_{H^1(\Omega_M\setminus\Omega^n)}\leq C \|\bu\|_{H^1(\Omega^n)}.$
		\end{enumerate} 
		The first condition will ensure that $(\bup+\bu_c)|_{\Gamma^n}=\bu,$ while the second condition is a compatibility condition corresponding to the fact that the integral of the right-hand side of problem \eqref{fintnew} has to be zero.
		
		To construct  such a $\bu_c$ we consider functions in $C_c^{\infty}(\overline{\omega}\times (R+\et^n,R_{max}]).$ It is clear that the first condition is then automatically satisfied,
		and the second condition becomes:
		\begin{equation}\label{kappa}
		\int_{\partial\Omega_M}(\bup+\bu_c)\cdot \bn = 0.
		\end{equation} 
		To satisfy condition~\eqref{kappa} we can take  $\bu_c :=-c\boldsymbol{i},$ where $c=\int_{\partial\Omega_M}\bup\cdot \bn$, and $\boldsymbol{i}$ 
		is such that $\int_{\partial\Omega_M}\boldsymbol{i}\cdot \bn=1$. 
		Finally, to obtain the desired $H^1$-estimate from condition (iii), 
		we can choose $\boldsymbol{i}$ independent of $\bup$ and $\bn$ 
		(for example, we take $\boldsymbol{i}$ such that $\text{supp}\,\boldsymbol{i}$ does not intersect any of $\Omega^n$), 
		to get:
		\begin{equation*}
		\|\bu_c\|_{H^1(\Omega_M\setminus\Omega^n)}\leq C\|\bup\|_{H^1(\Omega_M)}\leq C\|\bu\|_{H^1(\Omega^n)}.
		\end{equation*}
		
		We now focus on the construction of a function $\bv$ such that \eqref{fintnew} holds, and is such that
		\begin{equation*}
		\|\bv\|_{H^1(\Omega_M)}\leq C \|\bup+\bu_c\|_{H^1(\Omega_M)}.
		\end{equation*}
		This will follow directly from Theorem III.3.1 in \cite{GaldiNS}, if we can verify that the assumption on the regularity of domain $\Omega_M\setminus\Omega^n$ is satisfied. 
		More precisely, based on Lemma III.3.1 in \cite{GaldiNS}, if we can show that our domain is star-shaped with respect to every point of $B_R(\boldsymbol{x}_0)$, 
		such that $\overline{B_R}(\boldsymbol{x}_0)\subset \Omega_M\setminus\Omega^n$, then there exists at least one solution $\bv$. 
		
		Indeed, to verify the {\bf{star-shape property}} of our domain $\Omega_M\setminus\Omega^n$, 
		we notice that since our approximate domain $\Omega^n$ is Lipschitz, 
		the complementary domain $\Omega_M\setminus\Omega^n$ is also Lipschitz, so we can decompose it in a union of finitely many star-shaped domains. 
		Moreover,  because of the uniform Lipschitz property \eqref{Lipschitz}, the finite number of star-shaped domains is independent of $n$, as it only depends on
		the uniform Lipschitz constant $C$. 
		
		We can now apply Theorem III.3.1 from \cite{GaldiNS} with $\bff=-\nabla\cdot(\bup+\bu_c)$ to see that there exists a $\bv$ such that the following holds:
		\begin{equation*}
		\nabla\cdot\bv = -\nabla\cdot(\bup+\bu_c)
		\end{equation*}
		with
		\begin{equation*}
		\|\bv\|_{H^1(\Omega_M)}\leq C \|\bup+\bu_c\|_{H^1(\Omega_M)}.
		\end{equation*}
		Therefore, we have constructed an extension  $\bul := \bup + \bu_c + \bv \in V$ of the fluid velocity $\bu\in V_F^n$,  
		such that $\bul$ is divergence-free, and it satisfies the desired estimate $\|\bul\|_V\leq C\|\bu\|_{V_F^n}.$ 
	\end{proof}
	
	Using this result, we now define the mappings $I_{N,l,n}^i: V_N^{n+i}\to V_N^{n,l}$ in the following way:  
	\begin{equation}
	I_{N,l,n}^i(\bu_N^{n+i},\bv_N^{n+i},\bk_N^{n+i},\bz_N^{n+i})=(\bul_N^{n+i}|_{\Omega^{n,l}},(\bul_N^{n+i}|_{\Omega^{n,l}}\circ\bphi^{n})|_{\Gamma}\circ\bvarphi,\bk_N^{n+i},\bz_N^{n+i} ).
	\end{equation}
	The inequality \eqref{C2_first} from property (C2) now follows directly from the definition 
	 of the mappings $I_{N,l,n}^i$ and from Lemma~\ref{lemma:div}. 
	 
	 To see that inequality \eqref{C2_second} holds, we need to prove that there exists a universal, monotonically increasing function $g,$ which converges to $0$ as $h\to 0,$ where $h=l\dt,$ such that
	\begin{align*}
	\|I_{N,l,n}^i(\bu_N^{n+i},\bv_N^{n+i},\bk_N^{n+i},\bz_N^{n+i})-(\bu_N^{n+i},\bv_N^{n+i},\bk_N^{n+i},\bz_N^{n+i})\|_{H}\\
	\leq g(l\dt)\|(\bu_N^{n+i},\bv_N^{n+i},\bk_N^{n+i},\bz_N^{n+i})\|_{V_N^{n+i}}.
	\end{align*}
	To simplify notation, we drop the subscripts $N,l,n,$ and estimate each term separately:
	\begin{align*}
	\|I^i\bu_N^{n+i}-\bu_N^{n+i}\|_{L^2(\Omega_M)}&=
	\|\bul_N^{n+i}|_{\Omega^{n,l}}-\bu_N^{n+i}\|_{L^2(\Omega_M)}=\left(\int_{\Omega^{n,l}\backslash\Omega^{n+i}}|\bul_N^{n+i}|^2\right)^{1/2}\\
	&\leq C \|\nabla\bul_N^{n+i}\|_{L^2(\Omega_M)}\left|\int_{R+\et^{n+i}}^{R+\et^{n,l}}dr\right|^{1/2}dzd\theta\leq C\sqrt{l\dt}\|\bul_N^{n+i}\|_{H^1(\Omega_M)}\\
	&\leq C\sqrt{l\dt} \|\bu_N^{n+i}\|_{V_F^{n+i}}.
	\end{align*}
	The second term is estimated by using the mean value theorem, just like in \eqref{pom1}:
	\begin{align*}
	\|I^i\bv_N^{n+i}-\bv_N^{n+i}\|_{L^2(\omega)}&=\|(\bul_N^{n+i}|_{\Omega^{n,l}}\circ\bphi^{n})|_{\Gamma}\circ\bvarphi-(\bu_N^{n+i}\circ\bphi^{n+i})|_{\Gamma}\circ\bvarphi\|_{L^2(\omega)}\\
	&\leq C \|\nabla\bul_N^{n+i}\|_{L^2(\Omega_M)} \|\bet_N^{n+i}-\bet_N^{n}\|_{L^\infty(\omega)}\\
	&\leq C\sqrt{l\dt} \|\bu_N^{n+i}\|_{V_F^{n+i}}.
	\end{align*}
	By combining the estimates above we get:
	\begin{align*}
	&\|I_{N,l,n}^i(\bu_N^{n+i},\bv_N^{n+i},\bk_N^{n+i},\bz_N^{n+i})-(\bu_N^{n+i},\bv_N^{n+i},\bk_N^{n+i},\bz_N^{n+i})\|_{H}\\
	&= \|I^i\bu_N^{n+i}-\bu_N^{n+i}\|_{L^2(\Omega_M)} + \|I^i\bv_N^{n+i}-\bv_N^{n+i}\|_{L^2(\omega)}\\
	&\leq  C\sqrt{l\dt} \|\bu_N^{n+i}\|_{V_F^{n+i}}
	\leq g(l\dt) \|(\bu_N^{n+i},\bv_N^{n+i},\bk_N^{n+i},\bz_N^{n+i})\|_{V_N^{n+i}}.
	\end{align*}
	This finishes the proof of Property C2.
	\\[3mm]
	\noindent
	\textbf{Property C3.} We need to prove the uniform Ehrling property, stated in \eqref{ehrling}. The main difficulty comes from the fact that we have to work with moving domains, which are parameterized by $N,l,n.$ To show that the uniform Ehrling  estimate holds, independently of all three parameters, we simplify notation, and replace the indices $N,l,n$ with only one index $n$, so that our function spaces are now denoted $V_n, H_n$ and $Q_n'.$  We show the uniform Ehrling property by contradiction. We start by assuming that the statement of the uniform Ehrling property \eqref{ehrling} is false. More precisely, we assume that there exists a $\delta_0>0 $ and a sequence $\bh_n = (\bu_n,\bv_n,\bk_n,\bz_n)\in H_n$ such that
	\begin{equation*}
	\|\bh_n\|_{H}=\|\bh_n\|_{H_n}>\delta_0 \|\bh_n\|_{V_n}+n\|\bh_n\|_{Q_n'}.
	\end{equation*}
	Here we have extended the functions $\bu_n$ onto the maximal domain $\Omega_M$ by 0.
	We also replace the $V_n$ norm on the right-hand side by the norm on $V:$
	\begin{equation*}
	\|\bh_n\|_H>\delta_0 \|\bh_n\|_{V_n}+n\|\bh_n\|_{Q_n'}\geq C\delta_0\|\bh_n\|_{V}+n\|\bh_n\|_{Q_n'}.
	\end{equation*}
	Without the loss of generality we can assume that our sequence $(\bh_n)$ is such that $\|\bh_n\|_H=1.$ The two terms on the right-hand side are uniformly bounded in $n$, which implies that there exists a subsequence, which we again denote by $(\bh_n),$ such that:
	\begin{equation}\label{condition}
	\|\bh_n\|_H=1,\quad \|\bh_n\|_{V}\leq \frac{1}{C\delta_0},\quad \|\bh_n\|_{Q_n'}\to 0.
	\end{equation}
	Since $(\bh_n)$ is uniformly bounded in $V$, and by the compactness of the embedding of $V$ into $H$, we conclude that there exists a subsequence $(\bh_n)$ that converges to $\bh$ strongly in $H$. We want to show that $\bh=(\bu,\bv,\bk,\bz)=0,$ i.e. that $\bh_n\to 0$ in $H$,
	 which would contradict the assumption $\|\bh_n\|_H=1.$
	
	We start by showing that $\bu=0$ in $\Omega_M.$ 
	Recall that $\bet_n=\bet_N^{n,l}$ is the maximum of finitely many functions $\bet_N^{n+i},\;i=0,\dots,l,$ which converge uniformly, and denote by $\bet^{\star}$ the limit of $\bet_n$ as $n$ tends to infinity. Furthermore,  denote by $\Omega^{\star}$ the fluid domain determined by the function $\bet^{\star}.$ We will consider the set $\Omega_M\setminus\Omega^{\star}$ and show that $\bu=0$ there, and then the set $\Omega^{\star}$ and show that $\bu=0$ there. For that purpose, we introduce the following characteristic functions, which are defined on the maximal domain $\Omega_M:$\\
	\begin{minipage}{0.5\linewidth}
		\centering
		\begin{equation*}
		\chi_n (\mathbf{x})=\begin{cases}
		1,\quad \mathbf{x}\in\Omega^{n,l},\\
		0,\quad \text{otherwise},
		\end{cases}
		\end{equation*}
	\end{minipage}
	\begin{minipage}{0.5\linewidth}
		\centering
		\begin{equation*}
		\chi^{\star} (\mathbf{x})=\begin{cases}
		1,\quad \mathbf{x}\in\Omega^{\star},\\
		0,\quad \text{otherwise}.
		\end{cases}
		\end{equation*}
	\end{minipage}
	By using the fact that functions $\bu_n$ are extended by zero outside $\Omega^{n,l},$ we easily obtain that $\bu$ is zero outside $\Omega^{\star}.$ More precisely:
	\begin{equation*}
	(1-\chi^{\star})\bu=\lim_{n}(1-\chi_n)\bu_n=0.
	\end{equation*}
	Next we  show that $\bu$ is zero in $\Omega^{\star}.$  
	We start by recalling the 
	definition of appropriate test functions, since we want to use
	the condition $\|h_n\|_{Q_n'} = \|(\bu_n,\bv_n,\bk_n,\bz_n)\|_{Q_n'}\to 0$ from \eqref{condition} to prove  $\bu=0$ in $\Omega^{\star}$:
	\begin{equation*}
	Q_n=\{(\bups,\bpsi,\bxi,\bzeta)\in (V_F(\Omega^{n,l})\cap H^5(\Omega^{n,l}))\times V_K\times V_S\times V_S: (\bups\circ\bphi^n)|_{\Gamma}\circ\bvarphi=\bpsi,\bpsi\circ\bpi=\bxi\}.
	\end{equation*}
	We will now consider ``special'' test functions, which will help us conclude the desired result.
	Take $\bups\in V_F(\Omega_M)\cap H^5(\Omega_M)$ and consider the test function $\bpsi$ for the shell velocity such that $\bpsi=(\bups\circ\bphi^{\bet^{\star}})|_{\Gamma}\circ\bvarphi$, consider $\bxi$ the test function for the mesh velocity
	such that $\bxi=\bpsi\circ\bpi$, and consider $\bzeta$ the test function for the mesh rotation velocity to be an arbitrary 
	function $\bzeta\in H^{-s}(\mathcal{N}).$ It is then clear that
	$(\bups,\bpsi,\bxi,\bzeta)\in H.$ By the density of $Q_n$ in $H,$ and 
	by the uniform convergence of $\bet_n,$
	we can rewrite the duality pairing on $H$ in the following way:
	\begin{align}\label{H_duality}
	\begin{split}
	\langle (\bu,\bv,\bk,\bz),(\bups,\bpsi,\bxi,\bzeta)\rangle_H&=\lim_n \langle (\bu_n,\bv_n,\bk_n,\bz_n),(\bups,\bpsi_n,\bxi_n,\bzeta)\rangle_H\\
	&=\lim_n\,_{Q_n'}\langle (\bu_n,\bv_n,\bk_n,\bz_n),(\bups,\bpsi_n,\bxi_n,\bzeta)\rangle_{Q_n}\\
	&\leq \|(\bu_n,\bv_n,\bk_n,\bz_n)\|_{Q_n'}\|(\bups,\bpsi_n,\bxi_n,\bzeta)\|_{Q_n},
	\end{split}
	\end{align} 
	where $\bpsi_n=(\bups\circ\bphi^n)|_{\Gamma}\circ\bvarphi$ and $\bxi_n=\bpsi_n\circ\bpi.$
	Since $\|(\bu_n,\bv_n,\bk_n,\bz_n)\|_{Q_n'}\to 0$ and $\|(\bups,\bpsi_n,\bxi_n,\bzeta)\|_{Q_n}\leq C,$ we obtain that
	$\langle (\bu,\bv,\bk,\bz),(\bups,\bpsi_n,\bxi_n,\bzeta)\rangle_H=0.$
	
	Now, take a test function $(\bups,0,0,0)\in Q_n$ such that $\text{supp}\,\bups\subset \Omega^{\star}.$ Using the uniform convergence of the sequence $\bet_n,$ we can find an $n_0$ such that $\chi_n(\mathbf{x})=\chi^{\star}(\mathbf{x})=1,\; \forall n\geq n_0,\;\mathbf{x}\in\text{supp}\,\bups.$ Therefore, we have
	\begin{align*}
	0=\langle (\bu,\bv,\bk,\bz),(\bups,0,0,0)\rangle_H=\lim_n\int_{\Omega_M}\chi_n\bu_n\cdot\bups=\int_{\Omega_M}\chi^{\star}\bu\cdot\bups =\int_{\Omega^{\star}}\bu\cdot\bups,
	\end{align*}
	i.e. $\bu=0$ in $\Omega^{\star}.$ 
	
	To see that $\bv=0$ in $\omega$ and $\bk = 0$ in $\mathcal{N},$ we take the test function $(\bups,\bpsi_n,\bxi_n,0)\in Q_n,$
	such that $\text{supp}\,\bpsi_n\subset \omega\setminus\omega_S$ and $\bpsi_n,\bxi_n$ satisfy the aforementioned coupling conditions. 
    Then we have
	\begin{align*}
	0=\langle (\bu,\bv,\bk,\bz),(\bups,\bpsi,\bxi,0)\rangle_H=\int_{\Omega_M}\bu\cdot\bups +\lim_n \int_{\omega}\bv_n\cdot \bpsi_n+\lim_n\sum_{i=1}^{n_E}\int_0^{l_i}(\bk_n)_i\cdot(\bxi_n)_i.
	\end{align*}
	Since $\bu=0$ in $\Omega_M$ and $\bxi_n=0$ in $\mathcal{N},$ we obtain
	that $\bv=0$ in $\omega.$ Now take the same test function, but without 
	a restriction on the support of $\bpsi_n$ to see that $\bk=0$ in $\mathcal{N}.$
	Finally, to show that $\bz=0$ in $\mathcal{N},$ take the test function $(\bups,\bpsi_n,\bxi_n,\bzeta)\in Q_n,$ calculate the scalar product, and
	use the just obtained result that $\bu=0,\bv=0$ and $\bk=0.$
	
	To conclude, we have shown that $\bh=0$ in $H,$ which is
	in contradiction with the assumption $\|\bh_n\|_H=1$,
	which implies that the uniform 
	Ehrling property is satisfied by the sequence of approximate solutions.
	\\[3mm]
	\noindent
	\textbf{Conclusion.} We have verified all the assumptions from Theorem~\ref{thm:compact}. 
	Therefore, $\{(\bu_N,\bv_N,\bk_N,\bz_N)\}_{N=1}^\infty$ is relatively compact in $L^2(0,T;H).$
\end{proof}
We summarize the strong convergence results obtained in Sec.~\ref{subsec:conv_shell} and Theorem~\ref{thm:strong}. We have shown that there exist subsequences $(\bu_N)_{N\in\N}, (\be_N)_{N\in\N}, (\bet_N)_{N\in\N},$ $ (\bv_N)_{N\in\N},(\bk_N)_{N\in\N},(\bz_N)_{N\in\N}$ such that
\begin{align}\label{strong_conv}
\begin{split}
\bu_N\to\bu&\text{ in } L^2(0,T;L^2(\Omega_M)),\\
\tau_{\dt}\buh_N\to \bu &\text{ in } L^2(0,T;L^2(\Omega_M)),\\
\be_N\to\be & \text{ in } C([0,T];W^{1,\infty}(\omega)),\\
\bet_N\to\bet & \text{ in } C([0,T];W^{1,\infty}(\omega)),\\
\bv_N\to\bv &\text{ in } L^2(0,T;L^2(\omega)),\\
\tau_{\dt}\bv_N\to\bv &\text{ in } L^2(0,T;L^2(\omega)),\\
\bk_N\to\bk &\text{ in } L^2(0,T;H^{-s}(\mathcal{N})),\\
\bz_N\to\bz &\text{ in } L^2(0,T;H^{-s}(\mathcal{N})).
\end{split}
\end{align}
The statements about convergence of $(\tau_{\dt}\buh_N)_{N\in\N}$ and $(\tau_{\dt}\bv_N)_{N\in\N}$ follow directly from Statement 3 of Theorem~\ref{thm:unif}. We conclude this section by stating one last convergence result that will be used in the next section to prove that the limiting functions satisfy the weak formulation \eqref{weak} of the full FSI problem, i.e.
\begin{equation*}
\tau_{\dt}\bu_N\to\bu\text{ in }L^2(0,T;L^2(\Omega_M)).
\end{equation*}

\section{The limiting problem and the main result}
We want to show that the limiting functions satisfy the weak formulation \eqref{weak} of the full fluid-mesh-shell interaction problem. 
For this purpose, we  consider the weak formulations of the coupled semi-discretized problems,
and take the limit as $N\to\infty$, or $\dt\to 0$. 
The strong convergence results, obtained in the previous section, are crucial in this step. 
Unfortunately, there is one more obstacle that needs to be overcome before we can pass to the limit:
the velocity test functions in the semi-discretized problems depend on the fluid domains, 
and so passing to the limit in the semi-discretized weak formulations requires special care.
To deal with this problem we plan to construct {\emph{appropriate}} divergence-free test functions, 
whose dependence on $N$ can be controlled.

\subsection{Construction of the appropriate test functions}
We begin by recalling that the test functions $(\bups,\bpsi,\bxi,\bzeta)$ for the limiting problem are defined by the test space $\mathcal{Q}(0,T)$ in \eqref{test_space_Q},
 which depends on $\be.$ Similarly, the test spaces for the approximate problems depend on $N$ through the dependence on $\be_N.$ The fact that the velocity test functions depend on $N$ presents a technical difficulty when passing to the limit as $N\to\infty.$ For that reason, our goal is to construct the test functions, 
 both for the limiting problem and for the approximate problems, which are smooth, divergence-free, and are such that 
their dependence on $N$ can be controlled. 

Moreover, to pass to the limit as $N \to \infty$, it will be easier to work on the maximal domain $\Omega_M$.
Therefore, we will need the test functions to also be defined
on the maximal domain. In fact, we will construct the test functions on $\Omega_M$ to consist of two parts.
One with compact support in the 
given fluid domain;
such test functions can be extended to $\Omega_M$ by their  zero trace on the boundary,
and the other part which will handle the information about the boundary data. 
The construction of the second part is crucial to be able to control the behavior of the test functions 
in terms of $N$, and obtain uniform convergence results that will allow us to pass to the limit. 

We first deal with the test functions
that handle the information about the fluid domain boundary. 
To construct the divergence-free, smooth test functions that can handle the nonzero boundary data,
we rely on the approach similar to that used in Lemma~\ref{lemma:div}. Namely, we construct smooth 
extensions to $\Omega_M$ of the test functions defined on the fluid domain boundary,
namely of the test functions corresponding to the Koiter shell problem,
 and then ``correct'' the extensions so that the resulting functions are divergence-free. 
\vskip 0.1in
\noindent
{\bf{Extensions of the Koiter shell test functions to $\Omega_M$.}}
We start by taking a test function $\bpsi\in C_c^1([0,T);H^2(\omega))$ for the Koiter shell problem,
and then construct an extension
to $\Omega_M$, denoted by $\bupst$,  such that the extension $\bupst$ has the property that its trace on $\Gamma$ is $\bpsi$: $\bupst|_{\Gamma}\circ\bvarphi=\bpsi$. 
Notice that $\bupst\in C_c^1([0,T);H^2(\Omega_M))$.

Using the test functions $\bupst$, which are independent of $n$ and $N$, we now construct the test functions
for the fluid velocities defined on approximate domains $\Omega^n$, and also for the test functions defined
on the continuous (limiting) domain $\Omega^{\bet}$.
The test functions associated with the approximate domains $\Omega^n$ are defined as follows:
\begin{equation*}
\bups_N^n=\begin{cases}
\bupst\circ(\bphi^n)^{-1},\text{ in }\Omega^n,\\
\bupst\circ(\tilde{\bphi}^n)^{-1},\text{ in }\Omega_M\setminus\Omega^n,
\end{cases}
\end{equation*}
where the mapping $\tilde{\bphi}^n$ is an extension of the mapping $\bphi^n$ to the maximal domain $\Omega_M$, as introduced in Lemma~\ref{lemma:div}.
It is easy to check that $\bups_N^n\in H^1(\Omega_M),\forall n = 1,\dots, N.$ 

The test functions associated with the continuous (limiting) domain $\Omega^{\bet}$ are defined as follows:
\begin{equation*}
\bups=\begin{cases}
\bupst\circ(\bphi^{\bet})^{-1},\text{ in }\Omega^{\bet}(t),\\
\bupst\circ(\tilde{\bphi}^{\bet})^{-1},\text{ in }\Omega_M\setminus\Omega^{\bet}(t),
\end{cases}
\end{equation*}
where the mapping $\tilde{\bphi}^{\bet}$ is an extension of the mapping $\bphi^{\bet}$ to the maximal domain $\Omega_M$,  
as discussed in Lemma~\ref{lemma:div}.
It is clear that $\bups\in H^1(\Omega_M).$

We emphasize that the test functions $\bups_N^n$ and $\bups$ depend on the choice of the test function $\bpsi$. 
However, for simplicity, we will not be explicitly denoting that dependence.

\begin{prop}\label{convergence_test_functions}
The test functions $\bups_N^n$ constructed above have the following convergence properties:
\begin{equation*}
\bups_N\to\bups \text{ uniformly in  } (0,T)\times \Omega_M,
\end{equation*}
\begin{equation*}
\nabla\bups_N\to\nabla\bups\text{ in } L^2(0,T;L^p(\Omega_M)),\quad p < \infty.
\end{equation*}  
\end{prop}

\proof
From uniform convergence of $\bet_N,$ we obtain that
\begin{equation}\label{test_fun_conv}
\bups_N\to\bups \text{ uniformly in  } (0,T)\times \Omega_M,
\end{equation}
where $\bups_N=(\bups_N^1,\bups_N^2,\dots,\bups_N^N).$
Using the chain rule, and the fact that $\nabla\bet_N\to\nabla\bet$ in $L^2(0,T;L^p(\omega)),$ one can see that 
\begin{equation*}
\nabla\bups_N\to\nabla\bups\text{ in } L^2(0,T;L^p(\Omega_M)),\quad p < \infty.
\end{equation*}  
\qed

While the functions $\bups_N$ have good spatial regularity and convergence properties, 
they are discontinuous in time at points $n\dt$, since they are defined via the mappings
$\bphi^n(z,r,\theta)=\bphi^{\bet}(n\dt,z,r,\theta)=(z,(R+\et_N^n)r,\theta)$,
which are step functions in time.
For that reason, we introduce  {\emph{linear, continuous extensions}}  $\bupsl_N$ of the test functions $\bups_N^n$, 
 on each subinterval $[(n-1)\dt,n\dt], n=1,\dots,N,$ and such that
\begin{equation*}
\bupsl_N(n\dt,\cdot)=\bups_N(n\dt,\cdot).
\end{equation*}
Using strong convergence of approximate shell velocities in $L^2(0,T;L^2(\omega)),$ we get that
\begin{equation*}
\partial_t\bupsl_N\to\partial_t\bups\text{ in } L^2(0,T;L^p(\Omega_M)),\quad p<2.
\end{equation*} 

Unfortunately, $\bups_N$, $\bupsl_N$, and $\bups$ are not necessarily divergence-free. 
This is why we need to ``correct'' the construction of the appropriate test functions, 
in a way similar to the proof of Lemma \ref{lemma:div}.
\vskip 0.1in
\noindent
{\bf{Construction of divergence-free correction.}}
We plan to use Theorem III.3.1. from \cite{GaldiNS}, just like in  Lemma \ref{lemma:div}.
We construct the corrections to the velocity test functions $\bups^n_N$
via two functions: $\bw_N^n$ and ${\bf v}_N^n$. The function ${\bf v}_N^n$ will be obtained
as a solution of problem \eqref{galdi}, and $\bw_N^n$ will be constructed so that the compatibility conditions
necessary for the existence of ${\bf v}_N^n$ are satisfied. 
There is a slight difference with the proof of Lemma~\ref{lemma:div}. Since we do not already have a divergence-free
function inside $\Omega^n$, as we did in Lemma~\ref{lemma:div}, we need to define the divergence-free correction both
inside $\Omega^n$ and in its complement $\Omega_M\setminus\Omega^n$. The correction functions in the complement
$\Omega_M\setminus\Omega^n$ will have an extra "tilde" notation: $\tilde\bw_N^n$ and $\tilde{\bf v}_N^n$. 
The same approach will be used for the construction of the divergence-free correction of $\bups$. 

More precisely, we define the mappings ${\bw}_N^n:\Omega^n\to \R^3$ satisfying the following conditions:
\begin{enumerate}[(i)]
	\item $\text{supp}\,{\bw}_N^n \subseteq \Omega^n,$
	\item $\dis\int_{\Omega^n} \nabla\cdot (\bups_N^n +{\bw}_N^n)=0,$
	\item $\|{\bw}_N^n\|_{H^1(\Omega^n)}\leq C\|\bups_N^n\|_{H^1(\Omega_M)},$
\end{enumerate} 
and the mappings ${\tilde\bw}_N^n:(\Omega_M\setminus\Omega^n)\to\R^3$ satisfying the following conditions:
\begin{enumerate}[(i)]
	\item $\text{supp}\,{\tilde\bw}_N^n \subseteq \Omega_M\setminus\Omega^n,$
	\item $\dis\int_{\Omega_M\setminus\Omega^n} \nabla\cdot (\bups_N^n+{\tilde\bw}_N^n)=0,$
	\item $\|{\tilde\bw}_N^n\|_{H^1(\Omega_M\setminus\Omega^n)}\leq C\|\bups_N^n\|_{H^1(\Omega_M)},$
\end{enumerate} 
in the same way as in Lemma~\ref{lemma:div}. Furthermore, for the same reason as in Lemma~\ref{lemma:div} 
we conclude that we can decompose both $\Omega^n$ and $\Omega_M\setminus\Omega^n$ 
in a finite number of  star-shaped domains (with respect to fixed balls), where the number depends only of the uniform
Lipschitz constant, and not on $n$ or $N$. 
We can now apply Theorem III.3.1. from \cite{GaldiNS} to conclude that there exists a function ${\bf v}_N^n$ such that:
\begin{equation*}
\nabla\cdot{\bf v}_N^n = - \nabla\cdot (\bups_N^n+{\bw}_N^n)
\end{equation*}
with
\begin{equation*}
\|{\bf v}_N^n\|_{H^1(\Omega_M)}\leq C \|\bups_N^n+{\bw}_N^n\|_{H^1(\Omega_M)}\leq C \|\bups_N^n\|_{H^1(\Omega_M)},
\end{equation*}
and a function $\tilde{\bf v}_N^n$ such that:
\begin{equation*}
\nabla\cdot\tilde{\bf v}_N^n = - \nabla\cdot (\bups_N^n+\tilde{\bw}_N^n)
\end{equation*}
with
\begin{equation*}
\|\tilde{\bf v}_N^n\|_{H^1(\Omega_M)}\leq C \|\bups_N^n+\tilde{\bw}_N^n\|_{H^1(\Omega_M)} \leq C \|\bups_N^n\|_{H^1(\Omega_M)}.
\end{equation*}
Finally, if we set
\begin{equation}\label{test_functions_N}
\bups_N(\bpsi)=
\begin{cases}
\bups_N^n+{\bw}_N^n+{\bf v}_N^n,\text{ in } \Omega^n,\\
\bups_N^n+\tilde{\bw}_N^n+\tilde{\bf v}_N^n,\text{ in } \Omega_M\setminus\Omega^n,
\end{cases}
\end{equation}
we have that $\bups_N(\bpsi)$ is a smooth, divergence-free function on the maximal domain $\Omega_M.$ In the same way, we can construct a divergence-free extension $\bups(\bpsi)$ of the test function $\bups$ corresponding to the limiting problem:
\begin{equation}\label{test_functions}
\bups(\bpsi)=
\begin{cases}
\bups+{\bw}+{\bf v},\text{ in } \Omega^{\bet}(t),\\
\bups+\tilde{\bw}+\tilde{\bf v},\text{ in } \Omega_M\setminus\Omega^{\bet}(t),\\
\end{cases}
\end{equation}
with
\begin{align*}
\|{\bw}\|_{H^1(\Omega^{\bet}(t))},\|{\bf v}\|_{H^1(\Omega^{\bet}(t))}\leq C\|\bups\|_{H^1(\Omega_M)},\\
\|\tilde{\bw}\|_{H^1(\Omega_M\setminus\Omega^{\bet}(t))},\|\tilde{\bf v}\|_{H^1(\Omega_M\setminus\Omega^{\bet}(t))}\leq C\|\bups\|_{H^1(\Omega_M)}.
\end{align*}
The following convergence results hold:
\begin{prop}
The test functions \eqref{test_functions_N} and \eqref{test_functions} constructed above, have the following convergence properties:
\begin{equation*}
\bups_N(\bpsi)\to \bups(\bpsi)\text{  uniformly in } (0,T)\times\Omega_M,
\end{equation*}
\begin{align*}
\nabla\bups_N(\bpsi)\to \nabla\bups(\bpsi)&\text{ in } L^2(0,T;L^p(\Omega_M)),\quad p < \infty,\\
\partial_t\bupsl_N(\bpsi)\to \partial_t\bups (\bpsi)&\text{ in } L^2(0,T;L^p(\Omega_M)),\quad p<2.
\end{align*}
\end{prop}
\proof
Due to the fact that ${\bf v}_N^n, {\bf v},\tilde{\bf v}_N^n,\tilde{\bf v}$ are solutions of equation \eqref{galdi}, 
with the right-hand sides in $L^p(\Omega_M)$, given explicitly by $-\nabla\cdot(\bups_N^n+{\bw}_N^n),-\nabla\cdot(\bups+{\bw}),-\nabla\cdot(\bups_N^n+\tilde{\bw}_N^n)$,
and $-\nabla\cdot(\bups+\tilde{\bw})$,
respectively, we can write their explicit formulas by using Bogowskii construction, see \cite{GaldiNS}.
Theorem~III.3.3 in \cite{GaldiNS}  provides additional regularity of ${\bf v}_N^n, {\bf v},\tilde{\bf v}_N^n,\tilde{\bf v}$: 
\begin{align*}
\|\bups_N(\bpsi)-\bups(\bpsi)\|_{W^{1,p}(\Omega_M)}&=\|\bups_N^n+{\bw}_N^n+{\bf v}_N^n-\bups-{\bw}-{\bf v}\|_{W^{1,p}(\Omega^n)}\\
&+\|\bups_N^n+\tilde{\bw}_N^n+\tilde{\bf v}_N^n-\bups-\tilde{\bw}-\tilde{\bf v}\|_{W^{1,p}(\Omega_M\setminus\Omega^n)}\\
&\leq \|(\bups_N^n +{\bw}_N^n)-(\bups+{\bw})\|_{W^{1,p}(\Omega^n)}+\|{\bf v}_N^n-{\bf v}\|_{W^{1,p}(\Omega^n)}\\
&+\|(\bups_N^n +\tilde{\bw}_N^n)-(\bups+\tilde{\bw})\|_{W^{1,p}(\Omega_M\setminus\Omega^n)}+\|\tilde{\bf v}_N^n-\tilde{\bf v}\|_{W^{1,p}(\Omega_M\setminus\Omega^n)}.
\end{align*}
Due to the 
uniform convergence of $\bups_N\to \bups$ one obtains that the right-hand side tends to 0. Furthermore, using the Sobolev embedding of $W^{1,p}(\Omega_M)$ to $C(\overline{\Omega}_M),$ for $p>3$, we obtain that
\begin{equation*}
\bups_N(\bpsi)\to \bups(\bpsi)\text{  uniformly in } (0,T)\times\Omega_M.
\end{equation*}
Additionally, by using Remark~III.3.3 from \cite{GaldiNS}, we can show that
\begin{align*}
\nabla\bups_N(\bpsi)\to \nabla\bups(\bpsi)&\text{ in } L^2(0,T;L^p(\Omega_M)),\quad p < \infty,\\
\partial_t\bupsl_N(\bpsi)\to \partial_t\bups (\bpsi)&\text{ in } L^2(0,T;L^p(\Omega_M)),\quad p<2.
\end{align*}
This completes the proof.
\qed

\vskip 0.1in
\noindent
{\bf{Approximation of the test functions in $\mathcal{Q}(0,T)$.}}
We now show how the restrictions to $\Omega^{\bet}(t)$ of the test functions $\bups(\bpsi)$ constructed above, 
can be used to construct admissible test functions for the continuous problem. Such test functions will be dense in $\mathcal{Q}(0,T)$. 
These are the test functions which will be used in the proof of convergence to the weak solution, discussed in the next section. 
Similarly, the  restrictions to $\Omega^n$ of the test functions $\bups_N(\bpsi)$ constructed above, will be used to construct 
admissible test functions for the approximate problems, and  to study convergence to a weak solution.

More precisely, for any test function $(\bups,\bpsi,\bxi,\bzeta)\in\mathcal{Q}(0,T),$ the fluid velocity component $\bups$ can be written as $\bups-\bups(\bpsi)+\bups(\bpsi)$,
 where $\bups-\bups(\bpsi)$ can be approximated by a divergence-free function $\bups_0$, which has compact support in $\Omega^{\bet}(t)\cup\Gamma_{in}\cup\Gamma_{out}.$
Therefore, one can easily see that the functions 
$$(\bups,\bpsi,\bxi,\bzeta)=(\bups_0+\bups(\bpsi),\bpsi,\bxi,\bzeta)$$
 are dense in $\mathcal{Q}(0,T)$,
and are such that  $\nabla\cdot\bups=0$. Thus, the test functions $\bups$ are decomposed into two parts. The part $\bups_0$ which collects information 
about the solution in the interior of the fluid domain, and the part $\bups(\bpsi)$ which takes care of the boundary data. 
As we shall see in the next section, we will be working with weak formulations in physical space defined on  the maximal domain $\Omega_M$, 
which is the reason why the appropriate test functions, constructed above, are all defined on $\Omega_M$. 

 The corresponding test functions for approximate problems have the same form, i.e. 
 \begin{equation}\label{test_functions_full}
 (\bups_N,\bpsi,\bxi,\bzeta)=(\bups_0+\bups_N(\bpsi),\bpsi,\bxi,\bzeta).
 \end{equation}
 These functions will be used to study convergence  to a weak solution, defined on $\Omega^{\bet}(t)$.

\subsection{Passing to the limit}
We start by deriving the weak formulation of the coupled, semi-discretized problem, in the form which will be convenient to pass to the limit,
and obtain the weak formulation \eqref{weak} of the continuous, coupled problem. 
\vskip 0.1in
\noindent
{\bf{A weak formulation of the coupled, semi-discretized problem.}}
Let  $(\bups_N,\bpsi,\bxi,\bzeta)$ be the test functions \eqref{test_functions_full} constructed above. 
Use $(\bpsi(t),\bxi(t),\bzeta(t))$ as the test function in the structure subproblem \eqref{sub_struct}, and integrate with respect to $t$ from $n\dt$ to $(n+1)\dt.$ 
Then, take $(\bups_N(t),\bpsi(t))$ as the test functions in the fluid subproblem \eqref{sub_fluid}, and integrate over the same time interval. Add the two equations together to obtain:
\begin{align*}
&\dis \rho_F \int_{n\dt}^{(n+1)\dt}\int_{\Omega^{n+1}}\frac{\bu_N^{n+1}-\buh_N^{n}}{\dt}\cdot \bups_N ^{n+1} + \frac{\rho_F}{2}\int_{n\dt}^{(n+1)\dt}\int_{\Omega^{n+1}}(\nabla\cdot\bs)(\buh_N^n\cdot\bups_N^{n+1})\\
&+ \frac{\rho_F}{2}\int_{n\dt}^{(n+1)\dt}\int_{\Omega^{n+1}}\big[((\buh_N^n-\bs)\cdot\nabla)\bu_N^{n+1}\cdot \bups_N^{n+1}- ((\buh_N^n-\bs)\cdot\nabla)\bups_N^{n+1}\cdot \bu_N^{n+1} \big]\\
&+ 2\mu_F \int_{n\dt}^{(n+1)\dt}\int_{\Omega^{n+1}}\bD(\bu_N^{n+1}):\bD(\bups_N^{n+1})
+\rho_K h \int_{n\dt}^{(n+1)\dt}\int_{\omega}\frac{\bv_N^{n+1}-\bv_N^n}{\dt}\cdot\bpsi R\\
&+\int_{n\dt}^{(n+1)\dt}a_K(\be_N^{n+1},\bpsi)+\rho_S\int_{n\dt}^{(n+1)\dt}\sum_{i=1}^{n_E}A_i\int_{0}^{l_i}\frac{(\bk_N^{n+1})_i-(\bk_N^n)_i}{\dt}\cdot\bxi_i\\
&+\rho_S\int_{n\dt}^{(n+1)\dt}\sum_{i=1}^{n_E}\int_{0}^{l_i}M_i\frac{(\bz_N^{n+1})_i-(\bz_N^n)_i}{\dt}\cdot\bzeta_i
+\int_{n\dt}^{(n+1)\dt}a_S(\bw_N^{n+1},\bzeta)\\
&=\dis\int_{n\dt}^{(n+1)\dt} P_{in}^n\int_{\Gamma_{in}}\upsilon_z - \int_{n\dt}^{(n+1)\dt}P_{out}^n\int_{\Gamma_{out}}\upsilon_z.
\end{align*}
After taking the sum from $n=0,\dots,N-1$ we obtain:
\begin{align*}
&\dis \rho_F\sum_{n=0}^{N-1} \int_{n\dt}^{(n+1)\dt}\int_{\Omega^{n+1}}\frac{\bu_N^{n+1}-\buh_N^{n}}{\dt}\cdot \bups_N^{n+1}
+\frac{\rho_F}{2}\sum_{n=0}^{N-1}\int_{n\dt}^{(n+1)\dt}\int_{\Omega^{n+1}}(\nabla\cdot\bs)(\buh_N^n\cdot\bups_N^{n+1})\\
&+\frac{\rho_F}{2}\sum_{n=0}^{N-1}\int_{n\dt}^{(n+1)\dt}\int_{\Omega^{n+1}}\big[((\buh_N^n-\bs)\cdot\nabla)\bu_N^{n+1}\cdot \bups_N^{n+1}- ((\buh_N^n-\bs)\cdot\nabla)\bups_N^{n+1}\cdot \bu_N^{n+1} \big]\\
&+ 2\mu_F \sum_{n=0}^{N-1}\int_{n\dt}^{(n+1)\dt}\int_{\Omega^{n+1}}\bD(\bu_N^{n+1}):\bD(\bups_N^{n+1})
+\rho_K h \sum_{n=0}^{N-1}\int_{n\dt}^{(n+1)\dt}\int_{\omega}\frac{\bvl_N-\tau_{\dt}\bvl_N}{\dt}\cdot\bpsi R\\
&+\sum_{n=0}^{N-1}\int_{n\dt}^{(n+1)\dt}a_K(\be_N(t),\bpsi)+\rho_S\sum_{n=0}^{N-1}\int_{n\dt}^{(n+1)\dt}\sum_{i=1}^{n_E}A_i\int_{0}^{l_i}\frac{(\bkl_N)_i-\tau_{\dt}(\bkl_N)_i}{\dt}\cdot\bxi_i\\
&+\rho_S\sum_{n=0}^{N-1}\int_{n\dt}^{(n+1)\dt}\sum_{i=1}^{n_E}\int_{0}^{l_i}M_i\frac{(\bzl_N)_i-\tau_{\dt}(\bzl_N)_i}{\dt}\cdot\bzeta_i
+\sum_{n=0}^{N-1}\int_{n\dt}^{(n+1)\dt} a_S(\bw_N(t),\bzeta)\\
&=\dis\sum_{n=0}^{N-1}\int_{n\dt}^{(n+1)\dt} P_{in}^N\int_{\Gamma_{in}}\upsilon_z - \sum_{n=0}^{N-1}\int_{n\dt}^{(n+1)\dt} P_{out}^N\int_{\Gamma_{out}}\upsilon_z,
\end{align*}
where we have used the definition of $\be_N$ and $\bw_N$ as piecewise constant approximations,
defined in \eqref{approx_sol}, and the definition of $\bvl_N, \bkl_N$ an $\bzl_N$ as piecewise linear approximations defined in \eqref{linear_approx}.

%%%%%%%%%

 The terms that include the shell and mesh unknowns can be written as
\begin{align*}
&\rho_K h \int_0^T\int_{\omega}\partial_t\bvl_N\cdot\bpsi R
+\int_0^Ta_K(\be_N,\bpsi)+\rho_S\int_0^T\sum_{i=1}^{n_E}A_i\int_{0}^{l_i}\partial_t(\bkl_N)_i\cdot\bxi_i\\
&+\rho_S\int_0^T\sum_{i=1}^{n_E}\int_{0}^{l_i}M_i\partial_t(\bzl_N)_i\cdot\bzeta_i
+\int_0^T a_S(\bw_N,\bzeta),
\end{align*}
and integration by parts with respect to time gives:
\begin{align*}
&-\rho_K h \int_0^T\int_{\omega}\bvl_N\cdot\partial_t\bpsi R -\rho_K h\int_{\omega} \bv_0\cdot\bpsi(0) R
+\int_0^Ta_K(\be_N,\bpsi)\\
&-\rho_S\int_0^T\sum_{i=1}^{n_E}A_i\int_{0}^{l_i}(\bkl_N)_i\cdot\partial_t\bxi_i - \rho_S \sum_{i=1}^{n_E}A_i \int_0^{l_i}\bk_{0i}\cdot\bxi_i(0)\\
&-\rho_S\int_0^T\sum_{i=1}^{n_E}\int_{0}^{l_i}M_i(\bzl_N)_i\cdot\partial_t\bzeta_i- \rho_S \sum_{i=1}^{n_E} \int_0^{l_i} M_i\bz_{0i}\cdot\bzeta_i(0)
+\int_0^T a_S(\bw_N,\bzeta).
\end{align*}

To deal with the fluid part in the weak formulation, we recall the
characteristic functions, introduced in \eqref{char_fun}, which will 
enable us to rewrite the integrals corresponding to the fluid part
over the maximal domain $\Omega_M$.
We set $\chi_N(t,\cdot)=\chi_N^{n+1},$ for $t\in(n\dt,(n+1)\dt],$ and write the integrals over $\Omega^{n+1}$ as:
\begin{align*}
&\rho_F\sum_{n=0}^{N-1} \int_{n\dt}^{(n+1)\dt}\int_{\Omega_M}\chi_N^{n+1}\frac{\bu_N^{n+1}-\buh_N^{n}}{\dt}\cdot \bups_N^{n+1}
+\frac{\rho_F}{2}\sum_{n=0}^{N-1}\int_{n\dt}^{(n+1)\dt}\int_{\Omega_M}\chi_N^{n+1}(\nabla\cdot\bs)(\buh_N^n\cdot\bups_N^{n+1})\\
&+\frac{\rho_F}{2}\sum_{n=0}^{N-1}\int_{n\dt}^{(n+1)\dt}\int_{\Omega_M}\chi_N^{n+1}\big[((\buh_N^n-\bs)\cdot\nabla)\bu_N^{n+1}\cdot\bups_N^{n+1}
- ((\buh_N^n-\bs)\cdot\nabla)\bups_N^{n+1}\cdot \bu_N^{n+1} \big]\\
&+ 2\mu_F \sum_{n=0}^{N-1}\int_{n\dt}^{(n+1)\dt}\int_{\Omega_M}\chi_N^{n+1}\bD(\bu_N^{n+1}):\bD(\bups_N^{n+1}).
\end{align*}
We simplify (rewrite) each term separately.

In the first term we add and subtract $\bu_N^n$ from the numerator to obtain:
\begin{equation}\label{limit1}
\dis \rho_F\sum_{n=0}^{N-1} \int_{n\dt}^{(n+1)\dt}\int_{\Omega_M}\frac{\bu_N^{n+1}-\bu_N^{n}}{\dt}\cdot \chi_N^{n+1}\bups_N^{n+1}
+ \rho_F\sum_{n=0}^{N-1} \int_{n\dt}^{(n+1)\dt}\int_{\Omega_M}\frac{\bu_N^{n}-\buh_N^{n}}{\dt}\cdot \chi_N^{n+1}\bups_N^{n+1}.
\end{equation}
We now use the summation by parts formula (discrete analogue of the integration by parts formula) to take care of the first term in \eqref{limit1}:
\begin{align*}
&\rho_F\sum_{n=0}^{N-1} \int_{n\dt}^{(n+1)\dt}\int_{\Omega_M}\frac{\bu_N^{n+1}-\bu_N^{n}}{\dt}\cdot \chi_N^{n+1}\bups_N^{n+1} =
\rho_F\int_{\Omega_M}\bu_N^N\cdot\chi_N^N\bups_N^N - \rho_F\int_{\Omega_M}\bu_N^0\cdot\chi_N^1\bups_N^1\\ 
&\quad-\rho_F\sum_{n=1}^{N-1}\int_{n\dt}^{(n+1)\dt}\int_{\Omega_M}\frac{1}{\dt}\bu_N^n\cdot(\chi_N^{n+1}\bups_N^{n+1}- \chi_N^n\bups_N^n)\\
&= - \rho_F\int_{\Omega_M}\bu_N^0\cdot\chi_N^1\bups_N^1 -\rho_F
\sum_{n=1}^{N-1}\int_{n\dt}^{(n+1)\dt}\int_{\Omega_M}\frac{1}{\dt}\bu_N^n\cdot\chi_N^{n+1}(\bups_N^{n+1}-\bups_N^n)\\
&\quad -\rho_F\sum_{n=1}^{N-1}\int_{n\dt}^{(n+1)\dt}\int_{\Omega_M}\frac{1}{\dt}\bu_N^n\cdot(\chi_N^{n+1}-\chi_N^n)\bups_N^n\\
&= - \rho_F\int_{\Omega_M}\chi_N^1\bu_N^0\cdot\bups_N^1 -\rho_F \int_0^T\int_{\Omega_M}\chi_N\tau_{\dt}\bu_N\cdot\partial_t\bups_N\\
&\quad -\rho_F\sum_{n=1}^{N-1}\int_{n\dt}^{(n+1)\dt}\int_{\omega}\frac{R}{\dt}\int_{R+\et_N^{n}}^{R+\et_N^{n+1}}\bu_N^n\cdot\bups_N^n\\
&= - \rho_F\int_{\Omega_M}\chi_N^1\bu_N^0\cdot\bups_N^1 - \rho_F\int_0^T\int_{\Omega_M}\chi_N\tau_{\dt}\bu_N\cdot\partial_t\bups_N -\rho_F\int_0^T\int_{\omega}\partial_t\et_NR(\tau_{\dt}\bu_N\cdot\tau_{\dt}\bups_N).
\end{align*}
Notice that in the last equality we used the mean value theorem for integrals.

To deal with the second term in \eqref{limit1}, we recall that $\buh_N^n$ was defined 
in \eqref{ale_comp_fun} as a composition of $\bu_N^n$ and $\bA,$ and calculate:
\begin{align*}
&\rho_F\sum_{n=0}^{N-1} \int_{n\dt}^{(n+1)\dt}\int_{\Omega_M}\frac{\bu_N^{n}-\buh_N^{n}}{\dt}\cdot \chi_N^{n+1}\bups_N^{n+1}\\
&=\rho_F\sum_{n=0}^{N-1} \int_{n\dt}^{(n+1)\dt}\int_{\Omega_M}\frac{1}{\dt}\left(\bu_N^n(z,r,\theta)-\bu_N^n(z,\frac{R+\et_N^n}{R+\et_N^{n+1}}r,\theta)\right)\cdot\chi_N^{n+1}\bups_N^{n+1}\\
&= \rho_F\sum_{n=0}^{N-1} \int_{n\dt}^{(n+1)\dt}\int_{\Omega_M}\frac{1}{\dt} 
\left( (\nabla\bu_N^n)\frac{\et_N^{n+1}-\et_N^{n}}{R+\et_N^{n+1}}r\mathbf{e}_r \right)\cdot\chi_N^{n+1}\bups_N^{n+1}\\
&= \rho_F\sum_{n=0}^{N-1} \int_{n\dt}^{(n+1)\dt}\int_{\Omega_M} (\nabla\bu_N^n)\bs\cdot\chi_N^{n+1}\bups_N^{n+1}\\
&=\rho_F\sum_{n=0}^{N-1} \int_{n\dt}^{(n+1)\dt}\int_{\Omega_M} (\bs\cdot\nabla)\bu_N^n\cdot\chi_N^{n+1}\bups_N^{n+1}\\
&=\rho_F \int_0^T\int_{\Omega_M}(\mathbf{s}_N\cdot\nabla)\tau_{\dt}\bu_N\cdot\chi_N\bups_N\\
&=\rho_F \int_0^T\int_{\Omega_M}\chi_N(\mathbf{s}_N\cdot\nabla)\tau_{\dt}\bu_N\cdot\bups_N,
\end{align*}
where
\begin{equation*}
\mathbf{s}_N=\frac{\et_N-\tau_{\dt}\et_N}{\dt(R+\et_N)}r\mathbf{e}_r
=\frac{\partial_t\et_N}{R+\et_N}r\mathbf{e}_r.
\end{equation*}
\noindent
We rewrite the convective part in the following way:
\begin{align}\label{limit2}
\begin{split}
&\frac{\rho_F}{2} \sum_{n=0}^{N-1} \int_{n\dt}^{(n+1)\dt}\int_{\Omega_M}\chi_N^{n+1}\left[(\buh_N^n\cdot\nabla)\bu_N^{n+1}\cdot \bups_N^{n+1} - (\buh_N^n\cdot\nabla)\bups_N^{n+1}\cdot \bu_N^{n+1}  \right]\\
&+\frac{\rho_F}{2} \sum_{n=0}^{N-1} \int_{n\dt}^{(n+1)\dt}\int_{\Omega_M}\chi_N^{n+1}\left[ (\bs\cdot\nabla)\bups_N^{n+1}\cdot \bu_N^{n+1}-(\bs\cdot\nabla)\bu_N^{n+1}\cdot \bups_N^{n+1}\right].
\end{split}
\end{align}
Furthermore, we calculate:
\begin{align*}
&\int_{\Omega_M}\chi_N^{n+1}(\bs\cdot\nabla)\bups_N^{n+1}\cdot \bu_N^{n+1}=\int_{\Gamma_M}\chi_N^{n+1}(\bs\cdot\bn)\bups_N^{n+1}\cdot \bu_N^{n+1}\\
&- \int_{\Omega_M}\chi_N^{n+1}(\nabla\cdot\bs)\bups_N^{n+1}\cdot \bu_N^{n+1}- \int_{\Omega_M}\chi_N^{n+1}(\bs\cdot\nabla)\bu_N^{n+1}\cdot \bups_N^{n+1}.
\end{align*}
By using the definition of $\bs$, given in \eqref{sn}, the boundary term can be rewritten as follows:
\begin{align*}
\int_{\Gamma_M} &\chi_N^{n+1}(\bs\cdot\bn)\bups_N^{n+1}\cdot \bu_N^{n+1}
=\int_{\Gamma^{n+1}}\left(\frac{\et_N^{n+1}-\et_N^n}{\dt(R+\et_N^{n+1})}r\mathbf{e}_r\cdot\bn\right)\bups_N^{n+1}\cdot \bu_N^{n+1}\\
&=\int_{\Gamma}\left(\frac{\et_N^{n+1}-\et_N^n}{\dt}\right)\bups_N^{n+1}\cdot\bu_N^{n+1}
=\int_{\omega}\left(\frac{\et_N^{n+1}-\et_N^n}{\dt}R\right)\bups_N^{n+1}\cdot\bu_N^{n+1}.
\end{align*}
By inserting the previous calculations into \eqref{limit2}, we obtain that the convective term is equal to:
\begin{align*}
&\frac{\rho_F}{2}\int_0^T\int_{\Omega_M} \chi_N\left[(\tau_{\dt}\buh_N\cdot\nabla)\bu_N\cdot \bups_N - (\tau_{\dt}\buh_N\cdot\nabla)\bups_N\cdot \bu_N\right]\\
&\quad+ \frac{\rho_F}{2}\int_0^T\int_{\omega}\partial_t\et_N R\bups_N\cdot\bu_N -\frac{\rho_F}{2}\int_0^T\int_{\Omega_M}\chi_N (\nabla\cdot\mathbf{s}_N)\bups_N\cdot\bu_N\\
&\quad - \rho_F \int_0^T\int_{\Omega_M}\chi_N(\mathbf{s}_N\cdot\nabla)\bu_N\cdot\bups_N.
\end{align*}
In summary, the fluid portion of the coupled, semi-discretized weak formulation now reads:
\begin{align*}
&\rho_F\sum_{n=0}^{N-1} \int_{n\dt}^{(n+1)\dt}\int_{\Omega_M}\frac{\bu_N^{n+1}-\buh_N^{n}}{\dt}\cdot \chi_N^{n+1}\bups_N^{n+1}
+\frac{\rho_F}{2}\sum_{n=0}^{N-1}\int_{n\dt}^{(n+1)\dt}\int_{\Omega_M}\chi_N^{n+1}(\nabla\cdot\bs)(\buh_N^n\cdot\bups_N^{n+1})\\
&+\frac{\rho_F}{2}\sum_{n=0}^{N-1}\int_{n\dt}^{(n+1)\dt}\int_{\Omega_M}\chi_N^{n+1}\big[((\buh_N^n-\bs)\cdot\nabla)\bu_N^{n+1}\cdot\bups_N^{n+1}- ((\buh_N^n-\bs)\cdot\nabla)\bups_N^{n+1}\cdot \bu_N^{n+1} \big]\\
& + 2\mu_F \sum_{n=0}^{N-1}\int_{n\dt}^{(n+1)\dt}\int_{\Omega_M}\chi_N^{n+1}\bD(\bu_N^{n+1}):\bD(\bups_N^{n+1})\\
&= - \rho_F\int_{\Omega_M}\chi_N^1\bu_N^0\cdot\bups_N^1 - \rho_F \int_0^T\int_{\Omega_M}\chi_N\tau_{\dt}\bu_N\cdot\partial_t\bups_N -\rho_F\int_0^T\int_{\omega}\partial_t\et_NR(\tau_{\dt}\bu_N\cdot\tau_{\dt}\bups_N)\\
&+\rho_F\int_0^T\int_{\Omega_M}\chi_N(\mathbf{s}_N\cdot\nabla)
\tau_{\dt}\bu_N\cdot\bups_N
+\frac{\rho_F}{2}\int_0^T\int_{\Omega_M}\chi_N(\nabla\cdot\mathbf{s}_N)\tau_{\dt}\buh_N\cdot\bups_N\\
&+\frac{\rho_F}{2}\int_0^T\int_{\Omega_M} \chi_N \left[(\tau_{\dt}\buh_N\cdot\nabla)\bu_N\cdot \bups_N - (\tau_{\dt}\buh_N\cdot\nabla)\bups_N\cdot \bu_N\right]\\
&+ \frac{\rho_F}{2}\int_0^T\int_{\omega}\partial_t\et_NR\bups_N\cdot\bu_N -\frac{\rho_F}{2}\int_0^T\int_{\Omega_M}\chi_N (\nabla\cdot\mathbf{s}_N)\bups_N\cdot\bu_N\\
&- \rho_F \int_0^T\int_{\Omega_M}\chi_N(\mathbf{s}_N\cdot\nabla)\bu_N\cdot\bups_N+ 2\mu_F \int_{0}^{T}\int_{\Omega_M}\chi_N\bD(\bu_N):\bD(\bups_N).
\end{align*}

By combining the fluid and structure part of the weak formulations, calculated above, we obtain the following result:
%%%%%%%%%%
\begin{prop}\label{weak_form_discretize}
Let  $(\bups_N,\bpsi,\bxi,\bzeta)$ be the test functions \eqref{test_functions_full} constructed above. 
Then, weak solutions $(\bu_N, \boldsymbol\eta_N,\bd_N,\bw_N)$ of the
semi-discretized, coupled problem, stated in Sec.~\ref{splitting}, satisfy the following weak formulation:
\begin{align*}
& - \rho_F \int_0^T\int_{\Omega_M}\chi_N\tau_{\dt}\bu_N\cdot\partial_t\bups_N -\rho_F\int_0^T\int_{\omega}\partial_t\et_NR(\tau_{\dt}\bu_N\cdot\tau_{\dt}\bups_N)\\
&+\rho_F\int_0^T\int_{\Omega_M}\chi_N(\mathbf{s}_N\cdot\nabla)
\tau_{\dt}\bu_N\cdot\bups_N
+\frac{\rho_F}{2}\int_0^T\int_{\Omega_M}\chi_N(\nabla\cdot\mathbf{s}_N)\tau_{\dt}\buh_N\cdot\bups_N\\
&+\frac{\rho_F}{2}\int_0^T\int_{\Omega_M} \chi_N \left[(\tau_{\dt}\buh_N\cdot\nabla)\bu_N\cdot \bups_N - (\tau_{\dt}\buh_N\cdot\nabla)\bups_N\cdot \bu_N\right]\\
&+ \frac{\rho_F}{2}\int_0^T\int_{\omega}\partial_t\et_NR\bups_N\cdot\bu_N -\frac{\rho_F}{2}\int_0^T\int_{\Omega_M}\chi_N (\nabla\cdot\mathbf{s}_N)\bups_N\cdot\bu_N\\
&- \rho_F \int_0^T\int_{\Omega_M}\chi_N(\mathbf{s}_N\cdot\nabla)\bu_N\cdot\bups_N+ 2\mu_F \int_{0}^{T}\int_{\Omega_M}\chi_N\bD(\bu_N):\bD(\bups_N)\\
&-\rho_K h \int_0^T\int_{\omega}\bvl_N\cdot\partial_t\bpsi R 
+\int_0^Ta_K(\be_N,\bpsi) +\int_0^T a_S(\bw_N,\bzeta)\\
&-\rho_S\int_0^T\sum_{i=1}^{n_E}A_i\int_{0}^{l_i}(\bkl_N)_i\cdot\partial_t\bxi_i
-\rho_S\int_0^T\sum_{i=1}^{n_E}\int_{0}^{l_i}M_i(\bzl_N)_i\cdot\partial_t\bzeta_i
\\
& - \rho_S \sum_{i=1}^{n_E}A_i \int_0^{l_i}\bk_{0i}\cdot\bxi_i(0) - \rho_S \sum_{i=1}^{n_E} \int_0^{l_i} M_i\bz_{0i}\cdot\bzeta_i(0)
\\
&-\rho_K h\int_{\omega} \bv_0\cdot\bpsi(0) R
-\rho_F\int_{\Omega_M}\bu_N^0\chi_N^1\cdot\bups_N^1
\\
&=\int_0^T P_{in}^N(t) \int_{\Gamma_{in}}\upsilon_z- \int_0^T P_{out}^N(t)\int_{\Gamma_{out}}\upsilon_z, \quad \forall (\bups_N,\bpsi,\bxi,\bzeta),
\end{align*}
where
$$\bvl_N = \displaystyle{\frac{\tau_{\Delta t} \boldsymbol\eta_N - \boldsymbol\eta_N}{\Delta t}},\ 
 \bkl_N = \displaystyle{\frac{\tau_{\Delta t} \bd_N - \bd_N}{\Delta t}}, \ 
  \bzl_N = \displaystyle{\frac{\tau_{\Delta t} \bw_N - \bw_N}{\Delta t}}.
$$
\end{prop}
%%%%%%%%%%%

\vskip 0.1in
\noindent
{\bf{Passing to the limit.}}
Using the strong convergence results summarized in \eqref{strong_conv},
we can pass to the limit in all the terms to obtain:
\begin{align*}
&-\rho_F\int_{\Omega}\bu_0\cdot\bups(0)-\rho_F\int_0^T\int_{\Omega_M}\chi \bu\cdot\partial_t\bups - \rho_F\int_0^T\int_{\omega}\partial_t\et R (\bu\cdot\bups)\\
&+\rho_F\int_0^T\int_{\Omega_M}\chi(\frac{\partial_t\et}{R+\et}r\mathbf{e}_r\cdot\nabla)
\bu\cdot\bups
+\frac{\rho_F}{2}\int_0^T\int_{\Omega_M}\chi(\nabla\cdot\frac{\partial_t\et}{R+\et}r\mathbf{e}_r)\bu\cdot\bups\\
&+\frac{\rho_F}{2}\int_0^T\int_{\Omega_M} \chi \left[(\bu\cdot\nabla)\bu\cdot \bups - (\bu\cdot\nabla)\bups\cdot \bu\right]\\
&+ \frac{\rho_F}{2}\int_0^T\int_{\omega}\partial_t\et R\bups\cdot\bu -\frac{\rho_F}{2}\int_0^T\int_{\Omega_M}\chi (\nabla\cdot\frac{\partial_t\et}{R+\et}r\mathbf{e}_r)\bups\cdot\bu\\
&- \rho_F \int_0^T\int_{\Omega_M}\chi(\frac{\partial_t\et}{R+\et}r\mathbf{e}_r\cdot\nabla)\bu\cdot\bups+ 2\mu_F \int_{0}^{T}\int_{\Omega_M}\chi\bD(\bu):\bD(\bups)\\
&-\rho_K h \int_0^T\int_{\omega}\bv\cdot\partial_t\bpsi R -\rho_K h\int_{\omega} \bv_0\cdot\bpsi(0) R
+\int_0^Ta_K(\be,\bpsi)\\
&-\rho_S\int_0^T\sum_{i=1}^{n_E}A_i\int_{0}^{l_i}\bk_i\cdot\partial_t\bxi_i - \rho_S \sum_{i=1}^{n_E}A_i \int_0^{l_i}\bk_{0i}\cdot\bxi_i(0)\\
&-\rho_S\int_0^T\sum_{i=1}^{n_E}\int_{0}^{l_i}M_i \bz_i\cdot\partial_t\bzeta_i- \rho_S \sum_{i=1}^{n_E} \int_0^{l_i} M_i\bz_{0i}\cdot\bzeta_i(0)
+\int_0^T a_S(\bw,\bzeta)\\
&=\int_0^T P_{in}(t) \int_{\Gamma_{in}}\upsilon_z- \int_0^T P_{out}(t)\int_{\Gamma_{out}}\upsilon_z.
\end{align*}
Using the definition of the characteristic function $\chi,$ we write the weak formulation on the physical domain $\Omega^{\be}(t):$
\begin{align*}
&-\rho_F\int_0^T\int_{\Omega^{\be}(t)} \bu\cdot\partial_t\bups 
- \frac{\rho_F}{2}\int_0^T\int_{\omega}\partial_t\et R (\bu\cdot\bups)
+\frac{\rho_F}{2}\int_0^T\int_{\Omega^{\be}(t)} b(t,\bu,\bu,\bups)\\
&+ 2\mu_F \int_{0}^{T}\int_{\Omega^{\be}(t)}\bD(\bu):\bD(\bups)
-\rho_K h \int_0^T\int_{\omega}\bv\cdot\partial_t\bpsi R 
+\int_0^Ta_K(\be,\bpsi)\\
&-\rho_S\int_0^T\sum_{i=1}^{n_E}A_i\int_{0}^{l_i}\bk_i\cdot\partial_t\bxi_i  -\rho_S\int_0^T\sum_{i=1}^{n_E}\int_{0}^{l_i}M_i \bz_i\cdot\partial_t\bzeta_i
+\int_0^T a_S(\bw,\bzeta)\\
&=\int_0^T P_{in}(t) \int_{\Gamma_{in}}\upsilon_z- \int_0^T P_{out}(t)\int_{\Gamma_{out}}\upsilon_z
+ \rho_F\int_{\Omega}\bu_0\cdot\bups(0)\\ 
&+ \rho_K h\int_{\omega} \bv_0\cdot\bpsi(0) R + \rho_S \sum_{i=1}^{n_E}A_i \int_0^{l_i}\bk_{0i}\cdot\bxi_i(0) + \rho_S \sum_{i=1}^{n_E} \int_0^{l_i} M_i\bz_{0i}\cdot\bzeta_i(0).
\end{align*}
To see that we obtained exactly the weak formulation \eqref{weak}, we have to rewrite the second term from the right-hand side, i.e. $\dis\int_0^T\int_{\omega}\partial_t\et R (\bu\cdot\bups).$
Using the fact that $\et(t,\zt,\tht)=\eta_r(t,z,\theta),$ it is easy to see that the following equality holds true:
\begin{equation*}
\partial_t\et=\partial_t\eta_r - \frac{\partial_z\eta_r}{1+\partial_z\eta_z}\cdot \partial_t\eta_z - \frac{\partial_\theta\eta_r}{1+\partial_\theta\eta_\theta}\cdot \partial_t\eta_\theta.
\end{equation*}
Additionally, the outer normal $\bn_1$ on $\Gamma^{\bet}(t)$ is equal to $\dis(-\partial_{\zt}\et,1,-\partial_{\tht}\et),$  
and after rewriting it in the "original" coordinates, we obtain that the outer normal is equal to
$\dis\left(-\frac{\partial_z\eta_r}{1+\partial_z\eta_z},1,-\frac{\partial_\theta\eta_r}{1+\partial_\theta\eta_\theta} \right).$
This yields that on $\Gamma^{\be}(t)$ we have:
\begin{equation}\label{normal}
\partial_t\et = \partial_t\be \cdot\bn_1.
\end{equation}
Finally, using \eqref{normal} and the kinematic coupling condition on $\Gamma^{\be}(t),$ we obtain:
\begin{align*}
\int_0^T\int_{\omega}\partial_t\et R (\bu\cdot\bups)&= \int_0^T\int_{\Gamma}\partial_t\et  (\bu\cdot\bups)
= \int_0^T\int_{\Gamma^{\be}(t)}(\partial_t\be\cdot \bn_1)J^{-1}  (\bu\cdot\bups)\\
&= \int_0^T\int_{\Gamma^{\be}(t)}(\partial_t\be\cdot \bn)(\bu\cdot\bups)
=\int_0^T\int_{\Gamma^{\be}(t)}(\bu\cdot\bn)(\bu\cdot\bups),
\end{align*}
where $J=\|\bn_1\|$ is the Jacobian of the transformation from $\Gamma$ to $\Gamma^{\be}(t),$ and $\bn$ is the outer unit normal on $\Gamma^{\be}(t).$
Thus, we have shown that the limiting functions satisfy the weak form \eqref{weak}.

This proves the following main result of this work:
\begin{theorem}[{\bf{Main result}}] \label{thm:exist}
	Let $\bu_0\in L^2(\Omega^{\be}(t))$, $\be_0\in H^1(\omega)$, $\bv_0\in L^2(R;\omega)$, $(\bd_0,\bw_0)\in V_S$, $(\bk_0,\bz_0)\in L^2(\mathcal{N};\R^6)$ 
	be such that
	\begin{equation*}
	\nabla\cdot\bu_0=0,\quad \bu_0|_{\Gamma^{\be_0}}\cdot \bn^{\be_0}=\bv_0\cdot \bn^{\be_0},\quad  \be_0\circ\bpi=\bd_0, 
	\end{equation*}
	and let  $P_{in/out}\in L^2_{loc}(0,\infty)$. 
	Furthermore, let  all the physical constants be positive: $\rho_K,\rho_S,\rho_F,\lambda,\mu,\mu_F>0$ and $A_i>0, \forall i = 1,\dots,n_E$. 
	
	Assume that the uniform Lipschitz property specified in Assumption~\ref{LipschitzAssumption} holds, and that the subgraph property specified in Assumption~\ref{subgraph} holds.
	Then, for every $t \le T$, where $T$ is the maximal time for which the subgraph property holds,
	 there exists a weak solution to problem \eqref{FSIeq1}-\eqref{FSIeqIC} satisfying the weak formulation
	\eqref{weak}.
\end{theorem}
\begin{remark}
	The additional regularity assumption on the approximate shell displacements is not artificial. 
	This assumption is satisfied, for example, for structures with an additional regularization term of
	sixth order, 
	which have been studied by Boulakia \cite{Boulakia2005} (this term can be physically interpreted as a tripolar material, see \cite{multipolar}).
	For such materials, the elastic operator $\mathcal{L}$ is coercive in $H^3$, and, by the Sobolev embedding, 
       the shell displacements are Lipschitz functions.
       However, even for the structures with less regularity, such as the Koiter shell studied in this manuscript, the Lipschitz property will be satisfied 
       for the appropriate data. Specifying such classes of data is an open problem. 
\end{remark}
\begin{remark}
Requiring the additional assumption on the regularity of structure displacement is intimately related to the fact that we allow shell displacement in all three spatial directions
to be different from zero.
If we had assumed that only  radial displacement is different from zero, as is the case in most FSI literature, we would not need the additional regularity assumption on
shell displacement since, in that case, the Korn's equality for the fluid space would hold true (see \cite{CDEM},\cite{BorSunMultiLayered}), 
and the convergence of the gradients would be straightforward.
\end{remark}

\section{Conclusions}

We proved the existence of a weak solution to a 3D fluid-structure interaction problem involving a composite structure,
consisting of a thin shell supported by a mesh of curved rods. The mesh supported shell serves as a lateral wall of a cylinder
filled with an incompressible, viscous fluid. The fluid flow is driven by the time-dependent inlet and outlet dynamic pressure data. 
The main challenges associated with studying this problem from the analysis point of view are the nonlinearity in the fluid equations,
the geometric nonlinearity due to the motion of the fluid domain, the mixed, parabolic-hyperbolic nature of the coupled FSI problem, 
and the inclusion of all three components of structure displacement. 
A constructive existence proof is designed based on the time discretization via Lie operator splitting, combined with an Arbitrary Lagrangian-Eulerian approach 
to deal with the motion of the fluid domain, and a compactness argument based on the generalization of the 
Aubin-Lions-Simon compactness lemma to problems on moving domains. 
To the best of our knowledge, this is the first existence result involving a composite structure, with all
three components of thin structure displacement assumed to be non-zero functions. 
Because of the constructive nature of the existence proof, the main steps in the proof can be used as a foundation for a design 
of a numerical scheme to find solutions to this class of problems. 
Further research in the direction of relaxing the subgraph assumption on the moving boundary is under way. 
Because of the presence of the thin mesh, the weak solution space framework presented in this work seems to be the 
natural (physical) framework to study solutions of this class of problems.

\bibliographystyle{amsplain} 
%\bibliography{references}

\end{document}